\documentclass{article}

\usepackage{amsfonts}
\usepackage{amssymb}
\usepackage{amsmath}
\usepackage{amsthm}
\usepackage{ifthen}
\usepackage{enumerate}
\usepackage{paralist}
\usepackage{tikz}

\newcommand{\C}{{\mathbb C}}
\newcommand{\Q}{{\mathbb Q}}
\newcommand{\R}{{\mathbb R}}
\newcommand{\Z}{{\mathbb Z}}
\newcommand{\F}{{\mathbb F}}

\newcommand{\eu}{{\mathrm e}}
\newcommand{\jac}{{\mathcal J}}
\newcommand{\hess}{{\mathcal H}}
\newcommand{\grad}{\nabla}
\newcommand{\rk}{\operatorname{rk}}

\newcommand{\GL}{\operatorname{GL}}
\newcommand{\trdeg}{\operatorname{trdeg}}
\newcommand{\tp}{^{\rm t}}
\newcommand{\Mat}{\operatorname{Mat}}
\newcommand{\p}{\mathfrak{p}}
\newcommand{\q}{\mathfrak{q}}

\newcommand{\imp}{{\mathversion{bold}$\Rightarrow$ }}
\newcommand{\bcdot}{$\discretionary{\mbox{$ \cdot $}}{}{}$}
\newcommand{\parder}[3][Default]{
	\frac{\partial \ifthenelse{\equal{#1}{Default}}{}{^{#1}}#2}{
              \partial #3 \ifthenelse{\equal{#1}{Default}}{}{^{#1}}}}

\theoremstyle{plain}
\newtheorem{theorem}{Theorem}[section]
\newtheorem{proposition}[theorem]{Proposition}
\newtheorem{lemma}[theorem]{Lemma}
\newtheorem{corollary}[theorem]{Corollary}
\theoremstyle{definition}
\newtheorem{definition}[theorem]{Definition}
\theoremstyle{remark}

\theoremstyle{plain}
\newtheorem{conjecture}[theorem]{Conjecture}

\allowdisplaybreaks
\numberwithin{equation}{section}

\makeatletter
\newcommand{\nolisttopbreak}{\vspace{\topsep}\nobreak\@afterheading}
\makeatother

\newenvironment{listproof}[1][\proofname]{\begin{proof}[#1]\mbox{}\nolisttopbreak}{\end{proof}}

\newcommand{\HE}{%
\begin{tikzpicture}[x=0.25ex,y=0.25ex]
\useasboundingbox (-11,2) rectangle (1,8);
\fill[black] (-10.15,10.15) rectangle (0.15,-0.15);
\fill[blue!10!cyan!20] (-9.85,9.85) rectangle (-0.15,0.15);
\fill[green] (-0.15,0.15) -- (-4.15,0.15) arc (150:120:10.92859) -- cycle;
\fill[brown] (-3.3,0.4) -- (-5.7,2.7) -- (-4.9,3.6) -- (-2.3,1.6);
\fill[brown] (-0.4,3.3) -- (-2.7,5.7) -- (-3.6,4.9) -- (-1.6,2.3);
\fill[green!80!black] (-5.7,2.7) arc (-30:-160:1) arc (-70:-200:1) arc (-100:-240:1) arc (240:-20:1) arc (120:-30:1) arc (115:-30:1.1);
\fill[green!80!black] (-2.7,5.7) arc (-60:70:1) arc (-20:110:1) arc (10:150:1) arc (30:290:1) arc (150:300:1) arc (155:297:1.1);
\end{tikzpicture}}

\title{Polynomial Hessians with small rank}

\author{Michiel de Bondt\footnote{The author's Ph.D. project was supported by The Netherlands
Organization for Scientific Research (NWO).}}

\begin{document}

\maketitle

\begin{abstract}
In this paper, the results in [Singular Hessians, J. Algebra 282 (2004), no.\@ 1, 195--204], for polynomial Hessians with determinant zero in small dimensions $r+1$, are generalized to similar results in arbitrary dimension, for polynomial Hessians with rank $r$. All of this is over a field $K$ of characteristic zero. The results in [Singular Hessians, J. Algebra 282 (2004), no.\@ 1, 195--204] are also reproved in a different perspective.

One of these results is the classification by Gordan and Noether of homogeneous polynomials in $5$ variables, for which the Hessians determinant is zero. This result is generalized to homogeneous polynomials in general, for which the Hessian rank is 4. Up to a linear transformation, such a polynomial is either contained in $K[x_1,x_2,x_3,x_4]$, or contained in
$$
K[x_1,x_2,p_3(x_1,x_2)x_3+p_4(x_1,x_2)x_4+\cdots+p_n(x_1,x_2)x_n]
$$
for certain $p_3,p_4,\ldots,p_n \in K[x_1,x_2]$ which are homogeneous of the same degree.

Furthermore, a new result which is similar to those in [Singular Hessians, J. Algebra 282 (2004), no.\@ 1, 195--204], is added, namely about polynomials $h \in K[x_1,x_2,x_3,x_4,x_5]$, for which the last four rows of the Hessian matrix of $t h$ are dependent. Here, $t$ is a variable, which is not one of those with respect to which the Hessian is taken. This result is generalized to arbitrary dimension as well: the Hessian rank of $t h$ is $4$ and the first row of the Hessian matrix of $t h$ is independent of the other rows.
\end{abstract}

\section{Introduction}

Let $K$ be an arbitrary field and let $L$ be an arbitrary extension field of $K$.

Write $x = (x_1,x_2,\ldots,x_n)$, $X = (x_1,x_2,\ldots,x_N)$, $y = (y_1,y_2,\ldots,y_m)$ and $Y = (y_1,y_2,\ldots,y_M)$.

Let $F \in K[x]^M$. Write $\jac F$ for the Jacobian matrix of $F$, so
$$
\jac F = \left( \begin{array}{cccc}
\parder{}{x_1} F_1 &  \parder{}{x_2} F_1 &  \cdots & \parder{}{x_n} F_1 \\
\parder{}{x_1} F_2 &  \parder{}{x_2} F_2 &  \cdots & \parder{}{x_n} F_2 \\
\vdots & \vdots & & \vdots \\
\parder{}{x_1} F_M &  \parder{}{x_2} F_M &  \cdots & \parder{}{x_n} F_M
\end{array} \right)
$$
Let $f \in K[x]$. To define $\jac f$, we associate $K[x]$ with $K[x]^1$. Write $\grad f$ for the gradient map of $f$, i.e.\@ the polynomial map of partial derivatives of $f$, which we see as a column vector of polynomials. So
$$
\grad f = (\jac f)\tp = \left( \begin{array}{c}
\parder{}{x_1} f \\ \parder{}{x_2} f \\ \vdots \\ \parder{}{x_n} f
\end{array} \right)
$$
Write $\hess f$ for the Hessian matrix of $f$, so
$$
\hess f = \jac(\grad f) = \left( \begin{array}{cccc}
\parder{}{x_1} \parder{}{x_1} f &  \parder{}{x_2} \parder{}{x_1} f &  \cdots & \parder{}{x_n} \parder{}{x_1} f \\
\parder{}{x_1} \parder{}{x_2} f &  \parder{}{x_2} \parder{}{x_2} f &  \cdots & \parder{}{x_n} \parder{}{x_2} f \\
\vdots & \vdots & \ddots & \vdots \\
\parder{}{x_1} \parder{}{x_n} f &  \parder{}{x_2} \parder{}{x_n} f &  \cdots & \parder{}{x_n} \parder{}{x_n} f
\end{array} \right)
$$
Let $B \in \Mat_{m,M}(K)$ and $C \in \Mat_{N,n}(K)$. 
If $H \in K[X]^M$ and $h \in K[X]$, then
\begin{align}
\jac \big(BH(Cx)\big) &= B \cdot (\jac H)|_{X=Cx} \cdot C \\
\grad \big(h(Cx)\big) &= C\tp \cdot (\grad h)|_{X=Cx} \\
\hess \big(h(Cx)\big) &= \jac \big(\grad(h(Cx))\big) = C\tp \cdot (\hess h)|_{X=Cx} \cdot C
\end{align}
Assume from now on in this section that $K$ has characteristic zero. Let $h \in K[x]$ such that $\det \hess h = 0$. 
In 
\cite{MR2095579}, $h$ is classified in the following cases
\begin{itemize}

\item $n \le 3$,

\item $n = 4$ and the last $3$ rows of $\hess h$ are dependent.

\item $n \le 5$ and $h$ is homogeneous.

\end{itemize}
To obtain these classification results, some other results are used. To explain these results, we first observe that $\det \hess h = 0$ is equivalent to the existence of a nonzero $f \in K[y_1,y_2,\ldots,y_n]$, such that $f(\grad h) = 0$, see e.g.\@ proposition 1.2.9 of either \cite{MR1790619} or \cite{homokema}.

Now take any $T \in \GL_n(K)$. Then
$$
(f \circ T) \big(\grad\big(h((T^{-1})\tp x)\big)\big) = 
f \big(T T^{-1} (\grad h) \big((T^{-1})\tp x\big)\big) =
0|_{x=(T^{-1})\tp x} = 0
$$
So if we transform $f$ linearly, then we can transform
$h$ linearly in a corresponding way. 

The results that are used to obtain the above-mentioned classification results are (i) and (iii) of the following.
\begin{enumerate}[\upshape (i)]

\item If $2 \le n \le 3$, then we can choose $f$ and $T$ such that $(f \circ T) \in K[y_2,\ldots,y_n]$.

\item If $2 \le n \le 4$ and $f$ can be taken such that $f(y_1+a_1,y_2+a_2,\ldots,y_n+a_n)$ is homogeneous for some $a \in K^n$, then we can choose $f$ and $T$ such that $f(y_1+a_1,y_2+a_2,\ldots,y_n+a_n)$ is homogeneous and $(f \circ T) \in K[y_2,\ldots,y_n]$.

\item If $3 \le n \le 5$ and $h$ is homogeneous, then we can choose $f$ and $T$ such that $(f \circ T) \in K[y_3,\ldots,y_n]$.

\end{enumerate}
To obtain these results, so-called quasi-translations are studied. For more information about quasi-translations, see \cite{MR3794325} and \cite{MR3804055} and the references therein. The quasi-translation results are proved with geometric techniques for the case where $K = \C$. The reduction to the general case is not very hard, but \cite{MR2095579} is vague on this point. This is fixed in this paper.

\subsection{Apices at infinity}

The study of quasi-translations yields more than just results (i), (ii) and (iii) above. These additional results could have been used further in the process of obtaining the classifications of $h$ in \cite{MR2095579}. But we did not make use of quasi-translations any further in \cite{MR2095579}.

In this paper, we will not make use of quasi-translations beyond (i), (ii) and (iii) above either. This is because we were not able to extend the theory of quasi-translations for the above results from rank $r$ in dimension $r + 1$ to rank $r$ in general. 

Since we generalize from rank $r$ in dimension $r + 1$ to rank $r$ in general with techniques that do not require quasi-translations, we only need the results (i), (ii) and (iii) for the case where $\rk \hess h = n - 1$. The following are reformulations of (i), (ii) and (iii) combined with the condition that $\rk \hess h = n - 1$.
\begin{enumerate}[\upshape (i$'$)]

\item If $2 \le n \le 3$ and $\rk \hess h = n-1$, then the image of $\grad h$ has an apex at infinity. 

\item If $2 \le n \le 4$, $\rk \hess h = n-1$ and the image of $\grad h$ has an apex, then the image of $\grad h$ has an apex at infinity.

\item If $3 \le n \le 5$, $\rk \hess h = n-1$ and $h$ is homogeneous, then the image of $\grad h$ has two apices at infinity.

\end{enumerate}
We define apex (at infinity) in section \ref{sec:apex}. Some authors use the term vertex instead of apex. Since the term vertex is already used in graph theory, we follow Robin Hartshorne here.

The first part of the generalization from rank $r$ in dimension $r + 1$ to rank $r$ in general is doing this for (i$'$), (ii$'$) and (iii$'$), i.e.\@ obtaining the following.
\begin{enumerate}[\upshape (i$''$)]

\item If $1 \le \rk \hess h \le 2$, then the image of $\grad h$ has an apex at infinity. 

\item If $1 \le \rk \hess h \le 3$ and the image of $\grad h$ has an apex, then the image of $\grad h$ has an apex at infinity.

\item If $2 \le \rk \hess h \le 4$ and $h$ is homogeneous, then the image of $\grad h$ has two apices at infinity.

\end{enumerate}
If we take $n$ one higher than its upper bound in (i$'$) and (ii$'$), then we get claims for which we do not know if they are true. These problems are equivalent to the claims in Problem 1 and Problem 2 in the last section of \cite{MR3794325}, except that Problem 2 is misformulated: $R$ is allowed to be adapted to meet the formula (otherwise Problem 2 has a negative answer). 

But we do know that the second claim implies the first claim. Namely, if $h$ is a counterexample to (i$'$) with $n = 4$, then (ii$'$) tells us that the image of $\grad h$ does not have an apex. From that, one can infer that $(h + x_5)^2$ is a counterexample to (ii$'$) with $n = 5$.

So we did not prove (i$'$) with $n=4$ and neither (ii$'$) with $n = 5$. But we did derive their generalizations from them. More precisely, we proved the following.
\begin{enumerate}[\upshape(i$''$)]
\setcounter{enumi}{3}

\item Suppose that $\rk \hess h = 3$ (and the image of $\grad h$ does not have an apex). If the image of $\grad h$ does not need to have an apex at infinity for general $n$, then it does not need that already for $n = 4$. 

\item Suppose that $\rk \hess h = 4$ and the image of $\grad h$ does have an apex. If the image of $\grad h$ does not need to have an apex at infinity for general $n$, then it does not need that already for $n = 5$. 

\end{enumerate}
If the image of $\grad h$ does have an apex in (iv$''$), then the image of $\grad h$ will have an apex at infinity, because of (ii$''$). So one may assume that the image of $\grad h$ does not have an apex in (iv$''$).

So we were able to generalize the existence of apices at infinity to arbitrary dimensions for the situations where we needed it. The proofs of these results start in section \ref{sec:reddimrplus1}, where (i$''$) and (ii$''$) are reduced to (i$'$) and (ii$'$) respectively. The claims of lemma \ref{slem} yield (i$''$), (ii$''$) and (iii$''$). The proofs are finished in section \ref{iiiniv}, where (iv$''$), and (v$''$) are proved. 

The proofs are quite technical, especially the proof of (v$''$). Since (iv$''$) and (v$''$) are not real results, but only reductions of problems to the lowest dimension, the reader can choose to ignore these results, by skipping section \ref{iiiniv}.
 
So we obtained the reductions that we wanted. But it would be more desirable to get a general result saying that one can always reduce to dimension $r+1$. The following would suffice to obtain such a result.

\begin{conjecture}
Let $H \in K[x]^m$, such that $H$ does not have an apex at infinity and $\rk \jac H = m - 2$. Then there exists a $p \in K^m$, such that $(p_1:p_2:\cdots:p_m)$ is the only apex at infinity of $H + x_{n+1}p$.
\end{conjecture}

\subsection{Extension of scalars}

In \cite{MR2095579}, the proof of the Gordan-Noether classification comes down to the following if $\rk \hess h = n - 1 = 4$. Let $s$ be the number of independent apices at infinity of the image of $\grad h$. Then $s \le \rk \hess h = 4$. On account of (iii$'$) above, $s \ge 2$, so $2 \le s \le 4$. 

Since $\grad h$ is homogeneous of degree $\deg h - 1 > 0$, one infers that $(0,0,0,0,0)$ is an apex of the image of $\grad h$. It can be shown that due to this, the case $s = n - 2 = 3$ does not occur. So either $s = 2$ or $s = 4$. If $s = 4$, then $h$ is degenerate. So suppose that $s = 2$. By way of a linear transformation, the apices at infinity become $(1:0:0:0:0)$ and $(0:1:0:0:0)$, and linear combinations thereof.

Then there exist a nonzero $f \in K[y_3,y_4,y_5]$, such that $f(\grad h) = 0$. So 
$$
\rk \jac\Big(\parder{}{x_3} h, \parder{}{x_4} h, \parder{}{x_5} h\Big) = 2
$$
The image of the map $(\parder{}{x_3} h, \parder{}{x_4} h, \parder{}{x_5} h)$ does not have an apex $(p_3 : p_4 : p_5)$ at infinity, because otherwise $(0 : 0 : p_3 : p_4 : p_5)$ would be an apex at infinity of the image of $\grad h$. 

But we can view $h$ as a polynomial over $K(x_1,x_2)$, with
variables $x_3,x_4,x_5$, too. Then $(\parder{}{x_3} h, \parder{}{x_4} h, \parder{}{x_5} h)$ is the gradient map of 
$h$. Furthermore,
\begin{equation} \label{extscaleq}
\rk \hess_{x_3,x_4,x_5} h \le \rk \jac\Big(\parder{}{x_3} h, \parder{}{x_4} h, \parder{}{x_5} h\Big) = 2
\end{equation}

If $\rk \hess_{x_3,x_4,x_5} h = 2$, then it follows from (i$'$) above that the image of $(\parder{}{x_3} h, \parder{}{x_4} h, \parder{}{x_5} h)$ does have an apex at infinity as a map over $K(x_1,x_2)$. This base field-relative existence of an apex at infinity requires \eqref{extscaleq} to be strict, i.e.\@
\begin{equation} \label{extscaleqs}
\rk \hess_{x_3,x_4,x_5} h \le \rk \jac\Big(\parder{}{x_3} h, \parder{}{x_4} h, \parder{}{x_5} h\Big) - 1 = 1
\end{equation}

Suppose first that $\rk \hess_{x_3,x_4,x_5} h = 1$. From the fact that $(\parder{}{x_3} h, \parder{}{x_4} h, \parder{}{x_5} h)$ is homogeneous of degree $\deg h - 1 > 0$, one infers that $(0,0,0)$ is an apex over $K$ of the image of $(\parder{}{x_3} h, \parder{}{x_4} h, \parder{}{x_5} h)$. If $(0,0,0)$ is an apex over $K(x_1,x_2)$ of $(\parder{}{x_3} h, \parder{}{x_4} h, \parder{}{x_5} h)$ as well, then $h$ is a polynomial in a linear form over $K(x_1,x_2)$, i.e.\@
\begin{equation} \label{extscalrk1}
h \in K(x_1,x_2)[p_3(x_1,x_2) x_3 + p_4(x_1,x_2) x_4 + p_5(x_1,x_2) x_5]
\end{equation}
If $(0,0,0)$ is \emph{not} an apex over $K(x_1,x_2)$ of $(\parder{}{x_3} h, \parder{}{x_4} h, \parder{}{x_5} h)$, then 
the property that $(0,0,0)$ is an apex is base field-relative as well. This requires \eqref{extscaleqs} to be strict as well, i.e.\@
\begin{equation} \label{extscaleqss}
\rk \hess_{x_3,x_4,x_5} h \le \rk \jac\Big(\parder{}{x_3} h, \parder{}{x_4} h, \parder{}{x_5} h\Big) - 2 = 0
\end{equation}
This contradicts $\rk \hess_{x_3,x_4,x_5} h = 1$. 

Suppose next that $\rk \hess_{x_3,x_4,x_5} h = 0$. Then $h$ is of the form
\begin{equation} \label{extscalrk0}
h = b_3(x_1,x_2) x_3 + b_4(x_1,x_2) x_4 + b_5(x_1,x_2) x_5
\end{equation}
which is also of the form of \eqref{extscalrk1}. So $h$ is of the form \eqref{extscalrk1}.

\eqref{extscaleqss} is only proved for the situation at hand in \cite{MR2095579}, namely in the proof of \cite[Th.\@ 3.6]{MR2095579}. But \eqref{extscaleqs} appears in a general form already in \cite{MR2095579}, namely as \cite[Prop.\@ 2.3]{MR2095579}. But this results still restrict to the case where the ideal of relations over $K$ between $(\parder{}{x_3} h, \parder{}{x_4} h, \parder{}{x_5} h)$ is principal. 

In theorem \ref{xsc} (i), this restriction is removed. A reduction techique is used to do this: see the proof of theorem \ref{piaLK}. The reduction technique utilizes theorem \ref{cominher}. Theorem \ref{xsc} (iii) generalizes \eqref{extscaleqss} in a similar manner as theorem \ref{xsc} (i) generalizes \eqref{extscaleqs}

For other classifications in \cite{MR2095579}, \cite[Prop.\@ 2.3]{MR2095579} is used as well to reduce to smaller dimensions. In this paper, we reduce to smaller dimensions in similar ways, using theorem \ref{xsc}, to finally obtain corollary \ref{gradrelcor} as a classification result.

\subsection{Integral domains, extra condition, and homogeneity}

In \eqref{extscalrk0} and \eqref{extscalrk1}, $h$ is a polynomial over $R = K[x_1,x_2]$, in variables $x_3,x_4,x_5$, with Hessian rank $0$ and $1$ respectively. Plain induction only yields that $h$ is (a polynomial in) a linear form over $L = K(x_1,x_2)$. 

But a more desirable conclusion is that $h$ is (a polynomial over $R$ in) a linear form over $R$. For \eqref{extscalrk0}, this is trivial. But for \eqref{extscalrk1}, some work needs to be done. The conclusion that $h$ is a polynomial over $R$ in a linear form over $R$ can be obtained from the fact that $R$ is a gcd-domain.

Generalizations of these results are obtained in theorem \ref{rtheorem}. In the case where $h$ has Hessian rank $0$, $R$ may be any integral domain. In the case where $h$ has Hessian rank $1$, $R$ must be a $\gcd$-domain. There is also a case where $h$ has Hessian rank $2$. In that case, we assume that the image of $\grad h$ has an apex over $L$, and that $R$ is a B{\'e}zout domain. The proof of theorem \ref{rtheorem} covers the largest part of section \ref{sec:rtheorem}.  

Theorem \ref{rtheorem} can be used to obtain some formulas for polynomials with small Hessian rank, but these formulas are not always closed formulas. To indicate this, assume from now on that $h$ is not homogeneous, but that $(0,0,0,0,0)$ is an apex of the image of $\grad h$. Then we obtain \eqref{extscalrk0} and \eqref{extscalrk1} again for the case where $s = 2$. 

\eqref{extscalrk0} is still a closed formula, because 
if $h$ is of this form, then $(\parder{}{x_3} h, \parder{}{x_4} h, \allowbreak\parder{}{x_5} h) = (b_3,b_4,b_5)$ are automatically algebraically dependent over $K$. But \eqref{extscalrk1} is no longer a closed formula. Take for instance $h = (x_1 x_3 + x_1 x_2 x_4 + x_2 x_5)^2$. Then $(\parder{}{x_3} h, \parder{}{x_4} h, \parder{}{x_5} h)$ are not algebraically dependent over $K$, and $\det \hess h \ne 0$.

There is an extra condition required for \eqref{extscalrk1}, namely that $tp_3, tp_4, tp_5$ are algebraically dependent over $K$. With one variable less, \eqref{extscalrk0} loses its last term, and the condition that $b_3$ and $b_4$ are algebraically dependent over $K$ is required. This condition appears in \cite[Th.\@ 3.5]{MR2095579}.

Besides the above-described generalizations from small dimensions to small rank, these are the only cases where extra conditions are needed in corollary \ref{gradrelcor} (the above-mentioned classification result). To be able to talk about algebraic dependence over $K$ in such conditions, $R$ must be a $K$-algebra. Therefore, we assume in our main theorem, which is theorem \ref{gradrel}, that $R$ is as such. In addition, we assume that $K$ is algebraically closed in the fraction field of $R$. 

But theorem \ref{gradrel} does not restrict to the situation of corollary \ref{gradrelcor}. No, $r = \rk \hess h$ may be taken larger, as long as $s$ increases likewise, so $r - s$ remains the same. This is possible because the induction by way of extension of scalars still works. But due to putting conditions on $r - s$ instead of $r$, the formulas apply to more situations. Consequently, all formulas in theorem \ref{gradrel}, except the first formula, require extra conditions.

There is however another problem, namely that a polynomial ring in more than one variable is not a B{\'e}zout domain, so theorem \ref{rtheorem} cannot be applied with Hessian rank $2$. This problem does not occur when we restrict to the situation of corollary \ref{gradrelcor}, because $s = 1$ and $R = K[x_{m+1}]$ in the case where we need theorem \ref{rtheorem} with Hessian rank $2$. 

For that reason, we provide alternative conditions on $R$ to cover the case of a polynomial ring. These conditions are that $R$ is a $\gcd$-domain, and that its fraction field satisfies L{\"u}roth theorem as an extension field of $K$. Section \ref{sec:gradrel} is mainly theorem \ref{gradrel} and its proof.

Finally, we prove that we can take $p$ and $q$ homogeneous, as far as they are present, if $h$ is homogeneous. This was proved in the proof of \cite[Th.\@ 3.6]{MR2095579} already, but only for the case where $p$ but not $q$ is present. With both $p$ and $q$ present, things are harder to prove. The homogeneity result consists of the last two claims of theorem \ref{gradrel}, in which it is assumed that $R[x_1,x_2,\ldots,x_m]$ is graded in a rather arbitrary manner. Something like that is necessary, because the variables $x_{m+1}, x_{m+2}, \ldots, x_n$ in corollary \ref{gradrelcor} are scalar variables in theorem \ref{gradrel}.

\section{\protect\HE-pairing}

Inspired by the concept of `GZ-pairing', which is defined as `pairing' in \cite{MR1621913}, we define our own pairing, namely Hessen-pairing. The icon of our pairing is inspired by the symmetry of the Hessian matrix, and the trees on the icon represent `essen', which is Dutch for `ash trees'.

\begin{definition}
Assume that $1 \le n \le N$ and $1 \le m \le M$. Let $h \in L[x]^m$ and $H \in L[X]^M$.
We say that $h$ and $H$ are \emph{\HE-paired over $K$} through matrices 
$B \in \Mat_{m,M}(L)$ and $C \in \Mat_{N,n}(L)$, if
\begin{enumerate}[\upshape (i)]
 
\item $h = B(H(Cx))$,

\item $\trdeg_K K(h) = \trdeg_K K(H)$.

\end{enumerate}
If additionally 
\begin{enumerate}[\upshape (i)]
\setcounter{enumi}{2} 

\item $B = C\tp$

\end{enumerate}
then we say that $h$ and $H$ are \emph{symmetrically \HE-paired over $K$} 
through $B$ and $C$. We say that $h$ and $H$ are (symmetrically) \HE-paired 
over $K$ if there exists matrices $B$ and $C$ as indicated above.
\end{definition}

\begin{proposition} \label{HEphi}
Suppose that $h \in L[x]^m$ and $H \in L[X]^M$ are \HE-paired over $K$ 
through matrices $B \in \Mat_{m,M}(L)$ and $C \in \Mat_{N,n}(L)$. Then 
\begin{equation} \label{HEphieq}
\trdeg_K K(h) = \trdeg_K K(BH) = \trdeg_K K(H)
\end{equation}
Furthermore, $f(h) = 0 \Longleftrightarrow f(BH) = 0$ for every $f \in K[y]$. 
\end{proposition}

\begin{proof}
Since $\trdeg_K K(h) \le \trdeg_K K(BH) \le \trdeg_K K(H)$, we see that 
\eqref{HEphieq} follows from property (ii) of \HE-pairing. 

From property (i) of \HE-pairing, it follows that $f(BH) = 0 \Longrightarrow f(h) = 0$ for all $f \in K[y]$. So the last claim follows from $\trdeg_K K(h) = \trdeg_K K(BH)$.
\end{proof}

\begin{lemma} \label{HEtrdeg}
Let $H \in K[X]^s$, such that $H_1,H_2,\ldots,H_s,x_{s+1},x_{s+2},\ldots,x_N$ 
is a transcendence basis over $K$ of $K(X)$. 
Then there exists a nonzero polynomial $g \in K[x]$ such that 
$$
g(H|_{x_N=cx_{N-1}},x_{s+1},x_{s+2},\ldots,x_{N-1}, c x_{N-1}) \ne 0
$$ 
implies that 
$$
H_1|_{x_N=cx_{N-1}},H_2|_{x_N=cx_{N-1}},\ldots,H_s|_{x_N=cx_{N-1}},
x_{s+1},x_{s+2},\ldots,x_{N-1}
$$ 
is a transcendence basis over $K$ of $K(x_1,x_2,\ldots,x_{N-1})$.
\end{lemma}

\begin{proof}
Take $i \le s$. Then $x_i$ is algebraically dependent over $K$ on 
$H_1,H_2,\ldots,H_s,\allowbreak x_{s+1}, \ldots, x_N$, say that $f \in K[X,t]$ such
that $f_i(H,x_{s+1},x_{s+2},\ldots,x_N,x_i) = 0$ and $f_i(H,x_{s+1},x_{s+2},\ldots,x_N,t) \ne 0$. 
Take $g_i \in K[X]$ such that $g_i(H,x_{s+1},\allowbreak x_{s+2},\allowbreak \ldots,x_N)$ is any nonzero
coefficient of $f_i(H,x_{s+1},x_{s+2},\ldots,x_N,t)$ as a polynomial in $t$ over 
$K[x]$. 

Suppose that $g_i(H|_{x_N=cx_{N-1}},x_{s+1},x_{s+2},\ldots,x_{N-1},c x_{N-1}) \ne 0$. Then 
$$
f_i(H|_{x_N = c x_{N-1}},x_{s+1},x_{s+2},\ldots,x_{N-1},c x_{N-1},x_i) = 0
$$
and 
$$
f_i(H|_{x_N = c x_{N-1}},x_{s+1},x_{s+2},\ldots,x_{N-1},c x_{N-1},t) \ne 0
$$
So $x_i$ is algebraically dependent over $K$ on 
$$
{\cal B} := H_1|_{x_N = c x_{N-1}}, H_2|_{x_N = c x_{N-1}}, \ldots, H_s|_{x_N = c x_{N-1}},
x_{s+1},x_{s+2},\ldots,x_{N-1}
$$

Put $g := g_1 g_2 \cdots g_s$ and suppose that $g(H|_{x_N=cx_{N-1}},x_{s+1},x_{s+2},\ldots,x_{N-1}, 
\allowbreak c x_{N-1}) \ne 0$. From the above paragraph, we deduce that $x_i$ is 
algebraically dependent over $K$ on ${\cal B}$ for all $i \le s$. Hence ${\cal B}$ 
is a transcendence basis over $K$ of $K(x_1,x_2,\ldots,x_{N-1})$.
\end{proof}

\begin{theorem} \label{HEpairinga}
Assume that $K$ is infinite, and let $H \in K[X]^m$. 
Suppose that $n \le N$, such that $x_n, x_{n+1}, \ldots, x_N$ are algebraically 
independent over $K(H)$. Then there exists a $C \in \Mat_{N,n}(K)$ of the form
\begin{equation} \label{Cform}
C\tp = \left( \begin{array}{cccccc}
& & & 0 & \cdots & 0 \\
& I_{n-1} & & \vdots & \cdots & \vdots \\
& & & 0 & \cdots & 0 \\
0 & \cdots & 0 & c_n & \cdots & c_N
\end{array} \right)
\qquad \mbox{where $c_n \cdots c_N \ne 0$}
\end{equation}
such that $H(Cx)$ and $H$ are \HE-paired.
\end{theorem}

\begin{proof}
Let $s := \trdeg_K K(H)$ and assume without loss of generality that $x_{s+1}, \allowbreak
x_{s+2}, \ldots, x_N$ is a transcendence basis over $K(H)$ of $K(X)$. Assume without loss of 
generality that $H_1, H_2, \ldots, H_s, x_{s+1}, x_{s+2}, \ldots, x_N$ is a transcendence 
basis over $K$ of $K(X)$. 

We may assume that $n = N-1$, because the general case follows by induction. But in order to
have the properties of $H$ of this theorem in the induced situation, we must show that 
$H_1(Cx),H_2(Cx),\ldots,H_s(Cx),x_{s+1},x_{s+2},\ldots,x_n$ is a transcendence basis over $K$ 
of $K(x)$.

Take $g$ as in lemma \ref{HEtrdeg}. There are only finitely many
$c \in K$ such that $g(H_1|_{x_N=cx_{N-1}},\ldots,H_s|_{x_N=cx_{N-1}}, x_{s+1},\ldots,x_{N-1}, c x_{N-1}) = 0$.
Since $K$ is infinite, there exists a nonzero $c \in K$ such that 
$g(H_1|_{x_N=cx_{N-1}},\ldots,H_s|_{x_N=cx_{N-1}},\allowbreak x_{s+1},\ldots,x_{N-1},c x_{N-1}) \ne 0$. 
From lemma \ref{HEtrdeg}, it follows that $H_1|_{x_N = c x_{N-1}}, \allowbreak H_2|_{x_N = c x_{N-1}}, 
\ldots, H_s|_{x_N = c x_{N-1}}, \allowbreak x_{s+1}, \ldots,x_{N-1}$ is a transcendence basis over $K$ of $K(x_1,x_2,\ldots,x_{N-1})$. So the induction is effective and 
$\trdeg_K K(H_1|_{x_N = c x_{N-1}}, \allowbreak \ldots, H_s|_{x_N = c x_{N-1}}) = s = \trdeg_K K(H)$. Therefore $\trdeg_K K(H|_{x_N = c x_{N-1}}) = \trdeg_K K(H)$.
\end{proof}

\begin{corollary} \label{HEpairingc}
Assume that $K$ is infinite, and let $H \in K[X]^N$. 
Suppose that $n \le N$, such that $x_n, x_{n+1}, \ldots, x_N$ are algebraically 
independent over $K(H)$. Suppose additionally that $H_i$ is algebraically dependent over $K$
on $H_1, H_2, \ldots, \allowbreak H_{n-1}$ for all $i$. Then there exists a 
$C \in \Mat_{N,n}(K)$ of the form of \eqref{Cform}, such that $C\tp H(Cx)$ and $H$ are 
symmetrically \HE-paired.
\end{corollary}

\begin{proof}
Let $h = (H_1, H_2, \ldots, H_{n-1})$. From theorem \ref{HEpairinga}, it follows
that there exists a $C \in \Mat_{N,n}(K)$ of the form of \eqref{Cform}, such that
$h(Cx)$ and $h$ are \HE-paired. Furthermore, the first $n-1$ components of $C\tp H(Cx)$
are $h(Cx)$, so $\trdeg_K K(h) = \trdeg_K K\big(h(Cx)\big) \le \trdeg_K K\big(C\tp H(Cx)\big)
\le \trdeg_K K(H)$. Since $\trdeg_K K(h) = \trdeg_K K(H)$, we infer that $C\tp H(Cx)$ and 
$H$ are \HE-paired.
\end{proof}

\begin{theorem} \label{HEpairings}
Assume that $K$ has characteristic zero, and let $h \in K[X]$. Suppose that $n \le N$, 
such that the first $n-1$ rows of $\hess h$ generate the row space of $\hess h$
over $K(X)$. Then there exists a $C \in \Mat_{N,n}(K)$ of the form of \eqref{Cform}, 
such that $C\tp (\grad h)(Cx)$ and $\grad h$ are symmetrically \HE-paired.
\end{theorem}

\begin{proof}
From proposition 1.2.9 of either \cite{MR1790619} or \cite{homokema}, it follows that 
$H_i$ is algebraically dependent over $K$ on $H_1, H_2, \ldots, H_{n-1}$ for all $i$.
Since $\hess h$ is symmetric, the first $n-1$ columns of $\hess h$ generate the column 
space of $\hess h$. Hence $e_n\tp, e_{n+1}\tp,\ldots,e_N\tp$ are independent over $K(x)$
of the rows of $\hess h$. From proposition 1.2.9 of either \cite{MR1790619} or 
\cite{homokema}, it follows that $x_n, x_{n+1}, \ldots, x_N$ are algebraically 
independent over $K(H)$. So the result follows from corollary \ref{HEpairingc}.
\end{proof}

\section{(Projective) image apices} \label{sec:apex}

Let $Z$ be an algebraic subset of $L^m$. 
We say that $a \in L^m$ is an \emph{apex} of $Z$ 
if $(1-\lambda)c + \lambda a \in Z$ for all $\lambda \in L$ and all $c \in Z$. 
If we think of $a$ as a point of the projective horizon, i.e.\@ $\|a\| = \infty$, 
then we can denote the direction of $a$ by $p$ to get something that is properly defined.

We say that a $p \in L^m$ is a \emph{projective apex} of $Z$ if $p \ne 0$ and
$c + \lambda p \in Z$ for all $\lambda \in L$ and all $c \in Z$. We say that $Z$ has an 
\emph{apex at infinity} if $Z$ has a projective apex. Indeed, if $\|a\| = \infty$ and $p$ 
is the direction of $a$, then the meaning of that $p$ is a projective apex of $Z$ is 
just essentially that the point $a$ at infinity is an apex of $Z$.

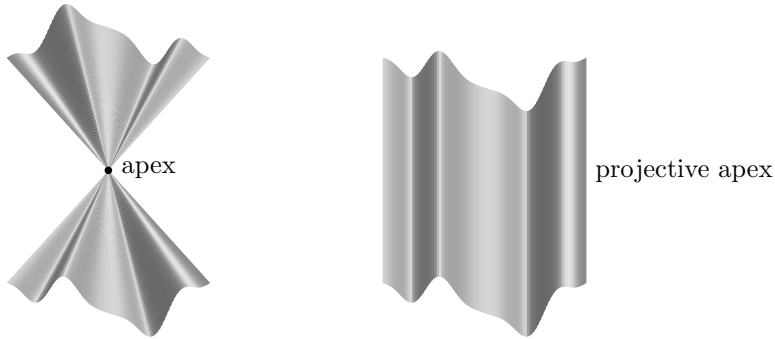
\begin{center}
\begin{tikzpicture}
\foreach \ang in {0.5,1.5,...,90} {
  \pgfmathsetmacro{\yl}{-2*sin(\ang*2-1)+sin(\ang*4-2)-0.7*sin(\ang*8-4)-
                        0.5*sin(\ang*12-6)+0.3*sin(\ang*16-8)+0.2*sin(\ang*20-10)}
  \pgfmathsetmacro{\yr}{-2*sin(\ang*2+1)+sin(\ang*4+2)-0.7*sin(\ang*8+4)-
                        0.5*sin(\ang*12+6)+0.3*sin(\ang*16+8)+0.2*sin(\ang*20+10)}
  \pgfmathsetmacro{\x}{-cos(\ang*2)+cos(\ang*4)-1.4*cos(\ang*8)-
                       1.5*cos(\ang*12)+1.2*cos(\ang*16)+cos(\ang*20)}
  \pgfmathsetmacro{\c}{10+0.5*abs((atan(\x)+20))};
  \fill[black!\c] (0.03*\ang-0.02,0.2*\yl) -- (0.03*\ang+0.02,0.2*\yr) -- 
                  (2.7-0.03*\ang-0.02,3-0.2*\yr) -- (2.7-0.03*\ang+0.02,3-0.2*\yl) -- cycle;
  \fill[black!\c] (0.03*\ang-0.02+5,0.2*\yl) -- (0.03*\ang+0.02+5,0.2*\yr) -- 
                  (0.03*\ang+0.02+5,3+0.2*\yr) -- (0.03*\ang-0.02+5,3+0.2*\yl) -- cycle;
}
\fill (1.35,1.5) circle (0.5mm);
\node[anchor=west] at (1.4,1.5) {apex};
\node[anchor=west] at (7.7,1.5) {projective apex};
\end{tikzpicture}
\end{center}

If $Z$ is the Zariski closure of the image of a polynomial map $H$ over an infinite field $L$, 
then we say that $a$ and $p$ as above are an \emph{image apex} of $H$ and a 
\emph{projective image apex} of $H$ respectively. For the case where $L$ is not necessarily 
infinite, we use the definition below.

Suppose that $H \in L[x]^m$. We say that $a \in L^m$ is an 
{\em image apex of $H$ (over $K$)} if
$$
f(H) = 0 \Longrightarrow f\big((1-t) H + ta\big) = 0
$$
for all $f \in L[y]$ ($f \in K[y]$). We say that $p \in L^m$ is a 
\emph{projective image apex of $H$ (over $K$)} if $p \ne 0$ and
$$
f(H) = 0 \Longrightarrow f\big(H + tp\big) = 0
$$
for all $f \in L[y]$ ($f \in K[y]$). We say that $H$ has an 
\emph{image apex at infinity (over $K$)} if $H$ has a projective 
image apex (over $K$). 

\begin{proposition} \label{apexmap}
Let $H$ be a polynomial map to $L^m$, $B \in \Mat_{m,M}(K)$, and $c \in K^m$.
\begin{enumerate}[\upshape (i)]
 
\item If $a \in L^m$ is an image apex of $H$ over $K$, then $B a + c$ is an image apex of $B H + c$ over $K$.

\item If $p \in L^m$ is a projective image apex of $H$ over $K$ and $B p \ne 0$, then $B p$ is a projective image apex of $B H + c$ over $K$.

\end{enumerate}
\end{proposition}

\begin{proof} Take $f \in K[y]$ arbitrary.
\begin{enumerate}[\upshape (i)]
 
\item Assume that $a \in L^m$ is an image apex of $H$ over $K$. 
Suppose that $f(BH + c) = 0$ for some $f \in K[y]$. Since $a$ is an image 
apex of $H$ over $K$, $f(B((1-t)H + t a) + c) = 0$. 
By linearity of $B$, $f((1-t) B H + t B a + c) = 0$. 
Hence $f((1-t) (B H + c) + t (B a + c)) = 0$ indeed.

\item Assume that $p \in L^m$ is a projective image apex of $H$ over $K$ and $B p \ne 0$. Suppose that $f(BH + c) = 0$ for some $f \in K[y]$. Since $p$ is a projective image apex of $H$ over $K$, $f(B(H + t p) + c) = 0$. By linearity of $B$, $f(B H + t B p + c) = 0$. Hence $f((B H + c) + t B p) = 0$ indeed. \qedhere

\end{enumerate}
\end{proof}

\begin{proposition} \label{trdegprop}
Let $H$ be a polynomial map to $L^m$.
\begin{enumerate}[\upshape (i)]

\item $a \in L^m$ is an image apex of $H$ over $K$, if and only if 
$$
\trdeg_K K\big((1-t)H+ta\big) = \trdeg_K K(H)
$$

\item $p \in L^m$ is a projective image apex of $H$ over $K$, if and only if 
$$
\trdeg_K K\big(H+tp\big) = \trdeg_K K(H)
$$ 

\end{enumerate}
\end{proposition}

\begin{proof}
We only prove (i), since (ii) is similar. Since $f(H) = 0 \Longleftarrow f\big((1-t)H+ta\big) = 0$ follows by substituting $t = 0$, we deduce that $\trdeg_K K\big((1-t)H+ta\big) = \trdeg_K K(H)$, if and only if $f(H) = 0 \Longrightarrow f\big((1-t)H+ta\big) = 0$ for all $f \in k[y]$, i.e.\@ $a$ is an image apex of $H$ over $K$.
\end{proof}

\begin{proposition} \label{algdepprop}
Let $H$ be a polynomial map to $L^m$
\begin{enumerate}[\upshape (i)]

\item Let $a \in L^m$. If the algebraic closure of $K\big((1-t)H+ta\big)$ in $L(x,t)$ contains a nonzero element whose numerator is divisible by $t$, then $a$ is \emph{not} an image apex of $H$ over $K$.

Conversely, if $a$ is \emph{not} an image apex of $H$ over $K$, then $K\big[(1-t)H+ta\big]$ contains a nonzero element which is divisible by $t$.

\item Let $p \in L^m$. If the algebraic closure of $K\big(H+tp\big)$ contains a nonzero element whose numerator is divisible by $t$, then $p$ is \emph{not} a projective image apex of $H$ over $K$.

Conversely, if $p$ is \emph{not} a projective image apex of $H$ over $K$, then $K\big[H+tp\big]$ contains a nonzero element which is divisible by $t$.

\end{enumerate}
\end{proposition}

\begin{proof}
We only prove (i), since (ii) is similar. Suppose that $a \in L^m$ is an image apex over $K$, and that $h \in L(x,t) \setminus \{0\}$ is algebraic over $K\big((1-t)H+ta\big)$, such that $h|_{t=0} = 0$. Take $f \in K[y,u]$ of minimum degree, such that $f\big((1-t)H+ta,h\big) = 0 \ne f\big((1-t)H+ta,u\big)$. Let $f_0 \in K[y]$ be the constant part of $f$ with respect to $u$. By substituting $t=0$ in $f\big((1-t)H+ta,h\big) = 0$, we see that $f_0(H) = 0$. Hence $f_0\big((1-t)H+ta\big) = 0$. So we can replace $f$ by $(f-f_0)/u$. This contradicts the fact $f$ has minimum degree.

Suppose that $a \in L^m$ is not an image apex over $K$. Then there exists an $f \in K[y]$, such that $f(H) = 0 \ne f\big((1-t)H+ta\big)$. Hence $t \mid f\big((1-t)H+ta\big) \ne 0$.
%To prove the converse, notice that
%$$
%\trdeg_K K\big((1-t)H+ta,t,a\big) = \trdeg_K K\big(H,t,a\big) = \trdeg_K K\big(H,a) + 1
%$$
%So if the components of $a$ are algebraic over $K\big((1-t)H+ta,t\big)$, then 
%$$
%\trdeg_K K\big((1-t)H+ta,t\big) = \trdeg_K K\big(H,a\big) + 1 > \trdeg_K K(H)
%$$
%By applying (i) of proposition \ref{trdegprop} on the right hand side, the result follows.
\end{proof}

\begin{proposition} \label{relprop}
Suppose that $H, H' \in L[x]^m$, such that $f(H) = 0 \Longleftrightarrow f(H') = 0$
for every $f \in K[y]$. If $q \in K^m$, then $q$ is a (projective) image apex of 
$H$ over $K$, if and only if $q$ is a (projective) image apex of $H'$ over $K$.
\end{proposition}

\begin{proof}
We only prove the case where $q$ is an image apex of $H$ over $K$, since the 
other case is similar. Suppose that $q \in K^m$ is an image apex of $H$ over $K$. 
Take $f \in K[y]$ such that $f(H') = 0$. Then $f(H) = 0$ and $f\big((1-t)H + t q\big) = 0$. 
Hence for each $i$, $g_i(H) = 0$, where $g_i$ is the coefficient of $t^i$ in 
$f\big((1-t)y + tq\big) \in K[y][t]$. By assumption, $g_i(H') = 0$ for every 
$i$ as well. So $f\big((1-t)H' + t q\big) = 0$.
\end{proof}

\begin{lemma} \label{tulem}
Let $H \in L[x]^m$ and $f \in K[y]$. Then for the following statements:
\begin{enumerate}[\upshape (1)]

\item $f\big((1-p_{m+1}t)H + tp\big) = 0$ and 
$$
g(H) = 0 \Longrightarrow g\big((1-q_{m+1}u)H+uq\big) = 0
$$
for every $g \in K(p,p_{m+1})(y)$;

\item $f\big((1-p_{m+1}t)\big((1-q_{m+1}u)H + uq\big) + tp\big) = 0$;

\item $f\big((1-p_{m+1}t-q_{m+1}u)H + tp + uq\big) = 0$;

\item $f\big((1-q_{m+1}u)\big((1-p_{m+1}t)H + tp\big) + uq\big) = 0$;

\end{enumerate}
we have $(1) \Rightarrow (2) \Leftrightarrow (3) \Leftrightarrow (4)$.
\end{lemma}

\begin{proof}
Notice that (3) is symmetric in some sense and that (2) and (4) are reflections
of each other in the same sense. Hence (3) $\Leftrightarrow$ (4) follows in a similar
manner as (2) $\Leftrightarrow$ (3). 
So the following implications remain to be proved.
\begin{description}
 
\item[(1) \imp (2)] Assume (1). Then $g_i(H) = 0$, where $g_i$ is the coefficient of $t^i$
of $f\big((1-p_{m+1}t)y+tp\big)$. By assumption, $g_i\big((1-q_{m+1}u)H + uq\big) = 0$
for each $i$ as well, which gives (2).

\item[(2) \imp (3)] Assume (2). Then 
$$
f\big((1-p_{m+1}t)H - (1-p_{m+1}t)q_{m+1}uH + (1-p_{m+1}t)uq + tp\big) = 0
$$
and (3) follows by substituting $u = u/(1-p_{m+1}t)$.

\item[(3) \imp (2)] This is the converse of (2) $\Rightarrow$ (3), so $u$ is to
be substituted by $u\cdot(1-p_{m+1}t)$. \qedhere

\end{description}
\end{proof}

\begin{corollary} \label{tucor}
Let $H \in L[x]^m$ and $q \in L^m$. Then the following holds.
\begin{enumerate}[\upshape (i)]

\item For every $a \in L^m$, $a$ is an image apex of $(1-t)H + ta$ over $L$.

\item For every $p \in L^m$, $p$ is a projective image apex of $H + tp$ over $L$.

\item Suppose that $a \in K^m$. If $q$ is a (projective) image apex of $H$ over $K$, 
then $q$ is a (projective) image apex of $(1-t)H + ta$ over $K$. 
The converse holds if $a$ is an image apex of $H$ over $K$.

\item Suppose that $p \in K^m$. If $q$ is a (projective) image apex of $H$ over $K$, 
then $q$ is a (projective) image apex of $H + tp$ over $K$. 
The converse holds if $p$ is a projective image apex of $H$ over $K$.

\end{enumerate}
\end{corollary}

\begin{proof}
If we substitute $t$ by $t + u$, then (i) and (ii) follow 
from (3) $\Rightarrow$ (4) of lemma \ref{tulem}, where we take 
$(p,p_{m+1}) = (q,q_{m+1}) = (a,1)$ and $(p,p_{m+1}) = (q,q_{m+1}) = (p,0)$ respectively.

The forward implications in (iii) and (iv) follow by taking $(p,p_{m+1}) = (a,1)$
and $(p,p_{m+1}) = (p,0)$ respectively in (1) $\Rightarrow$ (4) of lemma \ref{tulem}.

The backward implications in (iii) and (iv) follow by taking $(p,p_{m+1})$ as above,
and substituting $t = 0$ in
\begin{align*}
f(H) = 0 &\Longrightarrow f\big((1-p_{m+1}t)H + tp\big) = 0 \\
&\Longrightarrow f\big((1-q_{m+1}u)\big((1-p_{m+1}t)H + tp\big) + uq\big) = 0 \qedhere
\end{align*}
\end{proof}

\begin{theorem} \label{piacalc}
Let $H \in L[x]^m$. Then the following holds.
\begin{enumerate}[\upshape (i)]

\item If $a \in L^m$ and $b \in L^m$ are image apices of $H$ over $K$, then for all 
$\lambda \in L$, $(1-\lambda) a + \lambda b$ is an image apex of $H$ 
over $K$ as well.

\item If $a \in L^m$ and $b \in L^m$ are image apices of $H$ over $K$, then for all 
$\lambda \in L$, $\lambda(a-b)$ is either zero or a projective image 
apex of $H$ over $K$.

\item If $a \in L^m$ is an image apex and $p \in L^m$ is a projective image apex of $H$ 
over $K$, then for all $\lambda \in L$, $a + \lambda p$ is an image apex 
of $H$ over $K$ as well.

\item If $p \in L^m$ and $q \in L^m$ are projective image apices of $H$ over $K$, then 
for all $\lambda, \mu \in L$, $\lambda p + \mu q$ is either zero or a 
projective image apex of $H$ over $K$ as well.

\end{enumerate}
\end{theorem}

\begin{proof}
Take $f \in K[y]$ arbitrary, such that $f(H) = 0$. Suppose that $a$, $b$, $p$ and $q$ are as indicated. Then $f\big((1-t)H + ta\big) = 0$ in (i), (ii) and (iii). Furthermore, $f\big(H + tp\big) = 0$ in (iii) and (iv).
\begin{enumerate}[\upshape (i)]

\item From (1) $\Rightarrow$ (3) of lemma \ref{tulem}, it follows that
$f\big((1-t-u)H + ta + ub\big) = 0$. Now substitute $t = (1-\lambda)t$ 
and $u = \lambda t$ in the given order.

\item From (1) $\Rightarrow$ (3) of lemma \ref{tulem}, it follows that
$f\big((1-t-u)H + ta + ub\big) = 0$. Now substitute $t = \lambda t$ and 
$u = -\lambda t$ in the given order.

\item From (1) $\Rightarrow$ (3) of lemma \ref{tulem}, it follows that
$f\big((1-u)H + tp + ua\big) = 0$. Now substitute $t = \lambda u$.

\item From (1) $\Rightarrow$ (3) of lemma \ref{tulem}, it follows that
$f\big(H + tp + uq\big) = 0$. Now substitute $t = \lambda t$ and 
$u = \mu t$ in the given order. \qedhere

\end{enumerate}
\end{proof}

\begin{lemma} \label{upia}
Let $H \in L[x]^m$ and $i \le m$. Then $e_i$ is a projective image apex 
of $H$ over $K$, if and only if $H_i$ is algebraically independent over 
$K$ of $H_1,H_2,\ldots,H_{i-1},\allowbreak H_{i+1},H_{i+2},\ldots,H_m$.
\end{lemma}

\begin{proof}
We only prove the forward implication, since the converse is similar.
Assume that $e_i$ is a projective image apex of $H$.
From (ii) of proposition \ref{trdegprop} and the fact that the substitution 
$t = t + H_i$ has an inverse, it follows that
\begin{align*}
\trdeg_K K(H) &= \trdeg_K K(H+t e_i) \\
&= \trdeg_K K(H_1,H_2,\ldots,H_{i-1},t,H_{i+1},H_{i+2},\ldots,H_m) \\
&= \trdeg_K K(H_1,H_2,\ldots,H_{i-1},\hphantom{t,{}}H_{i+1},H_{i+2},\ldots,H_m) + 1 
\end{align*}
So $H_i$ is algebraically independent over $K$ of 
$H_1,H_2,\ldots,H_{i-1},H_{i+1},H_{i+2},\ldots,\allowbreak H_m$ indeed.
\end{proof}

\begin{proposition} \label{nopia}
Let $H \in L[x]^m$, such that $e_1, e_2, \ldots, e_s$ are projective image 
apices of $H$ over $K$.
\begin{enumerate}[\upshape (i)]

\item $a \in L^m$ is an image apex of $H$ over $K$, if and only if 
$(a_{s+1},a_{s+2},\ldots,a_m)$ is an image apex of 
$(H_{s+1},H_{s+2},\ldots,H_m)$ over $K$.

\item If $p \in L^m$ such that $p_i \ne 0$ for some $i > s$, 
then $p$ is a projective image apex of $H$ over $K$, if and only if 
$(p_{s+1}, p_{s+2}, \ldots, p_m)$ is a projective image apex of 
$(H_{s+1},H_{s+2},\ldots,H_m)$ over $K$.

\item $\trdeg_K K(H) = \trdeg_K K(H_{s+1},H_{s+2},\ldots,H_m) + s$.

\end{enumerate}
\end{proposition}

\begin{proof}
We may assume that $s=1$, because the general case follows by induction.
The `only if'-parts of (i) and (ii) follow from proposition \ref{apexmap}. 
To prove the `if'-parts, suppose that $\tilde{q} := (q_2,q_3,\ldots,q_m)$ 
is a (projective) image apex of $\tilde{H} := (H_2,H_3,\ldots,H_m)$. 

Using proposition \ref{apexmap} and (iv) of corollary \ref{tucor}, 
we deduce that $(0,\tilde{q})$ is a (projective) image apex of 
$(0,\tilde{H})$ and $(t,\tilde{H})$ respectively over $K$. 
By substituting $t = t - H_1$, $t = t + H_1 + u q_1$, and $t = t + H_1 + (u-1)^{-1}u q_1$, we see that
\begin{align*}
f\big(t,\tilde{H}\big) = 0 
&\Longleftarrow f\big(H_1+t,\tilde{H}\big) = 0 \\
f\big(t,\tilde{H}+u\tilde{q}\big) = 0 
&\Longrightarrow f\big(H_1+t+uq_1,\tilde{H}+u\tilde{q}\big) = 0 \\
f\big((1-u)t,(1-u)\tilde{H}+u\tilde{q}\big) = 0 
&\Longrightarrow f\big((1-u)(H_1+t)+uq_1,(1-u)\tilde{H}+u\tilde{q}\big) = 0 
\end{align*}
Consequently, $(q_1,\tilde{q})$ is a (projective) image apex of $(H_1+t,\tilde{H})$ over $K$. Since $e_1$ is a projective image apex of $H$ over $K$, we deduce from (iv) of corollary \ref{tucor} that $(q_1,\tilde{q})$ is a (projective) image apex of $H$ over $K$.
This proves the `if'-parts of (i) and (ii).

(iii) follows from the case where $i = 1$ of (the proof of) lemma \ref{upia}.
\end{proof}

\begin{corollary} \label{plan}
Suppose that $K$ is algebraically closed in $L$. Let $H \in L[x]^m$ 
and define $r := \trdeg_K K(H)$. Say that $H$ has exactly $s$ 
independent projective image apices $p \in K^m$ over $K$. 
Then the following statements are equivalent:
\begin{enumerate}[\upshape (1)]

\item $r \le s$;

\item $r-1 \le s \le r$ and $H$ has an image apex $a \in K^m$ over $K$;

\item the ideal $(f \in K[y]\,|\,f(H) = 0)$ is generated by polynomials
of degree at most one.

\end{enumerate}
Furthermore, if any of {\upshape (1)}, {\upshape (2)}, {\upshape (3)} is
satisfied, then the constant part of $H$ is an image apex of $H$ over $K$.
\end{corollary}

\begin{proof}
Notice that $K$ is algebraically closed in $L(x)$ as well as $L$.
We may assume that $e_1, e_2, \ldots, e_s$ are projective image apices
of $H$. Then $(f \in K[y]\,|\,f(H) = 0)$ is generated by polynomials
in $K[y_{s+1},y_{s+2},\ldots,y_m]$. From proposition \ref{nopia}, it 
follows that we may assume that $s = 0$.
\begin{description}

\item[(1) \imp (2), (3)]
Assume (1). Then $0 \le r \le s = 0$, so $r = 0$ and $r-1 \le s \le r$.
Since $r = 0$, it follows that every component of $H$ is algebraic
over $K$. As $K$ is algebraically closed in $L(x)$, we deduce that
$a := H \in K^m$. So $H = (1-t)H + ta$ and $a$ is an image apex of $H$ 
over $K$. Furthermore, $\big(f \in K[y]\,\big|\,f(H) = 0\big) = (y_1 - a_1, y_2 - a_2,
\ldots, y_m - a_m)$.

\item[(2) \imp (3)]
Assume (2). If $r = m$, then $\big(f \in K[y]\,\big|\,f(H) = 0\big) = (0)$, so assume that $r < m$.
If $r = 0$, then (1) is satisfied and (3) follows, so assume that $r > 0$.
Then $r \le s + 1 = 1$, so $r = 1$ and $m > 1$. Consequently, 
$\trdeg_K K(H_1,H_2) \le 1$, so there exists
a nonzero $f \in K[y_1,y_2]$ such that $f(H) = 0$. We show that we can take $f$ of degree $1$. If $H_1 - a_1 \in \{0,1\}$, then this is trivially the case, so assume the opposite. 

Since $H_1 - a_1 \ne 1$, we may assume that $y_1 - a_1 - 1 \nmid f$. Since $a$ is an image apex of $H$ over $K$, it follows that
$$
f\big((1-t) H + t a\big) = 0 \qquad \mbox{and} \qquad f\big((1-t)(H-a) + a\big) = 0
$$
As $H_1 - a_1 \ne 0$, we can substitute $t = 1 - (H_1 - a_1)^{-1}$, to obtain
$$
f\Big(1 + a_1, \frac{H_2-a_2}{H_1-a_1} + a_2\Big) = f\Big(\frac{H-a}{H_1 - a_1} + a\Big) = 0
$$
As $y_1 - a_1 - 1 \nmid f$ and $K$ is algebraically closed in $L(x)$, we infer that
$$
\frac{H_2-a_2}{H_1-a_1} \in K
$$
More generally, we can deduce that the components of $H$ are affinely linearly dependent over $K$ in pairs. This yields $m - 1 = m - r$ independent polynomials of degree $1$ in $\big(f \in K[y] \,\big|\, f(H) = 0\big)$, so these are a set of generators, and (3) follows.

\item[(3) \imp (1)] 
Assume (3). Then $\big(f \in K[y]\,\big|\,f(H) = 0\big)$ is generated by $m - r$
polynomials of degree at most one. There are $r$ independent vectors
$p$ which vanish on the linear parts of these $m - r$
polynomials of degree at most one. A straightforward computation shows
that these $r$ independent vectors $p$ are projective image apices 
of $H$ over $K$, which gives (1). 

\end{description}
The last claim follows from the proof of (1) $\Rightarrow$ (2), (3) and
(i) of proposition \ref{nopia}, because $H(0) - a$ is an $L$-linear 
combination of $e_1, e_2, \ldots, e_s$.
\end{proof}

\begin{proposition} \label{hmgapexred}
Let $H \in L[x]^m$. Then the zero vector is an image apex of $(tH,t)$
over $K$, and the following holds.
\begin{enumerate}[\upshape (i)]
 
\item $a$ is an image apex of $H$ over $K$, if and only if 
$(a,1)$ is a (projective) image apex of $(tH,t)$ over $K$.

\item $p$ is a projective image apex of $H$ over $K$, if and only if 
$(p,0)$ is a (projective) image apex of $(tH,t)$ over $K$ 
and $p \ne 0$.

\end{enumerate}
\end{proposition}

\begin{proof}
From (i) of corollary \ref{tucor} and proposition \ref{relprop}, 
we deduce that the zero vector is an image apex of 
$\big((1-t)H,(1-t)\big)$ and $(tH,t)$ respectively over $K$.
Hence it follows from (ii) and (iii) of theorem \ref{piacalc} 
that $(q,q_{m+1})$ is a nonzero image apex of $(tH,t)$ over $K$, 
if and only if $(q,q_{m+1})$ is a projective image apex of $(tH,t)$ 
over $K$. So it remains to prove (i) and (ii), where we may ignore the 
parentheses around `projective' in (i) and (ii).

Notice that
\begin{align*}
f\big((1-q_{m+1}u)H+uq\big)\big|_{u=u/(t+q_{m+1}u)} 
&= f\bigg(\frac{\big((t+q_{m+1}u)-q_{m+1}u\big)H + uq}{t+q_{m+1}u}\bigg) \\
&= f\big((t+uq_{m+1})^{-1}(tH+uq)\big)
\end{align*}
and $f\big((t+uq_{m+1})^{-1}(tH+uq)\big)\big|_{t=1-uq_{m+1}} = f\big((1-q_{m+1}u)H+uq\big)$.
Hence
$$
f\big((1-q_{m+1}u)H+uq\big) = 0 \Longleftrightarrow f\big((t+uq_{m+1})^{-1}(tH+uq)\big) = 0
$$
So if $g = y_{m+1}^i f (y_{m+1}^{-1}y)$ for some $i \ge \deg f$, then 
\begin{equation} \label{fghmg}
f\big((1-q_{m+1}u)H+uq\big) = 0 \Longleftrightarrow g(tH+uq,t+uq_{m+1}) = 0
\end{equation}
If we substitute $u = 0$ on both sides of \eqref{fghmg}, then we see that 
$f(H) = 0 \Longleftrightarrow g(tH,t) = 0$. Using that and \eqref{fghmg} itself,
the `if'-parts of (i) and (ii) follow by taking $q_{m+1} = 1$ and $q_{m+1} = 0$ 
respectively.

To prove the `only if'-parts of (i) and (ii) as well, take any $g \in K[y,y_{m+1}]$ 
such that $g(tH,t) = 0$ and let $g_i$ be the coefficient of $t^i$ of $g(ty,ty_{m+1})$. 
Then $g_i$ is homogeneous of degree $i$, so there exists an 
$f_i \in K[y]$ of degree at most $i$ such that $g_i = y_{m+1}^i f_i(y_{m+1}^{-1}y)$.
Furthermore, $g_i(tH,t)$ is $t^i$ times the coefficient of $t^i$ of $g(tH,t)$, which 
is zero. Hence the `only if'-parts of (i) and (ii) can be proved in a similar manner
as the `if'-parts.
\end{proof}

\section{Inheritance}

We need the following proposition to define inheritance.

\begin{proposition} \label{inherprop}
Suppose that $h \in L[x]^m$ and $H \in L[X]^M$ are \HE-paired through matrices
$B \in \Mat_{m,M}(K)$ and $C \in \Mat_{N,n}(L)$. 
\begin{enumerate}[\upshape (i)]

\item
Suppose that $a \in L^m$ is an image apex of $H$ over $K$. Then $B a$ is an image apex of both $h$ and $BH$ over $K$, and $(1-t)H + ta$ is algebraic over $K\big((1-t)BH+tBa\big)$. Furthermore, if $B a' = Ba$, then $a'$ is an image apex of $H$ over $K$, if and only if $a' = a$.

\item
Suppose that $p \in L^m$ is a projective image apex of $H$ over $K$ and $Bp \ne 0$. Then $B p$ is a projective image apex of both $h$ and $BH$ over $K$, and $H + tp$ is algebraic over $K\big(BH+tBp\big)$. Furthermore, if $B p' = Bp$, then $p'$ is a projective image apex of $H$ over $K$, if and only if $p' = p$.

\end{enumerate}
\end{proposition}

\begin{proof}
We only prove the first claim, because the second claim is similar. 
On account of proposition \ref{apexmap}, $Ba$ is an image apex of $BH$ over $K$.
To prove that $Ba$ is an image apex of $h$ as well, suppose that $f(h) = 0$ for some $f \in K[y]$. Then $f(BH) = 0$ on account of proposition \ref{HEphi}. Since $Ba$ is an is an image apex of $BH$ over $K$, $f\big((1-t)BH + t Ba\big) = 0$ follows. So $f\big((1-t)h + t Ba\big) = 0$ by definition of $h$.

From proposition \ref{trdegprop} and \eqref{HEphieq}, it follows that
\begin{align*}
\trdeg_K K\big((1-t)H + ta\big) &= \trdeg_K K\big(H\big) = \trdeg_K K\big(BH\big) \\
&= \trdeg_K K\big((1-t)BH + t Ba\big)
\end{align*}
Since $K\big((1-t)BH + t Ba\big) \subseteq K\big((1-t)H + ta\big)$, we infer that
$(1-t)H + ta$ is algebraic over $K\big((1-t)BH + t Ba\big)$.

Suppose that $a' \in L^m$ is another image apex of $H$ over $K$, such that $B a' = Ba$. 
Then $(1-t)H + ta'$ is algebraic over $K\big((1-t)BH + t Ba'\big) = K\big((1-t)BH + t Ba\big)$ 
as well. So $t(a'-a) = \big((1-t)H + ta'\big) - \big((1-t)H + ta\big)$ is algebraic over 
$K\big((1-t)BH + t Ba\big)$. Hence $a' = a$ on account of proposition \ref{algdepprop}. 
\end{proof}

\begin{definition}
Suppose that $h \in L[x]^m$ and $H \in L[X]^M$ are \HE-paired through matrices
$B \in \Mat_{m,M}(K)$ and $C \in \Mat_{N,n}(L)$. 

We say that an image apex $a'\in L^m$ of $h$ over $K$ \emph{is inherited from} 
$H$ if $a' = Ba$ for some image apex $a \in L^m$ of $H$ over $K$ 
(which is unique if it exists).

We say that a projective image apex $p' \in L^m$ of $h$ over $K$ \emph{is inherited from} 
$H$ if $p' = Bp$ for some projective image apex $p \in L^m$ of $H$  over $K$
(which is unique if it exists).
\end{definition}

The aim of theorem \ref{cominher} below is to reduce to the case where $\trdeg_L L(H) = m - 1$ in the proof of theorem \ref{piaLK}. As a consequence, theorems \ref{piaLK}, \ref{iaLK} and \ref{xsc} do not contain the condition $\trdeg_L L(H) \ge m - 1$.

\begin{theorem} \label{cominher}
Assume that $h \in L[x]^m$ and $H \in L[X]^M$ are \HE-paired through matrices 
$B \in \Mat_{m,M}(K)$ and $C \in \Mat_{N,n}(L)$. Suppose that
$$
B = B^{(3)}B^{(1)} = B^{(4)}B^{(2)}
\qquad \mbox{and} \qquad 
C = C^{(1)}C^{(3)} = C^{(2)}C^{(4)}
$$
such that $B^{(1)}, B^{(2)}, B^{(3)}, B^{(4)}$ are matrices over $K$, and 
$C^{(1)}, C^{(2)}, C^{(3)}, C^{(4)}$ are matrices over $L$, 
with appropriate dimensions.

Then $H^{(i)} := B^{(i)}H(C^{(i)}X)$ and $H$ are \HE-paired through 
$B^{(i)}$ and $C^{(i)}$ for both $i \le 2$. 
Furthermore, $h$ and $H^{(i)}$ are \HE-paired through $B^{(i+2)}$ and $C^{(i+2)}$ 
for both $i \le 2$. 
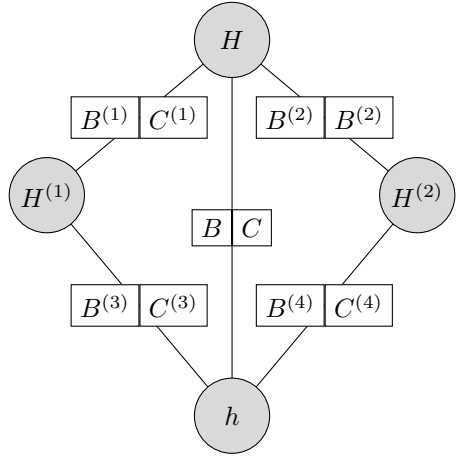
\begin{center}
\begin{tikzpicture}[size/.style={inner sep=0pt,minimum size=#1}]
\draw (-90:2.5) node[shape=circle,fill=black!15,size=1cm,draw] {$h$} -- 
      node[left,fill=white,draw] {$B^{(4)}$} 
      node[right,fill=white,draw] {$C^{(4)}$}
      (10:2.5) node[shape=circle,fill=black!15,size=1cm,draw] {$H^{(2)}$} --
      node[left,fill=white,draw] {$B^{(2)}$} 
      node[right,fill=white,draw] {$C^{(2)}$}
      (90:2.5) node[shape=circle,fill=black!15,size=1cm,draw] {$H$} --
      node[left,fill=white,draw] {$B^{(1)}$} 
      node[right,fill=white,draw] {$C^{(1)}$}
      (170:2.5) node[shape=circle,fill=black!15,size=1cm,draw] {$H^{(1)}$} -- 
      node[left,fill=white,draw] {$B^{(3)}$} 
      node[right,fill=white,draw] {$C^{(3)}$}
      (-90:2.5) -- node[left,draw] {$B$} node[right,draw] {$C$} (90:2.5);
\end{tikzpicture}
\end{center}
Let $q \in L^m$ be an image apex or a projective image apex of $h$ over $K$.
If $q$ is inherited from $H$, then $q$ is inherited from $H^{(i)}$ for 
both $i$ and from $B^{(i)}H$ for both $i$.

If $q$ is inherited from $B^{(i)}H$ for both $i$ and 
$$
\ker B^{(1)} \cap \ker B^{(2)} = \{0\}
$$
then $q$ is inherited from $H$.
Furthermore, the (projective) image apex $Q \in L^m$ of $H$ over $K$ such that 
$BQ = q$ is determined by $B^{(1)}Q = Q^{(1)}$ and $B^{(2)}Q = Q^{(2)}$,
where $Q^{(i)}$ is a (projective) image apex of $B^{(i)}H$ over $K$
such that $q = B^{(i+2)}Q^{(i)}$ for both $i$.
\end{theorem}

\begin{proof}
The claims about \HE-pairing follow from 
$$
\trdeg_K K(h) \le \trdeg_K K(H^{(i)}) \le \trdeg_K K(H) = \trdeg_K K(h)
$$
for both $i$.
From the other claims, we only prove the case where $q$ is a projective image 
apex of $h$.

To prove the first claim of this theorem, suppose that $q$ is inherited from $H$, 
say that $q = BQ$, such that $Q$ is a projective image apex of $H$. From proposition 
\ref{inherprop}, it follows that $B^{(i)}Q$ is a projective image apex of $H^{(i)}$ 
for both $i$ and of $B^{(i)}H$ for both $i$. So $q = B^{(i+2)}B^{(i)}Q$ is inherited from $H^{(i)}$ for both $i$ and from $B^{(i)}H$ for both $i$, 
which is the first claim of this theorem.

To prove the second claim, suppose that $q$ is inherited from $B^{(i)}H$ for both
$i$ and that $\ker B^{(1)} \cap \ker B^{(2)} = \{0\}$. Then for both $i$, 
there exists a $Q^{(i)}$ which is a pro\-jec\-tive image apex of $B^{(i)}H$, 
such that $q = B^{(i+2)}Q^{(i)}$. 

Take $i \le 2$ arbitrary. On account of proposition 
\ref{inherprop}, $q$ is also a projective image apex of $B^{(i+2)}B^{(i)}H = BH$.
From proposition \ref{trdegprop} and \eqref{HEphieq}, it follows that 
\begin{align*}
\trdeg_K K\big(B^{(i)}H + t Q^{(i)}\big) &= \trdeg_K K\big(B^{(i)}H\big) = \trdeg_K K\big(H\big) \\
&= \trdeg_K K\big(BH\big) = \trdeg_K K\big(BH + t q\big)
\end{align*}
Since 
$$
K\big(B^{(i)}H + t Q^{(i)}\big) \supseteq K\big(B^{(i+2)}(B^{(i)}H + t Q^{(i)})\big) 
= K\big(BH + t q\big)
$$
we deduce that $B^{(i)} H + t Q^{(i)}$ is algebraic over $K\big(BH + tq\big)$.

Take $j \le M$ arbitrary. Since $\ker B^{(1)} \cap \ker B^{(2)} = \{0\}$, the row spaces of $B^{(1)}$ and $B^{(2)}$ add up to $K^M$. Since the $j$-th standard basis unit vector of 
$K^M$ is included, there exists a $Q_j \in K$ such that $H_j + t Q_j$ is algebraic over $K\big(BH + tq\big)$. Let $Q = (Q_1, Q_2,\ldots, Q_M)$. Then $H + t Q$ and hence also $BH + t BQ$ is algebraic over $K\big(BH + tq\big)$. So $t(q-BQ) = (BH + tq) - (BH + t BQ)$ is algebraic over $K\big(BH + tq\big)$. Hence $q = BQ$ on account of proposition \ref{algdepprop}. Furthermore,
$$
\trdeg_K K\big(H + t Q\big) = \trdeg_K K\big(BH + tq\big) 
= \trdeg_K K\big(BH\big) = \trdeg_K K\big(H\big)
$$
It follows from proposition \ref{trdegprop} that $Q$ is a projective image apex of $H$, which is the second claim of this theorem.

From proposition \ref{inherprop}, we infer that $B^{(i)} Q = Q^{(i)}$ for both $i$.
If $\tilde{Q} \in L^m$ satisfies $B^{(i)} \tilde{Q} = Q^{(i)}$ for both $i$, then
$Q - \tilde{Q} \in \ker B^{(1)} \cap \ker B^{(2)} = \{0\}$. So 
$B^{(1)}Q = Q^{(1)}$ and $B^{(2)}Q = Q^{(2)}$ determine $Q$, which is the last claim
of this theorem.
\end{proof}

\begin{lemma} \label{fLK}
Let $L$ be an extension field of a field $K$ and $H \in K[x]^m$.
Suppose that $g \in L[y]$, say that $g = \lambda_1 g^{(1)} + \lambda_2 g^{(2)} + \cdots + 
\lambda_r g^{(r)}$, where $g^{(i)} \in K[y]$ and $\lambda_i \in L$ for all $i$.

If $\lambda_1, \lambda_2, \ldots, \lambda_r$ are linearly independent
over $K$, then $g(H) = 0$ implies $g^{(i)}(H) = 0$ for all $i$.
\end{lemma}

\begin{proof}
Suppose that $\lambda_1, \lambda_2, \ldots, \lambda_r$ are linearly independent
over $K$. Then $\lambda_1, \lambda_2, \allowbreak \ldots, \lambda_r$ are linearly independent
over $K(x)$ as well. Since
$$
g(H) = \lambda_1 g^{(1)}(H) +  \lambda_2 g^{(2)}(H) + \cdots + \lambda_r g^{(r)}(H)
$$
it follows that $g(H) = 0$ implies $g^{(i)}(H) = 0$ for all $i$.
\end{proof}

\begin{lemma} \label{ptLK}
Suppose that $K$ is algebraically closed in $L$. Assume that $f \in K[y]$ and $g \in L[y]$, 
such that $g \mid f$ over $L[y]$. 
Then there exists a unit $\lambda \in L$ such that $\lambda^{-1} g \in K[y]$.
\end{lemma}

\begin{proof}
Let $d := \deg g$. We distinguish two cases.
\begin{itemize}
 
\item \emph{$f \in K[y_1]$.} \\
Then there exists an extension field $\tilde{L}$ of $L$, such that $f$ 
decomposes into linear factors over $\tilde{L}$. So there are 
$\alpha_i \in \tilde{L}$ such that
$$
g = \lambda \prod_{i=1}^d (y_1 - \alpha_i)
$$
where $\lambda$ is the leading coefficient of $g$ and $d := \deg g$. 
From $g \mid f \in K[y_1]$ it follows that $f(\alpha_i)= 0$ for all $i$. 
So $\alpha_i$ is algebraic over $K$ for all $i$. 
Furthermore, since the coefficients of $\lambda^{-1} g$ are contained in 
$K[\alpha_1,\alpha_2,\ldots,\alpha_d]$, they are algebraic over $K$ as well.
On the other hand, $\lambda^{-1}g \in L[y]$ and $K$ is algebraically closed in 
$L$, so $\lambda^{-1}g \in K[y_1]$.

\item \emph{$f \notin K[y_1]$.} \\
Define 
\begin{align*}
\tilde{f} &:= f\Big(y_1,y_1^{d+1},y_1^{(d+1)^2},\ldots,y_1^{(d+1)^{m-1}}\Big) \\
\tilde{g} &:= g\Big(y_1,y_1^{d+1},y_1^{(d+1)^2},\ldots,y_1^{(d+1)^{m-1}}\Big)
\end{align*}
Then $\tilde{g} \mid \tilde{f}$ over $L$. From the case $f \in K[y_1]$ above, 
it follows that there exists a $\lambda \in L$ such that $\lambda^{-1} \tilde{g} \in K[y]$.
As $d+1 > \deg f$, the set of nonzero coefficients of $\tilde{g}$ is the same as that
of $g$. So $\lambda^{-1} g \in K[y]$. \qedhere

\end{itemize}
\end{proof}

\begin{lemma} \label{pquasi}
Let $L$ be an extension field of a field $K$ of characteristic zero.
Take $H \in L[x]^m$ and $p \in L^m$.

Suppose that $\trdeg_K K(H) \le \trdeg_L L(H) = m - 1$. Then there exists a nonzero 
$f \in K[y]$ such that $f(H) = 0$. If we take $f$ as such of minimum degree,
then $F := (\grad_y f)(H) \ne 0$. 
Furthermore, $F\tp \cdot \jac H = 0$ and the following statements are equivalent:
\begin{enumerate}[\upshape (1)]

\item $p$ is a projective image apex over $K$ of $H$ or zero;

\item $p$ is a projective image apex over $L$ of $H$ or zero;

\item $G\tp \cdot p = 0$ for every $G \in L(x)$, such that $G\tp \cdot \jac H = 0$;

\item there exists a nonzero $G \in L(x)^m$ such that $G\tp \cdot \jac H = 0$ and 
$G\tp \cdot p = 0$;

\item $\jac_y g \cdot p = 0$ for every $g \in L[y]$ of minimum degree, such that $g(H) = 0$;

\item there exists a nonzero $g \in L[y]$ such that $g(H) = 0$ and
$\jac_y g \cdot p = 0$.

\end{enumerate}
\end{lemma}

\begin{proof}
Suppose that $\trdeg_K K(H) \le \trdeg_L L(H) = m - 1$. Then $f$ exists.
Furthermore, it follows from proposition 1.2.9 of either \cite{MR1790619} 
or \cite{homokema} that $\rk \jac H = m - 1$.

Since $f$ has minimum degree and $\deg \grad_y f < \deg f$, we deduce that $F \ne 0$.
From the chain rule, it follows that $\jac f(H) = (\jac_y f)|_{y=H} \cdot \jac H$, 
so $F\tp \cdot \jac H = 0$.
\begin{description} 

\item[(2) \imp (1)] This follows from $K[y] \subseteq L[y]$ and the definition of 
projective image apex.

\item[(1) \imp (3)]  Assume (1). Then $f(H+tp) = 0$, so
$$
0 = \Big(\parder{}{t} f(H+tp) \Big)\Big|_{t=0} 
  = \big((\jac_y f)|_{y=H+tp} \cdot p \big)\big|_{t=0} = F\tp \cdot p
$$
Now take any $G \in L(x)$, such that $G\tp \cdot \jac H = 0$.
Since $\rk \jac H = m-1$, there is only one dependence between the rows of
$\jac H$, so $G$ is dependent as a vector over $L(x)$ on $F$. Hence $G\tp \cdot p = 0$. 

\item[(3) \imp (4)] As we can take $G = F$, there exists 
a $G \in L(X)^m$ such that $G\tp \cdot \jac H = 0$. This gives (3) $\Rightarrow$ (4).

\item[(4) \imp (5)] Assume (4). Take any $g \in L[y]$ of minimum degree,
such that $g(H) = 0$, if such a $g$ exists. 
Just like $F \ne 0$, $\tilde{G} := (\grad_y g)(H) \ne 0$.
Furthermore, $\tilde{G}\tp \cdot \jac H = 0$. 

Since $\rk \jac H = m-1$, there is only one dependence between the rows of
$\jac H$, so $\tilde{G}$ is dependent as a vector over $L(x)$ on $G$. Hence 
$\tilde{G}\tp \cdot p = 0$. Since $g$ has minimum degree, we 
conclude that $\jac_y g \cdot p = 0$. 

\item[(5) \imp (6)]  As $f \in L[y]$ and $f(H) = 0$, there exists 
a $g \in L[y]$ of minimum degree, such that $g(H) = 0$.
This gives (5) $\Rightarrow$ (6).

\item[(6) \imp (2)] Assume (6). Then $g(H) = 0$ and
$$
\parder{}{t} g(H + tp) = \big((\grad_y g)(H+tp)\big)\tp \cdot p 
= \big((\grad_y g)\tp \cdot p\big)\big|_{y=H+tp} = 0
$$
Hence $g(H + tp) = 0$, so $\trdeg_L L(H+tp) \le m-1 = \trdeg_L L(H)$. 
Since $\trdeg_L L(H+tp) \ge \trdeg_L L(H)$ follows by substituting $t = 0$, 
(2) follows from (ii) of proposition \ref{trdegprop}. \qedhere

\end{description} 
\end{proof}

Theorem \ref{piaLK} expresses that, in the situation which is indicated, the projective image apices over $K$ with coordinates in the fancy field $\tilde{L}$ are essentially projective images with coordinates in $K$. Theorem \ref{iaLK} expresses a similar thing for image apices instead of projective image apices. These results allow us to use geometry over $\C$ to prove the existence of (projective) image apices, because Lefschetz' principle allows us to assume that $\tilde{L} \subseteq \C$.

\begin{theorem} \label{piaLK}
Suppose that $K$ has characteristic zero and that $K$ is algebraically closed in
$L$. Assume that $H \in L[x]^m$, such that $\trdeg_K K(H) \le \trdeg_L L(H)$. 
Let $\tilde{L}$ be an extension field of $L$ and suppose that $p \in \tilde{L}^m$, 
such that $p$ is a projective image apex of $H$ over $K$. 

Take $\lambda_i \in \tilde{L}$, such that $\lambda_1, \lambda_2, \ldots, \lambda_r$
are linearly independent over $K$ and
$$
p = \lambda_1 p^{(1)} + \lambda_2 p^{(2)} + \cdots + \lambda_r p^{(r)}
$$
where $p^{(i)} \in K^m$ for each $i$. Then $p^{(i)}$ is a projective image apex of 
$H$ over $\tilde{L}$ or zero for each $i$.
\end{theorem}

\begin{proof}
We distinguish two cases.
\begin{itemize}
 
\item \emph{$\trdeg_L L(H) = m - 1$.} \\
Since $\trdeg_K K(H) \le \trdeg_L L(H) = m - 1$, there exists a nonzero polynomial 
$f \in K[y]$ such that $f(H) = 0$. Take $f$ as such of minimum degree and define
$g := \jac_y f \cdot p$. If $g_i := \jac_y f \cdot p^{(i)} = 0$ for all $i$, then
it follows from (6) $\Rightarrow$ (2) of lemma \ref{pquasi} that
$p^{(i)}$ is a projective image apex of $H$ over $\tilde{L}$ or zero for each $i$,
because 
\begin{equation} \label{tildequasi}
\trdeg_{\tilde{L}} \tilde{L}(H) = \trdeg_L L(H) = m - 1
\end{equation}
We show that $g_i = 0$ for every $i$ indeed.

Suppose first that $K = L$. From lemma \ref{pquasi}, it follows that 
$(\jac_y f)(H) \cdot \jac H = 0$. Since $f(H+tp) = 0$, we deduce from 
\eqref{tildequasi} and (1) $\Rightarrow$ (3) of lemma \ref{pquasi} that 
$g(H) = (\jac_y f)(H) \cdot p = 0$. Hence it follows from lemma \ref{fLK} 
that $g^{(i)}(H) = 0$ for all $i$. Since $f$ has minimum degree and 
$\deg g_i < \deg f$ for all $i$, we conclude that $g^{(i)} = 0$ for every $i$. 

Suppose next that $K \ne L$. Since $\trdeg_L L(H) = m-1$, the 
ideal $\p := (\tilde{f} \in L[y] \,|\, \tilde{f}(H) = 0)$ is principal. As 
$f(H) = 0$, any polynomial $\tilde{f}$ which generates $\p$ is a divisor of $f$. 
From lemma \ref{ptLK}, it follows that $\p = (f)$, so $f$ has minimum degree as
a polynomial over $L$ as well. Hence the case $K = L$ above (with different
$\lambda_1, \lambda_2, \ldots$) tells us that $g = 0 + 0 + \cdots = 0$.
As $f \in K[y]$, it follows that $g^{(i)} \in K[y]$ for all $i$. By taking $H = (x_1,x_2,\ldots,x_m)$ in lemma \ref{fLK}, we infer that $g^{(i)} = 0$ for every $i$. 

\item \emph{$\trdeg_L L(H) < m - 1$.} \\
We will apply theorem \ref{cominher} to show that $p^{(i)}$ is a 
projective image apex of $H$ for each $i$.
Since $\trdeg_L L(H) \le m-2$, there exists a matrix $B \in \Mat_{m-2,m}(K)$
of rank $m-2$, such that $\trdeg_L L(BH) = \trdeg_L L(H)$. Take any 
$B^{(1)} \in \Mat_{m-1,m}(K)$ or rank $m-1$, such that the row space of
$B^{(1)}$ contains that of $B$, say that $B = B^{(3)}B^{(1)}$.

Take $B^{(2)}$ and $B^{(4)}$ in a similar manner as 
$B^{(1)}$ and $B^{(3)}$ respectively, but in such a way that the rows 
spaces of $B^{(1)}$ and $B^{(2)}$ are different. Since the row spaces
of $B^{(1)}$ and $B^{(2)}$ have codimension one, we deduce that they add up to
a space of codimension less than one, i.e.\@ $K^m$. So 
$\ker B^{(1)} \cap \ker B^{(2)} = \{0\}$.

Notice that by assumption,
\begin{align*}
f\big(B^{(1)}H\big) = 0 &\Longrightarrow f\big(B^{(1)}H + tB^{(1)}p\big) = 0 \\
f\big(B^{(2)}H\big) = 0 &\Longrightarrow f\big(B^{(2)}H + tB^{(2)}p\big) = 0
\end{align*}
for every $f \in K[y_1,y_2,\ldots,y_{m-1}]$. Furthermore, it follows from
$\trdeg_K \allowbreak K(H) \le \trdeg_L L(H) = \trdeg_L L(BH)$ that 
\begin{align*}
\trdeg_K K(B^{(1)}H) &\le \trdeg_L L(B^{(1)}H) \\
\trdeg_K K(B^{(2)}H) &\le \trdeg_L L(B^{(2)}H)
\end{align*}
By induction on $m$, $B^{(1)}p^{(i)}$ and $B^{(2)}p^{(i)}$ are projective image 
apices of $B^{(1)}H$ and $B^{(2)}H$ respectively for each $i$. 
From proposition \ref{apexmap}, it follows that $Bp^{(i)}$ is a projective image 
apex of $BH$ for each $i$. Now apply theorem \ref{cominher}.
\qedhere

\end{itemize}
\end{proof}

\begin{theorem} \label{iaLK}
Suppose that $K$ has characteristic zero and that $K$ is algebraically closed in
$L$. Assume that $H \in L[x]^m$, such that $\trdeg_K K(H) \le \trdeg_L L(H)$. 
Let $\tilde{L}$ be an extension field of $L$ and suppose that $a \in \tilde{L}^m$, 
such that $a$ is an image apex of $H$ over $K$. 

Then there are $\mu_i \in \tilde{L}$ which are linearly independent over $K$, 
such that $\mu_1 + \mu_2 + \cdots + \mu_r = 1$,
$$
a = \mu_1 a^{(1)} + \mu_2 a^{(2)} + \cdots + \mu_r a^{(r)}
$$
and $a^{(i)} \in K^m$ is an image apex of $H$ over $\tilde{L}$ for each $i$.
\end{theorem}

\begin{proof}
From proposition \ref{hmgapexred}, it follows
that $p := (a,1)$ is a projective image apex of $\tilde{H} := (x_{n+1} H, x_{n+1})$. 
Take $\lambda_i$ and $p^{(i)}$ as in theorem \ref{piaLK}, except that $p^{(i)} \in K^{n+1}$
instead of $K^n$. Then $\lambda_1, \lambda_2, \ldots, \lambda_r$ are linearly independent over $K$, 
$p = \lambda_1 p^{(1)} + \lambda_2 p^{(2)} + \cdots + \lambda_r p^{(r)}$ and $p^{(i)} \in K^{n+1}$ 
is a projective image apex of $\tilde{H}$ over $L'$.

Since the last coordinate of $p$ is nonzero, the last row of 
$\big(p^{(1)} \,\big|\, p^{(2)} \,\big|\, \cdots \,\big|\, p^{(r)}\big)$ is nonzero as well, so there
exists a $T \in \GL_r(K)$ such that the last row of 
$$
\big(q^{(1)} \,\big|\, q^{(2)} \,\big|\, \cdots \,\big|\, q^{(r)}\big) := 
\big(p^{(1)} \,\big|\, p^{(2)} \,\big|\, \cdots \,\big|\, p^{(r)}\big) \cdot T
$$
equals $(1^1 ~ 1^2 ~ \cdots ~ 1^r)$. Furthermore, we deduce from (iv) of theorem \ref{piacalc}
that $q^{(i)}$ is a projective image apex of $\tilde{H}$ for each $i$ as well.

Define $\mu := T^{-1} \lambda$. Then $\mu_1 + \mu_2 + \cdots + \mu_r = 1$ indeed, 
because the last coordinate of $q^{(i)}$ equals $1$ for each $i$, as well as the last
coordinate of 
\begin{align*}
p &= \big(p^{(1)} \,\big|\, p^{(2)} \,\big|\, \cdots \,\big|\, p^{(r)}\big) T \cdot T^{-1} \lambda \\
  &= \big(q^{(1)} \,\big|\, q^{(2)} \,\big|\, \cdots \,\big|\, q^{(r)}\big) \cdot \mu \\
  &= \mu_1 q^{(1)} + \mu_2 q^{(2)} + \cdots + \mu_r q^{(r)}
\end{align*}
Furthermore, if we define $a^{(i)}$ by $(a^{(i)},1) := q^{(i)}$, then 
$a = \mu_1 a^{(1)} + \mu_2 a^{(2)} + \cdots + \mu_r a^{(r)}$, and we deduce from 
proposition \ref{hmgapexred} that $a^{(i)}$ is an image apex of $H$ for each $i$.
\end{proof}

\begin{theorem} \label{xsc}
Let $K$ be a field of characteristic zero, which is algebraically closed in $L$.
Let $H \in L[x]^m$.
\begin{enumerate}[\upshape (i)]

\item Suppose that $H$ has a (projective) image apex $p \in L^m$ over $K$, but not 
a (projective) image apex $p^{(1)} \in K^m$ over $K$. Then
$$
\trdeg_L L(H) \le \trdeg_K K(H) - 1
$$

\item Suppose that $H$ has a (projective) image apex $a \in L^m$ over $K$, which 
is not a (projective) image apex over $L$. Then
$$
\trdeg_L L(H) \le \trdeg_K K(H) - 1
$$

\item Suppose that $H$ has a projective image apex $p \in L^m$ over $K$, but not 
a projective image apex $p^{(1)} \in K^m$ over $K$. 
Suppose that $H$ has a (projective) image apex $a \in L^m$ over $K$, which is not a 
(projective) image apex over $L$. Then 
$$
\trdeg_L L(H) \le \trdeg_K K(H) - 2
$$

\item Suppose that $H$ has an image apex $b \in L^m$ over $K$, but not 
an image apex $b^{(1)} \in K^m$ over $K$.
Suppose that $H$ has a (projective) image apex $a \in L^m$ over $K$, which is not a
(projective) image apex over $L$. Then 
$$
\trdeg_L L(H) \le \trdeg_K K(H) - 2
$$

\end{enumerate}
\end{theorem}

\begin{proof}
We assume that $a$ is an image apex, since
that are the cases we need. The other cases are similar.
\begin{enumerate}[\upshape (i)]

\item 
Suppose that $\trdeg_K K(H) \le \trdeg_L L(H)$. From theorem \ref{piaLK}
and theorem \ref{iaLK}, it follows that $H$ has a (projective) image apex 
$p^{(1)} \in K^m$ over $L$, in particular over $K$. Contradiction.

\item 
From $K \subseteq L$ and proposition \ref{trdegprop}, it follows that 
\begin{align} \label{trdegia}
0 &\le \trdeg_K K\big((1-t)H + ta\big) - \trdeg_L L\big((1-t)H + ta\big) \nonumber \\
&= \trdeg_K K(H) - (\trdeg_L L(H) + 1)
\end{align}
which gives (ii).

\item
Suppose that $\trdeg_K K(H) \le \trdeg_L L(H) + 1$. Then \eqref{trdegia} is an 
equality, so 
$$
\trdeg_K K\big((1-t)H + ta\big) = \trdeg_L L\big((1-t)H + ta\big)
$$
From corollary \ref{tucor} and theorem \ref{piaLK}, it follows that $(1-t)H + ta$ 
has a projective image apex $p^{(1)} \in K^m$ over $L$, in particular over $K$. 
Since $a$ is an image apex of $H$ over $K$, we deduce from corollary \ref{tucor} 
that $p^{(1)}$ is a projective image apex of $H$ over $K$ as well. Contradiction.

\item 
The proof is the same as that of (iii), except that theorem \ref{iaLK} is used 
instead of theorem \ref{piaLK}. \qedhere

\end{enumerate}
\end{proof}

\section{Reduction of dimension with inheritance} \label{sec:reddimrplus1}

This section is `over $K$', i.e.\@ \HE-pairing is over $K$, through matrices
over $K$, and (projective) image apices are over $K$, both as such and as vectors.

The purpose of theorem \ref{cominherred} below is to reduce the proof of a reduction from arbitrary dimension to dimension $\rk \jac H + 1$ to a proof of a reduction from dimension $\rk \jac H + 2$ dimension to dimension $\rk \jac H + 1$. Namely, if $m = \rk \jac H + 2$ and $i = 1$, then the proof of the reduction from dimension $\rk \jac H + 2$ dimension to dimension $\rk \jac H + 1$ yields the properties of $h$ and $h'$. So $H$ can be reduced to $H'$, and further reductions follow by induction. 

\begin{theorem} \label{cominherred}
Suppose that $h \in K[x]^m$ and $H \in K[X]^M$ are (symmetrically) \HE-paired
through matrices $B$ and $C$, such that $\rk B = m$.
Suppose that $h' \in K[x_1,x_2,\ldots,x_{n-j}]^{m-i}$ and $h$ are (symmetrically) 
\HE-paired through matrices $B'$ and $C'$, such that every image apex of $h'$ is inherited from $h$ and/or every projective image apex of $h'$ is inherited from $h$. 

Then there exists an $H' \in K[x_1,x_2,\ldots,x_{N-j}]^{M-i}$ such that $H'$ and $H$ 
are (symmetrically) \HE-paired, and every image apex of $H'$ is inherited 
from $H$ and/or every projective image apex of $H'$ is inherited from 
$H$ respectively.
\end{theorem}

\begin{proof}
Notice that $h'$ and $H$ are \HE-paired through $B'B$ and $CC'$. 
Since $\rk B = m$, we can add $M-m$ rows to $B$ to obtain an element of $\GL_M(K)$.
But instead of doing that, we add these rows to $B'B$, to obtain a matrix 
$\tilde{B}' \in \Mat_{M-i,M}(K)$. In the symmetric case, we take 
$\tilde{C}' = (\tilde{B}')\tp$, otherwise we make $\tilde{C}' \in \Mat_{N,N-j}(K)$ 
by adding zero or more arbitrary columns to $CC'$.

Take $H' := \tilde{B}'H(\tilde{C}'x)$. Since $h'$ and $H$ are \HE-paired over $K$
through $B'B$ and $CC'$, one can deduce that $h'$ and $H'$ are 
\HE-paired through $\tilde{B} := (I_{m-i} ~ 0)$ and $\tilde{C} := (I_{n-j} ~ 0)\tp$ 
and that $H'$ and $H$ are \HE-paired through $\tilde{B}'$ and $\tilde{C}'$.
Furthermore, $\tilde{B}\tilde{B}' = B'B$ and $\tilde{C}'\tilde{C} = CC'$.
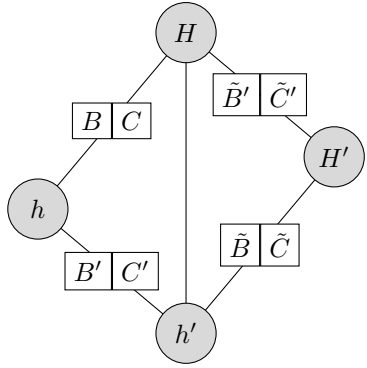
\begin{center}
\begin{tikzpicture}[size/.style={inner sep=0pt,minimum size=#1}]
\draw (-90:2) node[shape=circle,fill=black!15,size=8mm,draw] {$h'$} -- 
      node[left,fill=white,draw] {$\tilde{B}$} 
      node[right,fill=white,draw] {$\tilde{C}$}
      (10:2) node[shape=circle,fill=black!15,size=8mm,draw] {$H'$} --
      node[left,fill=white,draw] {$\tilde{B}'$} 
      node[right,fill=white,draw] {$\tilde{C}'$}
      (90:2) node[shape=circle,fill=black!15,size=8mm,draw] {$H$} --
      node[left,fill=white,draw] {$B$} 
      node[right,fill=white,draw] {$C$}
      (-170:2) node[shape=circle,fill=black!15,size=8mm,draw] {$h$} -- 
      node[left,fill=white,draw] {$B'$} 
      node[right,fill=white,draw] {$C'$}
      (-90:2) -- (90:2);
\end{tikzpicture}
\end{center}
Suppose that $Q'$ is a (projective) image apex of $H'$ over $K$.
Then $q' := \tilde{B}Q'$ is a (projective) image apex of $h'$ over $K$. 
By assumption, $q'$ is inherited from $h$. 
So there exists a (projective) image apex $q$ of $h$, such that $q' = B'q$. 

From propositions \ref{HEphi} and \ref{relprop}, it follows that $Q'$ and $q$ are
(projective) image apices of $\tilde{B}'H$ and $BH$ respectively as well.
Hence we can apply theorem \ref{cominher}, to deduce that there exists a (projective) 
image apex $Q$ of $H$ such that $\tilde{B}'Q = Q'$. So $Q'$ is inherited from $H$. 
\end{proof}

\begin{theorem} \label{Mred}
Let $K$ be a field of characteristic zero and $H \in K[x]^M$, such that
$\trdeg_K K(H) = M - 2$. Suppose that $0 \le s < k < M$ and that the polynomial
map $(H_{s+1}, H_{s+2}, \ldots, H_k)$ does not have a projective image apex.
Suppose additionally that either {\upshape (i)} or {\upshape (ii)} is
satisfied.
\begin{enumerate}[\upshape (i)]
 
\item $(H_{k+1}, H_{k+2}, \ldots, H_M)$ does not have a projective image apex. 

\item $(H_{k}, H_{k+1}, H_{k+2}, \ldots, H_M)$ does not have a projective image 
apex and $H$ has an image apex. 

\end{enumerate}
Let
$$
h = (H_1, H_2, \ldots, H_s, H_{s+1} + H_M, H_{s+2}, H_{s+3}, \ldots, H_{M-1})
$$
Then for both $H$ and $h$, the projective image apices are exactly those
nonzero vectors which are generated by $e_1, e_2, \ldots, e_s$. Furthermore,
$h$ and $H$ are \HE-paired and the projective image apices of $h$ are inherited 
from those of $H$.
\end{theorem}

\begin{proof}
Suppose that $P$ is a projective image apex of $H$. From (ii) of proposition 
\ref{apexmap}, it follows that $P_{s+1} = P_{s+2} = \cdots = P_k = 0$
and $(P_k =)\, P_{k+1} = P_{k+2} = \cdots = P_M = 0$. So $P$ is 
a linear combination of $e_1, e_2, \ldots, e_s$ indeed.

Since $(H_{s+1}, H_{s+2}, \ldots, H_k)$ does not have a projective image apex,
it follows from lemma \ref{upia} that $H_{s+1}$ is algebraically dependent over $K$ on 
$H_{s+2}, H_{s+3}, \allowbreak \ldots, H_k$. Since $((H_k,)\, H_{k+1}, H_{k+2}, 
\ldots, H_M)$ does not have a projective image apex, it follows from lemma \ref{upia} that 
$H_M$ is algebraically dependent over $K$ on $(H_k,)\, H_{k+1}, H_{k+2}, \ldots, H_{M-1}$. 
So 
\begin{equation} \label{trdbas}
\trdeg_K K(H_1,\ldots,H_s,H_{s+2},\ldots,H_{M-1}) = \trdeg_K K(H) = M-2
\end{equation}
It follows that $H_1,\ldots,H_s,H_{s+2},\ldots,H_{M-1}$ is a transcendence basis 
of $K(H)$ over $K$. So $h$ and $H$ are \HE-paired. Furthermore, 
$\trdeg_K K(H_{s+1}, H_{s+2}, \allowbreak \ldots, \allowbreak H_k) \ge k - s - 1$, because 
$k \le M-1$. From lemma \ref{upia}, it follows that $\trdeg_K K(H_{s+1}, \allowbreak
H_{s+2}, \ldots, H_k) = k - s - 1$.

Since $H_M$ and $H_{s+1}$ are algebraically dependent over $K$ on 
$H_{s+2}, H_{s+3}, \ldots, \allowbreak H_{M-1}$, it follows from \eqref{trdbas} that
for every $i \le s$, $H_i$ is algebraically independent over $K$ of 
$H_1, H_2, \ldots,H_{i-1},H_{i+1},H_{i+2},\ldots,H_M$. On account of lemma \ref{upia},
$e_i$ is a projective image apex of $H$ for each $i \le s$. It follows from
(iv) of theorem \ref{piacalc} that every nontrivial linear combination of
$e_1, e_2, \ldots, e_s$ is a projective image apex of $H$.

So we have proved that the projective image apices of $H$ are exactly those
nonzero vectors which are generated by $e_1, e_2, \ldots, e_s$.
Since $h$ and $H$ are \HE-paired, we deduce from (ii) of proposition \ref{inherprop}
that every nontrivial linear combination of $e_1, e_2, \ldots, e_s$ is a projective 
image apex of $h$.

Let $m = M-1$ and suppose that $p$ is a projective image apex of $h$. 
In order to prove that for $h$, the projective image apices are exactly those
nonzero vectors which are generated by $e_1, e_2, \ldots, e_s$ it suffices to show
that $p_{s+1} = p_{s+2} = \cdots = p_m = 0$. The last claim of this theorem
will then be clear as well, so it remains to show that 
$p_{s+1} = p_{s+2} = \cdots = p_m = 0$.

Since $\trdeg_K K(H_{s+1}, H_{s+2}, \ldots, H_k) = k - s - 1$, it follows from 
lemma \ref{pquasi} that there exist $F_{s+1}, F_{s+2}, \ldots, F_k \in K[x]$,
such that $F\tp \cdot \jac H = 0$, where
$$
F = (0^1,0^2,\ldots,0^s,F_{s+1},F_{s+2},\ldots,F_k,0^{k+1},0^{k+2},\ldots,0^M)
$$
As $(H_{s+1}, H_{s+2}, \ldots, H_k)$ does not have a projective image apex, we can
use (4) $\Rightarrow$ (1) of lemma \ref{pquasi} to deduce that 
$F_{s+1}, F_{s+2}, \ldots, F_k \in K[x]$ are linearly independent over $K$,

Let $N = 2n$ and write $\tilde{x} = x_{n+1},x_{n+2},\ldots,x_{2n}$.
\begin{enumerate}[\upshape (i)]

\item Define 
$$
G := \big(H_1(x),H_2(x),\ldots,H_k(x),
H_{k+1}(\tilde{x}),H_{k+2}(\tilde{x}),\ldots,H_M(\tilde{x})\big)
$$
Then $f(G) = 0 \Longrightarrow f(H) = 0$ for every $f \in K[Y]$. Since $G_{s+1}$ and
$G_M$ are algebraically dependent over $K$ on $G_{s+2}, G_{s+3}, \ldots, G_{M-1}$, 
we can deduce from $\trdeg_K K(H) = M - 2$ that $f(G) = 0 \Longleftrightarrow 
f(H) = 0$ for every $f \in K[Y]$.
Let 
$$
g = (G_1, G_2,\ldots,G_s,G_{s+1}+G_M,G_{s+2},G_{s+3},\ldots,G_m)
$$
Then $f(g) = 0 \Longleftrightarrow f(h) = 0$ for every $f \in K[y]$.

Just like $\trdeg_K K(H_{s+1}, H_{s+2}, \ldots, H_k) = k - s - 1$, we can show that
$\trdeg_K K(H_{k+1}, H_{k+2}, \ldots, H_M) = M - k - 1$. Furthermore, we can show 
that there are $\tilde{F}_{k+1}, \tilde{F}_{k+2}, \ldots, \tilde{F}_M \in K[x]$ which 
are linearly independent over $K$, such that $\tilde{F}\tp \cdot \jac H = 0$, where
$$
\tilde{F} = (0^1,0^2,\ldots,0^s,0^{s+1},0^{s+2},\ldots,0^k,
\tilde{F}_{k+1},\tilde{F}_{k+2},\ldots,\tilde{F}_M)
$$
Let $\tilde{f} \in K(x)^m$ be the vector consisting of the first $m = M-1$ coordinates of 
$F_{s+1}(x)^{-1} F(x) + \tilde{F}_M(\tilde{x})^{-1} \tilde{F}(\tilde{x})$. 
As $\tilde{f}_{s+1} = 1$, it follows that
$$
\tilde{f}\tp \cdot \jac_X g = (\tilde{f},1)\tp \cdot \jac_X G
= \big(F_{s+1}^{-1} \cdot F\tp \cdot \jac H\big) + 
   \big(\tilde{F}_M^{-1} \cdot \tilde{F}\tp \cdot \jac H\big)\big|_{x=\tilde{x}} = 0 
$$
Hence it follows from (1) $\Rightarrow$ (4) of lemma \ref{pquasi} that 
$\tilde{f}\tp \cdot p = 0$. 

Now suppose that $p_i \ne 0$ for some $i \ge s+1$.
If $p_{s+1} \ne 0$, then there exists an $i > s+1$ such that $p_i \ne 0$, because 
$\tilde{f}_{s+1} = 1 \ne 0$. Hence we can take $i \ge s+2$. We may assume 
without loss of generality that $s+2 \le i \le k$, because the case $k+1 \le i \le M-1$
is similar.

Take $\theta \in K^n$ such that $\tilde{F}_M(\theta) \ne 0$. 
Then each component of $\tilde{F}_M(\theta)^{-1} \tilde{F}(\theta)$ 
is linearly dependent over $K$ on $\tilde{f}_{s+1} = 1$. From 
$\tilde{f}|_{\tilde{x}=\theta}\tp \cdot p = 0$, it follows that
$F_{s+1}^{-1} F_i$ is linearly dependent over $K$ on the other
nonzero components of $F_{s+1}^{-1} F$. Contradiction, so 
$p_{s+1} = p_{s+2} = \cdots = p_m = 0$.

\item Let $a$ be an image apex of $H$. Suppose first that $k = s + 1$.
Then $(H_k)$ does not have a projective image apex, so $H_k \in K$ on account of lemma 
\ref{upia}. Furthermore, 
$(H_k, H_{k+1}, \ldots, H_M)$ does not have a projective image apex. So it follows 
from (ii) of proposition \ref{apexmap}, with $c = (H_k,0^{k+1},0^{k+2},\ldots,0^M)$,
that $H_{k+1}, H_{k+2}, \ldots, H_M$ does not have a projective image apex either. 
Hence (i) is satisfied and the claims follow.

Suppose next that $k > s + 1$. From (i) of lemma \ref{apexmap}, it follows that
the zero vector is an image apex of $H-a$. Define
\begin{multline*}
G := \Big(\big(H_1(x)-a_1,H_2(x)-a_2,\ldots,H_k(x)-a_k\big)\cdot\big(H_k(\tilde{x})-a_k\big), \\
          \big(H_k(x)-a_k\big) \cdot \big(H_{k+1}(\tilde{x})-a_{k+1},\ldots,H_M(x)-a_M\big)\Big)
\end{multline*}
We show that 
\begin{equation} \label{fga}
f(G) = 0 \Longrightarrow f(H-a) = 0
\end{equation}
So assume that $f(G) = 0$. Let 
$$
\gamma := \big((y_1,y_2,\ldots,y_k)\cdot y_k,y_k\cdot(y_{k+1},y_{k+2},\ldots,y_M)\big)
$$
Then $f\big(\gamma(H-a)\big) = 0$ follows by substituting $\tilde{x} = x$ in $f(G) = 0$. 
Since $a$ in an image apex of $H$, it follows that $f\big((1-t)^2 \gamma(H-a)\big) = 
f\big(\gamma\big((1-t) \bcdot (H-a)\big)\big) = 0$,
which gives $f(H-a) = 0$ after substituting $t$ by $1 - (\sqrt{H_k - a_k})^{-1}$.

If $G_{s+1}$ and $G_M$ are algebraically dependent over $K$ on $G_{s+2},
\ldots, G_{M-1}$, then we can deduce from $\trdeg_K K(H) = M - 2$ that 
$f(G) = 0 \Longleftrightarrow f(H) = 0$ for every $f \in K[Y]$. So let us show
for instance that $G_{s+1}$ is algebraically dependent over $K$ on $G_{s+2},
\ldots, G_{M-1}$. 

Since $H_{s+1}$ is algebraically dependent over $K$ on 
$H_{s+2}, \ldots, H_k$, there exists a nonzero $f \in K[y_{s+1},y_{s+2},\ldots,y_k]$
such that $f(H-a) = 0$. Since the zero vector is an image apex of $H-a$, we can 
subsitute $t = 1 - \big(H_k(\tilde{x}) - a_k\big)$ to obtain $f(G) = 0$,
because $f \in K[y_{s+1},y_{s+2},\ldots,y_k]$.

Since $H_{s+2}-a_{s+2},H_{s+3}-a_{s+3},\ldots,H_{M-1}-a_{M-1}$ are algebraically 
independent over $K$, it follows from \eqref{fga} that 
$G_{s+2},G_{s+3},\ldots,G_{M-1}$ are algebraically independent over $K$. 
As $f \in K[y_{s+1},y_{s+2},\ldots,y_{M-1}]$, we deduce that $G_{s+1}$ is algebraically 
dependent over $K$ on $G_{s+2}, G_{s+3}, \ldots, G_{M-1}$ indeed.

So $f(G) = 0 \Longleftrightarrow f(H-a) = 0$ for every $f \in K[Y]$. 
Let 
$$
g = (G_1, G_2,\ldots,G_s,G_{s+1}+G_M,G_{s+2},G_{s+3},\ldots,G_m)
$$
and $\tilde{a} = (a_1, a_2,\ldots, a_s, a_{s+1} + a_M, a_{s+2}, a_{s+3}, \ldots, a_{M-1})$.
Then $f(g) = 0 \Longleftrightarrow f(h-\tilde{a}) = 0$ for every $f \in K[y]$.
The rest of the proof is similar to that of (i). \qedhere
\end{enumerate}
\end{proof}

\begin{proposition} \label{HEredb}
Let $K$ be a field of characteristic zero and $H \in K[X]^M$.
Suppose that $\rk \jac H < m \le M$ and let $s$ be the number of independent 
projective image apices of $H$. Then there exists an $h \in K[X]^m$, 
such that $h$ and $H$ are \HE-paired, in such a way that every projective 
image apex of $h$ is inherited, in the following cases:
\begin{enumerate}[\upshape (i)]

\item $\rk \jac H \le s + 2$;

\item $H$ has an image apex and $\rk \jac H \le s + 3$.
      
\end{enumerate}
\end{proposition}

\begin{proof}
Assume without loss of generality that $p$ is a projective image apex of $H$, if and only
if $p$ is a nontrivial linear combination of $e_1, e_2, \ldots, e_s$. By induction, we may
assume that $m = M - 1$. Since $\rk \jac H < m = M - 1$, it follows that
$M \ge M' := \rk \jac H + 2$. 

From (iii) of proposition \ref{nopia},
it follows that we may assume without loss of generality that $H_1, H_2, \ldots, H_{M'-2}$ 
are a transcendence basis of $K(H)$ over $K$. Let $s'$ be the number of independent 
projective apices of $H' = (H_1,H_2,\ldots,\allowbreak H_{M'})$. Since $s' \ge s$, we can 
deduce that $H'$ satisfies (i) or (ii) as well as $H$. Hence theorem \ref{cominherred} allows
us to assume that $M = M' = \rk \jac H + 2$. Having been useful, we abandon the assumption 
that $H_1, H_2, \ldots, H_{M'-2}$ is a transcendence basis of $K(H)$ over $K$ from now on. 

From (ii) of proposition \ref{nopia}, it follows that $(H_{s+1},H_{s+2},\ldots,H_M)$ does not 
have a projective image apex. So for $k = M$, there exists a $B \in \Mat_{k-s,M}(K)$, 
such that $\rk B = k - s$ and $(B_1H, B_2H, \ldots,B_{k-s}H)$ does not have a projective 
image apex. So we can take $k$ as small as possible, such that there exists a 
$B \in \Mat_{k-s,M}(K)$, such that $\rk B = k - s$ and $(B_1H, B_2H, \ldots, B_{k-s}H)$ 
does not have a projective image apex. From (ii) of proposition \ref{apexmap}, it follows that 
the first $s$ columns of $B$ are zero. Hence there exists a $T \in \GL_M(K)$ of the form
$$
T = \left(\begin{array}{cccc}
I_s & \emptyset \\
\emptyset & *
\end{array}\right)
$$
such that $B$ is the submatrix of $T$ which consists of the rows $s+1, s+2, \ldots, k$.
Now replace $H$ by $TH$. Then $(H_{s+1},H_{s+2},\ldots,H_k)$ 
does not have a projective image apex.

Since $M = \rk \jac H + 2$, it follows from proposition 1.2.9 of either \cite{MR1790619} or \cite{homokema} that $\trdeg_K K(H) = M - 2$. On account of (iii) of proposition \ref{nopia},
$$
\trdeg_K K(H_{s+2},H_{s+3},\ldots,H_M) \le \trdeg_K K(H) - s < M - (s+1)
$$
Again from (iii) of proposition \ref{nopia}, we deduce that there exists a nonzero vector
in $K^{M-(s+1)}$ which is not a projective image apex of $(H_{s+2},H_{s+3},\ldots,H_M)$.

From (iv) of theorem \ref{piacalc}, it follows that there exist a matrix $Q$ in column-echelon form, of which the column space is the space of projective image apices of $(H_{s+2},H_{s+3},\ldots,H_M)$. Take $\tilde{Q} \in \GL_{M-(s+1)}$ lower triangular, such that $Q$ is a submatrix of $\tilde{Q}$, and let 
$$
\tilde{T} := \left( \begin{array}{cc} I_{s+1} & \emptyset \\
\emptyset & \tilde{Q}^{-1} \end{array} \right)
$$
From (ii) of proposition \ref{apexmap}, it follows that the space of projective image apices of $\big((\tilde{T}H)_{s+2},(\tilde{T}H)_{s+3},\ldots,(\tilde{T}H)_M\big)$ is generated by standard basis unit vectors. Let $k' - (s+1)$ be the number of generating standard basis unit vectors. Now replace $H$ by $\tilde{T}H$.
Since $\tilde{T}$ is lower triangular, the property that $(H_{s+1},H_{s+2},\ldots,H_k)$ 
does not have a projective image apex is preserved. 

There are $\big(M - (s+1)\big) - \big(k' - (s+1)\big) = M - k'$ standard basis 
unit vectors which are \emph{not} projective image apices of $(H_{s+2},H_{s+3},\ldots,H_M\big)$. 
If we subsequently replace $H$ by $PH$ for a suitable permutation matrix $P$, 
then we can obtain that the above $M - k'$ standard basis unit vectors will become 
successive, in such a way that the property that $(H_{s+1},H_{s+2},\ldots,H_k)$ 
does not have a projective image apex is preserved. From (ii) of proposition 
\ref{nopia}, it follows that  $(H_{k'+1-j},H_{k'+2-j},\ldots,H_{M-j})$ 
does not have a projective image apex for some $j$ such that $0 \le j \le k' - (s+1)$.

We show that $j = 0$. Notice that $s+1 < k'+1-j$. If $M-j \le k$, then $(H_{k'+1-j},H_{k'+2-j},\ldots,H_{M-j})$ has less 
components than $(H_{s+1},H_{s+2},\ldots,H_k)$, which contradicts the minimality of $k$.
So $M-j > k$. It follows from lemma \ref{upia} that 
$H_{s+1}$ and $H_{M-j}$ are algebraically dependent over $K$ on $H_{s+2}, H_{s+3}, \ldots,\allowbreak
H_{M-j-1}$. But $\trdeg_K K(H) = M - 2$, so for each $i > M-j$, $H_i$ is 
algebraically independent over $K$ of $H_1,H_2,\ldots,H_{i-1},H_{i+1},H_{i+2},\ldots,H_M$.
On account of lemma \ref{upia}, we have $j = 0$ indeed.

If $k' = s+1$, then we can take $h = (H_1, H_2, \ldots, H_s, H_{s+2}, H_{s+3}, \ldots, H_M)$.
So assume that $k' \ge s+2$.
\begin{enumerate}[\upshape (i)]
 
\item Assume that $\rk \jac H \le s + 2$.
Since $M - 2 = \rk \jac H \le s + 2 \le k'$, we have $M \le k' + 2$. By minimality of $k$,
$k - s \le M - k' = 2$. So $k \le s+2 \le k'$. It follows from (i) of theorem \ref{Mred}
that we can take $h$ as in theorem \ref{Mred}.

\item Assume that $H$ has an image apex and $\rk \jac H \le s + 3$.
Since $M - 2 = \rk \jac H \le s + 3 \le k' + 1$, we have $M \le k' + 3$. By minimality of $k$,
$k - s \le M - k' = 3$. So $k \le s+3 \le k'+1$. It follows from (ii) of theorem \ref{Mred}
that we can take $h$ as in theorem \ref{Mred}. \qedhere

\end{enumerate}
\end{proof}

\begin{theorem} \label{HEredc}
Let $K$ be a field of characteristic zero and $H \in K[X]^N$.
Suppose that $\rk \jac H < n \le N$ and let $s$ be the number of independent 
projective apices of $H$. Then there exists an $h \in K[x]^n$ such that
$h$ and $H$ are symmetrically \HE-paired, in such a way that every projective 
image apex of $h$ is inherited, in the following cases:
\begin{enumerate}[\upshape (i)]

\item $\rk \jac H \le s + 2$;

\item $H$ has an image apex and $\rk \jac H \le s + 3$;
      
\item $\jac H$ is symmetric and $\rk \jac H \le 3$;

\item $H$ has an image apex, $\jac H$ is symmetric and $\rk \jac H \le 4$.
      
\end{enumerate}
\end{theorem}

\begin{proof}
We will prove (iii) and (iv) in section \ref{iiiniv}. We will not need them before
section \ref{iiiniv}.

The proof of (i) and (ii) is similar to that of (i) and (ii) of proposition 
\ref{HEredb} respectively, but there are some points of attention. 
First of all, $M = N$ and $m = n$ because the \HE-pairing must be symmetric.
So we may assume by induction that $n = N-1$. 

The reduction to the case $N = \rk \jac H + 2$ can be accomplished by means of theorem \ref{HEpairinga}, because $H'_{M'}$ may be any nontrivial linear combination of $H_{M'}, H_{M'+1}, \ldots, H_{M}$ (not just $H_{M'})$.
The same holds for the use of theorem \ref{Mred} in (i) and (ii) if $k' \ge s+2$, because it shows that by way of replacing $H$ by $DH(DX)$ for an invertible diagonal matrix $D$, we can obtain that $H|_{x_M = x_{s+1}}$ and $H$ are \HE-paired.

So the case $k' = s+1$ remains to be studied. 
If $\trdeg_K K(H|_{x_{s+1}=0}) = \trdeg_K K(H)$, then 
$h(x_1,x_2,\ldots,x_s,0,x_{s+1},x_{s+2},\ldots,x_n)$ and $H$ are symmetrically
\HE-paired, where $h$ is as in the proof of proposition \ref{HEredb}.

In the general case, it follows from theorem \ref{HEpairinga} that there
are $c_{s+1}, c_N \in K$ such that 
$$
H(x_1,x_2,\ldots,x_s,c_{s+1} x_N,x_{s+2},x_{s+3},\ldots,x_n,c_N x_N)
$$
and $H$ are \HE-paired. Now we replace $H$ by $C\tp H (Cx)$, where
$C \in \GL_N(K)$ such that 
$$
Cx = x_1,x_2,\ldots,x_s,x_{s+1} + c_{s+1} x_N,x_{s+2},x_{s+3},\ldots,x_n,c_N x_N
$$
after which $\trdeg_K K(H|_{x_{s+1}=0}) = \trdeg_K K(H)$. This replacement will
undo that $(H_{k'+1},H_{k'+2},\ldots,H_N)$ does not have a projective image apex,
but restoring this property as in the proof of proposition \ref{HEredb} will not 
affect $\trdeg_K \allowbreak K(H|_{x_{s+1}=0}) = \trdeg_K K(H)$. But $k' = s+1$ 
does not need to hold any more.
\end{proof}

\section{Hessians with small rank over integral domains} \label{sec:rtheorem}

We will apply lemma \ref{slem} below by taking for $L$ the fraction field of the integral domain at hand. The proof of lemma \ref{slem} is obtained from results about quasi-translations in other papers. These results are obtained by geometric techniques, so lemma \ref{slem} is the interface to geometry of this paper.

\begin{lemma} \label{slem}
Let $L$ be a field of characteristic zero and $h \in L[x]$. Let $r := \rk \hess h$ 
and let $s$ be the number of independent projective image apices of $\grad h$ over 
$L$.
\begin{enumerate}[\upshape (i)]

\item
Suppose that $\grad h$ does not have an image apex. Then the following holds:
\begin{compactitem}
\item $s \ne r$;
\item If $r \le 2$, then $r = 2$ and $s = 1$.
\end{compactitem}

\item
Suppose that $\grad h$ does have an image apex. Then the following holds:
\begin{compactitem}
\item $s \ne r - 1$;
\item If $s \ne r \le 3$, then $r = 3$ and $s = 1$.
\end{compactitem}

\item
Suppose that $h$ is homogeneous. Then the following holds:
\begin{compactitem}
\item $s \ne r - 1$;
\item If $s \ne r \le 4$, then $r = 4$ and $s = 2$.
\end{compactitem}

\end{enumerate}
\end{lemma}

\begin{proof}
Let $H := \grad h$. From proposition 1.2.9 of either \cite{MR1790619} or \cite{homokema}, 
it follows that $r = \trdeg_L L(H)$. 

On account of Lefschetz's principle, we may assume that $L \subseteq \C$.
From theorem \ref{piaLK}, it follows that there are only $s$ independent projective image 
apices of $H$ over $\C$ as well.
\begin{enumerate}[\upshape (i)]

\item
From (1) $\Rightarrow$ (2) of corollary \ref{plan}, it follows that $s < r$, which
gives the first claim. To prove the second claim, suppose that $r \le 2$.
Since $0 \le s < r$, it follows that $1 \le r \le 2$ and that it suffices to show that 
$s \ge 1$.

Now (i) of theorem \ref{HEredc} allows us to assume that $n = r+1$. 
So $\trdeg_L \allowbreak L(H) = n - 1$. Define $F$ as in lemma \ref{pquasi}.
Since $2 \le n = r + 1 \le 3$, it follows from \cite[Prop.\@ 2.1]{MR3794325} and \cite[Cor.\@ 5.6 (i)]{MR3794325} that there exists a nonzero $p \in K^n$ such that 
$p_1 F_1 + p_2 F_2 + \cdots + p_n F_n = 0$. On account of (3) $\Rightarrow$ (1) of 
lemma \ref{pquasi}, $p$ is a projective image apex of $H$. So $s \ge 1$.

\item
From (2) $\Rightarrow$ (1) of corollary \ref{plan}, it follows that $s \ne r-1$, 
which is the first claim. To prove the second claim, suppose that $s \ne r \le 3$.
From (1) $\Rightarrow$ (2) of corollary \ref{plan}, it follows that $s \le r$,
so $s \le r - 2$. If $r \le 2$, 
then it follows in a similar manner as in the proof of (i) that $s \ge 1$, which
contradicts $s \le r - 2 \le 2 - 2 = 0$. So $r = 3$ and 
$s \le r-2 = 1$, and it suffices to show that $s \ge 1$.

Just as before, (i) of theorem \ref{HEredc} allows us to assume that $n = r+1$. So $n = 4$ and $\trdeg_L \allowbreak L(H) = 3$.
Take $g \in \C[y]$ of minimum degree, such that $g(H) = 0$. Suppose first that $g$ is homogeneous. From
\cite[Th.\@ 4.6]{MR3804055}, it follows that $g$ can be expressed as a 
polynomial in $3$ linear forms. There exist a $p \in L^4$
on which these $3$ linear forms vanish, and a straightforward computation
shows that $\jac_y g \cdot p = 0$. So $p$ is a projective
image apex of $H$ on account of (6) $\Rightarrow$ (1) of lemma \ref{pquasi}. 
Hence $s \ge 1$.

Suppose next that $g$ is not homogeneous. From proposition \ref{apexmap}, it follows that we may assume that the origin is an image apex of $H$. Hence $g\big((1-t)H\big) = 0$. So $g(tH) = 0$. Therefore, $g$ can be replaced by any nonzero homogeneous component of it. Since $g$ has minimum degree, it follows that $g$ is homogeneous already.

\item
If $\deg H < 1$, then $r = 0$ and it follows from (1) $\Rightarrow$ (2) of corollary \ref{plan}
that $s = r = 0$, which gives both claims. So assume from now on that $\deg H \ge 1$.
Since $H$ is homogeneous of positive degree, it follows that $f(H) = 0$ implies $f(tH) = 0$, 
which is equivalent to $f\big((1-t)H\big) = 0$. Hence the zero vector is an image apex of $H$,
and the first claim follows from (ii). 

To prove the second claim, suppose that $s \ne r \le 4$. Just as in the proof of (ii), 
we have $s \le r - 2$, so it suffices to show that $s \ge 2$.
If $2 \le r \le 3$, then $s \ge 2$ follows in a similar manner as $s \ge 1$
follows in the case $1 \le r \le 2$ in (i), except that (ii) instead of (i) of
\cite[Cor.\@ 5.6]{MR3794325} is used.
So assume from now on that $r = 4$.

If $n = r+1 = 5$, then $s \ge 2$ follows in a similar manner as $s \ge 1$ follows in the 
case $n = r+1 = 4$ in (ii), except that \cite[Th.\@ 4.5]{MR3804055} is used instead of 
\cite[Th.\@ 4.6]{MR3804055}. So it suffices to reduce to the case $n = r+1 = 5$. 

Just like for $M$ in the proof of proposition \ref{HEredb} and $N$ in 
the proof of (i) and (ii) of theorem \ref{HEredc}, theorem \ref{cominherred} allows us to 
assume that $n = r + 2 = 6$ prior to the reduction to dimension $r+1 = 5$.
Take $k$ and $k'$ as in the proofs of proposition \ref{HEredb} and (i) and (ii) of 
theorem \ref{HEredc}.
The case $k' = s+1$ follows in a similar manner as in the proofs of 
proposition \ref{HEredb} and (i) and (ii) of theorem \ref{HEredc}.

So assume that $k' \ge s+2$. The case $k \le k' + 1$ follows in a similar manner
as in the proofs of (ii) of proposition \ref{HEredb} and (ii) of theorem \ref{HEredc}.
So assume that $k \ge k' + 2 \ge s + 4$. We will derive a contradiction with the minimality 
of $k$.

From theorem \ref{HEpairings}, it follows that there exists an 
$\tilde{h} \in K[x_1,x_2,x_3,x_4,x_5]$ such that 
$$
\tilde{H} := \Big(\parder{}{x_1}\tilde{h},\parder{}{x_2}\tilde{h},
\parder{}{x_3}\tilde{h},\parder{}{x_4}\tilde{h},\parder{}{x_5}\tilde{h}\Big)
$$
and $H$ are \HE-paired through matrices $C\tp$ and $C$, such that $C$ has a left inverse.
As we have seen above, $\tilde{H}$ has at least $2$ independent
projective image apices. Assume without loss of generality that $e_1$ and $e_2$ are
projective image apices of $\tilde{H}$. From (iii) of proposition \ref{nopia},
it follows that
$$
\trdeg_L L(\tilde{H}_3,\tilde{H}_4,\tilde{H}_5) = \trdeg_K K(\tilde{H}) - 2 = r - 2 = 2
$$
So $\tilde{H}_3,\tilde{H}_4,\tilde{H}_5$ are algebraically dependent over $L$.
From proposition \ref{HEphi}, it follows that $(C\tp)_3 H, (C\tp)_4 H, (C\tp)_5 H$
are algebraically dependent over $L$ as well. Since the rows of $C\tp$ are independent 
over $K$, we can deduce that $k - s \le 3$. This contradicts $k \ge s+4$. \qedhere
\end{enumerate}
\end{proof}

\begin{theorem} \label{rtheorem}
Let $R$ be an integral domain of characteristic zero, with fraction field $L$.
Assume that $H \in R[x_1,x_2,\ldots,x_m]^m$, such that $\jac H$ is symmetric. 
On account of Poincar{\'e}'s lemma or {\upshape\cite[Lem.\@ 1.3.53]{MR1790619}}, 
there exists an $h \in R[x_1,x_2,\ldots,x_m]$ such that $H_i = \parder{}{x_i} h$ 
for each $i$.

Write $r := \rk \jac H$.
\begin{enumerate}[\upshape (i)]

\item If $r = 0$, then $h$ has an image apex over $L$ which is contained in 
$R^m$, say $b$, and $h$ is of the form
$$
h = g + b_1 x_1 + b_2 x_2 + \cdots + b_m x_m
$$
where $g \in R$. 

Furthermore, $\{b\}$ is both the Zariski closure of the image of $H$ and 
the set of image apices over $L$ of $H$, and
$$
f(H) = 0 \Longleftrightarrow f\big((1-t)H + tb\big) = 0
\Longleftrightarrow f(b) = 0
$$
for every $f \in L[y]$.

\item If $r \le 1$ and $R$ is a $\gcd$-domain, then $h$ has an image apex 
over $L$ which is contained in $R^m$, say $b$, and $h$ is of the form
$$
h = g(p_1 x_1 + p_2 x_2 + \cdots + p_m x_m) + b_1 x_1 + b_2 x_2 + \cdots + b_m x_m
$$
where $g \in R[t]$ and $p \in R^m$, such that $r = \rk\,(p)$.

Furthermore, $\{b + \lambda p \,|\, \lambda \in L\}$ is both the Zariski 
closure of the image of $H$ and the set of image apices of over $L$ of $H$, and
$$
f(H) = 0 \Longleftrightarrow f\big(tp + (1-u)H + ub\big) = 0
\Longleftrightarrow f(tp + b) = 0
$$
for every $f \in L[y]$.

\item Suppose that $b \in R^m$ is an image apex of $H$ over $L$.
If $r \le 2$ and $R$ is a B{\'e}zout domain, then $h$ is of the form
\begin{align*}
h = g(p_1 x_1 + p_2 x_2 + \cdots + p_m x_m, q_1 x_1 + q_2 x_2 + \cdots + q_m x_m) + {} \\
b_1 x_1 + b_2 x_2 + \cdots + b_m x_m 
\end{align*}
where $g \in R[t,u]$, $p \in R^m$ and $q \in R^m$, such that $r = \rk\,(p\,|\,q)$.

Furthermore, $\{b + \lambda p + \mu q \,|\, \lambda, \mu \in L\}$ is both the Zariski 
closure of the image of $H$ and the set of image apices over $L$ of $H$, and
$$
f(H) = 0 \Longleftrightarrow f\big(tp + uq + (1-v)H + vb\big) = 0
\Longleftrightarrow f(tp + uq + b) = 0
$$
for every $f \in L[y]$, where $v$ is a variable.

\end{enumerate}
\end{theorem}

\begin{proof}
Let $s$ be the number of independent projective image apices of $H$ over $L$. 
From (i) of lemma \ref{slem}, it follows that $H$ has an image apex over $L$ 
if $r \le 1$. So we may assume that $H$ has an image apex over $L$. 
From (ii) of lemma \ref{slem}, if follows that $s = r$ if $r \le 2$ and 
$H$ has an image apex over $L$. So we may assume that $s = r$ as well.
In particular, (1) of corollary \ref{plan} is satisfied.

On account of corollary \ref{plan}, $H(0)$ is an image apex of $H$
over $L$, which is contained in $R^m$. So $H$ has an image apex $b$
over $L$, which is contained in $R^m$. From (i) of proposition \ref{apexmap}, 
it follows that the origin is an image apex of $H - b$ over $L$. So if the 
origin is not an image apex of $H$ over $L$, then we can alter that by replacing 
$H$ by $H - b$. So we may assume that the origin is an image apex of $H$ over $L$.

We assume from now on that $R = L$. The general case reduces to this 
case by way of lemma \ref{gcddomain} below, because we can take $b = 0$
on account of the assumption that the origin is an image apex of $H$ over $L$.

There exists a $Q \in \GL_m(L)$ such that the first $s$ columns 
$Q e_1, Q e_2, \ldots, Q e_s$ of $Q$ are projective image apices of $H$. 
Let $p$ be the first column of $Q$ if $r = s \ge 1$ and let $q$ be the 
second column of $Q$ if $r = s \ge 2$. We can write $h$ in the form 
$h = \tilde{h}(Q\tp x)$, so $H = Q \tilde{H}(Q\tp x)$, where 
$\tilde{H} = \grad \tilde{h}$. 

Since $\tilde{H}$ and $H$ are \HE-paired through $Q^{-1}$ and $(Q^{-1})\tp$,
it follows from (ii) of proposition \ref{inherprop} that $e_1, e_2, \ldots, 
e_s$ are projective image apices of $\tilde{H}$. From (iii) of proposition
\ref{nopia}, it follows that $\trdeg_L L(\tilde{H}_{s+1},\tilde{H}_{s+2},
\ldots,\tilde{H}_m) = 0$, so $\tilde{H}_i \in L$ for all $i > s$.

Suppose that $\tilde{H}_i \ne 0$ for some $i > s$. Then $t$ is algebraically 
dependent on the $i$-th component of $(1-t)\tilde{H}$. From (i) of proposition
\ref{algdepprop}, it follows that the origin is not an image apex of $\tilde{H}$,
which contradicts (i) of proposition \ref{inherprop}. So $\tilde{H}_i = 0$
for all $i > s$. Hence $\tilde{h} \in L[x_1,x_2,\ldots,x_s]$ and
$$
h \in L[(Q\tp x)_1, (Q\tp x)_2, \ldots, (Q\tp x)_s] 
    = L[x\tp Q e_1, x\tp Q e_2, \ldots, x\tp Q e_s]
$$
So $h$ is as given.

Let $W$ be the Zariski closure of the image of $H$. From (1) $\Rightarrow$ (3) of
corollary \ref{plan}, it follows that $W$ is an affine plane of dimension $s$.
Using (iii) and (iv) of corollary \ref{tucor}, we deduce that 
\begin{equation} \label{zareq}
f(H) = 0 \Longrightarrow f\big((1-t)H + y_1 Q e_1 +  y_2 Q e_2 + \cdots + y_s Q e_s\big) = 0
\end{equation}
for every $f \in [y]$.

Take $c \in L^s$ arbitrary. If we substitute $t = 1$ and $y_i = c_i$ for each $i \le s$ in 
\eqref{zareq}, then we see that $W$ contains the linear span of $Q e_1, Q e_2, \ldots, Q e_s$. 
Comparing dimensions, we deduce that $W$ is equal to this linear span. 
If we substitute $y_i = t c_i$ for each $i \le s$ in \eqref{zareq}, then we see 
that every point of $W$ is an image apex of $H$. The converse follows by substituting
$t = 1$ in the definition of image apex.
\end{proof}

\begin{lemma} \label{gcddomain}
Let $R$ be an integral domain  with fraction field $L$, and $h \in R[x]$. 
Suppose that $p, q \in L^n$.
\begin{enumerate}[\upshape (i)]
 
\item If $h \in L$, then $h \in R$. 
 
\item If $h \in L[p\tp x]$ and $R$ is a $\gcd$-domain, then 
$h \in R[\tilde{p}\tp x]$ for some $\tilde{p} \in R^n$ which is dependent 
over $L$ on $p$.

\item If $h \in L[p\tp x,q\tp x]$ and $p,q \in L \cdot S^n$ for a B{\'e}zout 
subdomain $S$ of $R$, then $h \in R[\tilde{p}\tp x,\tilde{q}\tp x]$ for some 
$\tilde{p},\tilde{q} \in S^n$ which are dependent over $L$ on $p$ and $q$.

\item If $h \in L[p\tp x,q\tp x]$, $R$ is a $\gcd$-domain and there exists an 
$i$ such that $\tilde{p} := p_i^{-1}p \in R^n$, then 
$h \in R[\tilde{p}\tp x,\tilde{q}\tp x]$ and $\tilde{q}_i = 0$ for some 
$\tilde{q} \in K^n$ which is dependent over $L$ on $p$ and $q$.

\end{enumerate}
\end{lemma}

\begin{listproof}
\begin{enumerate}[\upshape (i)]

\item Suppose that $h \in L$. Then $h \in R[x] \cap L = R$.

\item Suppose that $h \in L[p\tp x]$ and that $R$ is a $\gcd$-domain. 
Then there exists a $\tilde{p} \in R^n$ which is dependent over $L$ on $p$, 
such that $\gcd\{\tilde{p}_1,\tilde{p}_2,\ldots,\tilde{p}_n\} = 1$. 
Since every homogeneous part of $h$ is contained in $R[x]$ and an 
$L$-multiple of a power of $\tilde{p}\tp x$, it follows from 
Gauss's lemma that $h \in R[\tilde{p}\tp x]$.

\item Suppose that $h \in L[p\tp x,q\tp x]$ and that 
$p,q \in L\cdot S^n$ for a B{\'e}zout subdomain $S$ of $R$. Then the 
column space of $P := (p\,|\,q)$ is generated by column vectors over $S$.

From lemma \ref{weaksmith} below, it follows that $P = Q\cdot B$ for some 
$B \in \Mat_2(L)$ and a $Q \in \Mat_{n,2}(R)$ which is left invertible. 
If we define $(\tilde{p}\,|\,\tilde{q}) := Q$, then 
$h \in L[\tilde{p}\tp x, \tilde{q}\tp x]$, say that 
$h = f(\tilde{p}\tp x,\tilde{q}\tp x)$, where $f \in L[t,u]$.

Let $C \in \Mat_{n,2}(R)$ be a left inverse of $Q$. Then 
$$
h(C\tp(t,u)) \in R[C\tp(t,u)] \subseteq R[t,u]
$$
Since 
$$
h\big(C\tp(t,u)\big) = f\big(\tilde{p}\tp \cdot C\tp(t,u), \tilde{q}\tp \cdot C\tp(t,u)\big) = 
f\big(Q\tp C\tp(t,u)\big) 
$$
and $Q\tp C\tp = (CQ)\tp = I_2\tp = I_2$, we conclude that $f \in R[t,u]$. 

\item Suppose that $h \in L[p\tp x, q\tp x]$, that $R$ is a $\gcd$-domain and 
that there exists an $i$ such that $\tilde{p} := p_i^{-1}p \in R^n$.
If $p$ and $q$ are dependent over $L$, then $q$ 
is redundant and the result follows from (i), because $\tilde{p}_i = 1$. So assume that 
$p$ and $q$ are independent over $L$. Let $\tilde{q}$ be the $L$-multiple of $q - q_i \tilde{p}$ 
from which the coefficients are relatively prime. Then $\tilde{q}_i = 0$ because 
$q_i - q_i \tilde{p}_i = 0$.

It is clear that $h \in L[\tilde{p}\tp x,\tilde{q}\tp x]$, say that 
$h = f(\tilde{p}\tp x,\tilde{q}\tp x)$, where $f \in L[y_1,y_2]$.
Then $f(x_i,\tilde{q}\tp x) = h|_{x_i = 2 x_i - \tilde{p}\tp x} \in R[x]$. Now apply the proof of (i) 
on the coefficient of $x_i^k$ of $f(x_i,\tilde{q}\tp x)$ for each $k$ to obtain that $f \in R[y_1,y_2]$.
So $h \in R[\tilde{p}\tp x,\tilde{q}\tp x]$. \qedhere

\end{enumerate}
\end{listproof}

\begin{lemma} \label{weaksmith}
Let $S$ be a B{\'e}zout domain which is contained in a field $L$, and 
$P \in \Mat_{m,n}(S)$.
If $r \ge \rk P > 0$, then there exists a $Q \in \Mat_{m,r}(S)$
which is left invertible over $S$, such that $P = QA$ for some 
$A \in \Mat_{r,n}(S)$.
\end{lemma}

\begin{proof}
Just like in the algorithm to compute the Smith normal form of a matrix over a principal ideal domain, we can multiply $P$ with an invertible matrix from the left, to obtain a matrix of which the first nonzero column is zero below the first row. If we recurse this on the submatrix without the first row, we see that we can multiply $P$ with an invertible matrix $\tilde{Q}$ from the left, to obtain a matrix $\tilde{A}$ in row-echelon form. All nonzero entries of $\tilde{A}$ are within the submatrix of the first $r$ rows of $\tilde{A}$, which we call $A$. For $Q$, we take the submatrix of the first $r$ columns of $\tilde{Q}$.
\end{proof}

\begin{corollary} \label{weaksmithcor}
Let $S = K[t]$, $L = K(t)$ and $P \in \Mat_{m,n}(S)$. 
If $r \ge \rk P > 0$, then there exists matrices $Q$ and $A$ as in lemma \ref{weaksmith},
such that in addition, the leading homogeneous parts of the columns of $Q$ are independent over $L$.
In other words, the matrix $\hat{Q} \in \Mat_{m,n}(K)$, which is defined by that 
$\hat{Q} e_i$ is the column of the coefficients of $t^{\deg (Q e_i)}$ of $Q_i$, 
has rank $r$.
\end{corollary}

\begin{proof}
Since $K[t]$ is a B{\'e}zout domain, we can find matrices $Q$ and $A$ as in lemma 
\ref{weaksmith}. Assume without loss of generality that 
$$
\deg (Q e_1) \le \deg (Q e_2) \le \cdots \le \deg (Q e_r)
$$
Suppose that the leading homogeneous parts of the columns of $Q$ are 
dependent over $L$. Then there exists an $i$ such that $\hat{Q} e_i$ is dependent 
over $K$ on $\hat{Q} e_1, \hat{Q} e_2, \ldots, \hat{Q} e_{i-1}$, say that 
$\hat{Q} e_i = c_1 \hat{Q} e_1 + c_2 \hat{Q} e_2 + \cdots + c_{i-1} \hat{Q}e_{i-1}$. 

Now replace $Q$ by
$$
Q - \sum_{j=1}^{i-1} c_j t^{\deg (Q e_i) - \deg (Q e_j)} \hat{Q} e_j
$$
Then $Q$ is replaced by $QU$ for an upper diagonal $U \in \Mat_r(S)$ with ones 
on the diagonal. So $U \in \GL_r(S)$ and we can replace $A$ by $U^{-1}A$ to 
preserve $P = QA$. Similarly, it follows that $Q$ remains left invertible over $S$. 

By induction on $\deg Q e_1 + \deg Q e_2 + \cdots + \deg Q e_r$, it follows that
we can get $A$ and $Q$ as claimed after repeating the above several times. \qedhere
\end{proof}

One can apply lemma \ref{weaksmith} another time on the transpose of the matrix $A$ from the first application, to obtain that $A$ is the product of a triangular matrix $B$ of size $r \times r$ and a right-invertible matrix. This yields a weak version of a result from 1861 by Henri John Stephen Smith in \cite{smith1861}. The result of Smith is essentially that $B$ is even a diagonal matrix, such that the diagonal entries of $B$ are ordered 
by divisibility. The Smith normal form of matrix $P$ is just the diagonal matrix $B$, filled up with zero entries on the right and below so it has the same dimensions as $P$.

But Smith assumed the ring $S$ to be noetherian, i.e.\@ a principal ideal domain. However, there exist B{\'e}zout domains for which Smith's result is valid, and no B{\'e}zout domains are known for which Smith's result is not valid, see \cite{MR0031470}.

\section{Hessians with small rank over polynomial rings} \label{sec:gradrel}

In (iii) of theorem \ref{rtheorem}, $R$ is assumed to be a B{\'e}zout domain, so $R$  cannot be a multivariate polynomial ring over a field. But we will allow $R$ to be any polynomial ring over a field of characteristic zero in our main theorem. To cover the case of $R$ being a polynomial ring in (v) of theorem \ref{gradrel} below, we make the following definition.

\begin{definition}
Let $K$ be a field and $R$ be an integral $K$-domain. Let $L$ be the fraction field 
of $R$. We call $R$ \emph{Bogal} if one of the following statements is satisfied.
\begin{itemize}

\item $R$ is a B{\'e}zout domain.

\item $R$ is a $\gcd$-domain and $L$ satisfies L{\"u}roth's theorem as an
extension field of $K$.

\end{itemize}
Here, Bogal stands for `B{\'e}zout or $\gcd$ and L{\"u}roth'.
\end{definition}

\begin{theorem}[Main theorem] \label{gradrel}
Let $K$ be a field of characteristic zero. Assume that $R$ is an integral 
$K$-domain. Let $L$ be the fraction field of $R$ and suppose that $K$
is algebraically closed in $L$.

Assume that $H \in R[x_1,x_2,\ldots,x_m]^m$, such that $\jac R$ is symmetric. 
On account of Poincar{\'e}'s lemma or {\upshape\cite[Lem.\@ 1.3.53]{MR1790619}}, 
there exists an $h \in R[x_1,x_2,\allowbreak \ldots,x_m]$ such that $H = \grad h$. 
Suppose that $H$ does \emph{not} have a projective image apex $p^{(1)} \in K^m$ over $K$.
\begin{enumerate}[\upshape (i)]

\item 
Suppose that $\trdeg_K K(H) = 0$. Then $H$ \emph{does} have an image apex 
$a \in K^m$ over $K$, and $h$ is of the form 
$$
h = g + a_1 x_1 + a_2 x_2 + \cdots + a_m x_m
$$
where $g \in R$. Furthermore, $H = a$.

\item 
Suppose that $\trdeg_K K(H) = 1$. Then $H$ does \emph{not} have an image
apex $a \in K^m$ over $K$, and $h$ is of the form
$$
h = g + b_1 x_1 + b_2 x_2 + \cdots + b_m x_m
$$
where $g \in R$ and $b \in R^m \setminus K^m$. Furthermore, $H = b$.

\item 
Suppose that $\trdeg_K K(H) = 2$ and that $H$ \emph{does} have an image apex 
$a \in K^m$ over $K$. If $R$ is a $\gcd$-domain, then $h$ is of the form
$$
h = g(p_1 x_1 + p_2 x_2 + \cdots + p_m x_m) + a_1 x_1 + a_2 x_2 + \cdots + a_m x_m  
$$
where $g \in R[t] \setminus R$ and $p \in R^m \setminus R \cdot K^m$, such that
$$
\trdeg_{K(\gamma)} K(\gamma)(tp) = 1
$$ 
for some $\gamma \in L$. Furthermore, $f(H) = 0 \Longleftrightarrow f(tp + a) = 0$
for all $f \in K[y]$.

\item 
Suppose that $\trdeg_K K(H) = 2$ and that $H$ does \emph{not} have an image
apex $a \in K^m$ over $K$. If $R$ is a $\gcd$-domain, then $h$ is of the form
$$
h = g(p_1 x_1 + p_2 x_2 + \cdots + p_m x_m) + b_1 x_1 + b_2 x_2 + \cdots + b_m x_m
$$
where $g \in R[t]$ and $p, b \in R^m$, such that $\deg b \ge 1$. In addition, 
$$
\trdeg_{K(\gamma)} K(\gamma)(tp) + \trdeg_{K(\gamma)} K(\gamma)(b - \lambda p) = 1
$$ 
for some $\gamma, \lambda \in L$.
Furthermore, $f(H) = 0 \Longleftrightarrow f(tp+b) = 0$ for all $f \in K[y]$.

If it is possible to take $p \in K^m$, then we can take $g \in R$ and
$p = 0$. In other words, $h$ is as in {\upshape (ii)} in that case, except
that $\trdeg_K K(b) = 2$ instead of $\trdeg_K K(b) = 1$.

\item
Suppose that $\trdeg_K K(H) = 3$ and that $H$ \emph{does} have an image apex 
$a \in K^m$ over $K$. If $R$ is a Bogal domain, then $h$ is of the form
\begin{align*}
h = g(p_1 x_1 + p_2 x_2 + \cdots + p_m x_m, 
q_1 x_1 + q_2 x_2 + \cdots + q_m x_m) + {} \\
a_1 x_1 + a_2 x_2 + \cdots + a_m x_m 
\end{align*}
where $g \in R[t,u]$ and $p,q \in R^m$ for each $i$. In addition, if $R = L$
or $L$ satisfies L{\"u}roth's theorem as an extension field of $K$, then
$$
\trdeg_{K(\gamma)} K(\gamma)(tp) + \trdeg_{K(\gamma)} K(\gamma)(uq) = 2
$$ 
for some $\gamma \in L$. Furthermore, $f(H) = 0 \Longleftrightarrow 
f(tp+uq+a) = 0$ for all $f \in K[y]$.

If it is possible to arrange that some nontrivial $L$-linear
combination of $p$ and $q$ is contained in $K^m$,
then we can take $g \in R[t] \setminus R$ and $q = 0$. 
In other words, $H$ is as in {\upshape (iii)}
in that case, except that $\trdeg_{K(\gamma)} K(\gamma)(tp) = 2$ instead
of $\trdeg_{K(\gamma)} K(\gamma)(tp) = 1$.

\end{enumerate}
If $R$ is a polynomial ring over $K$, then $\deg b \ge 2$ in {\upshape (ii)} and 
in the case $p = 0$ of (iv), and $\deg p \ge 2$ in {\upshape (iii)} and 
in the case $q = 0$ of (v).

Suppose that $R[x]$ is graded with a torsion-free group, such that $K$ has trivial grading. If the components of $H$ are either zero or graded homogeneous of fixed degree, then $a = 0$ is an image apex of $H$.

Suppose that the grading of $R[x]$ induces on $R$, and that $x_1,x_2,\ldots,x_m$ are graded homogeneous. If $h$ is graded homogeneous and $a = 0$ is an image apex of $H$, then $p_1 x_1 + p_2 x_2 + \cdots + p_m x_m$ can be taken graded homogeneous in (iii) and (v), and $q_1 x_1 + q_2 x_2 + \cdots + q_m x_m$ can be taken graded homogeneous in (v), under the condition that $L$ satisfies L{\"u}roth's theorem as an extension field of $K$ in (v).
\end{theorem}

\begin{corollary}[Main corollary] \label{gradrelcor}
Let $K$ be a field of characteristic zero and $h \in K[x]$. Let $r := \rk \hess h$ and let $s$ be the number of independent projective image apices of $\grad h$. Let $m := n - s$ and assume without loss of generality that $e_{m+1}, e_{m+2}, \ldots, e_n$ are $s$ independent projective image apices of $\grad h$.

Define $R := K[x_{m+1},x_{m+2},\ldots,x_n]$.
\begin{enumerate}[\upshape (i)]

\item
Suppose that $\grad h$ does not have an image apex. 
\begin{compactitem}
\item If $r \le 2$, then $r = 2$, $h$ is of the form of {\upshape (ii)} of 
theorem \ref{gradrel} and $s = 1$.
\item If $r = 3$ and $s \ge 1$, then $h$ is of the form of {\upshape (ii)} or 
{\upshape (iv)} of theorem \ref{gradrel} and $s = 2$ or $s = 1$ respectively.
\end{compactitem}

\item
Suppose that $\grad h$ does have an image apex. 
\begin{compactitem}
\item If $r \le 2$, then $h$ is of the form of {\upshape (i)} of theorem \ref{gradrel} 
and $s = r$.
\item If $r = 3$, then $h$ is of the form of {\upshape (i)} or {\upshape (iii)} of 
theorem \ref{gradrel} and $s = 3$ or $s = 1$ respectively.
\item If $r = 4$ and $s \ge 1$, then $h$ is of the form of {\upshape (i)}, 
{\upshape (iii)} or {\upshape (v)} of theorem \ref{gradrel} and $s = 4$, $s = 2$ or 
$s = 1$ respectively. 
\end{compactitem}

\item
Suppose that $h$ is homogeneous. If $r \ge 1$, then the zero vector is an image apex 
of $\grad h$.
\begin{compactitem}
\item If $1 \le r \le 3$, then $h$ is of the form of {\upshape (i)} of theorem 
\ref{gradrel} with $a_i = 0$ for all $i$ and $s = r$.
\item If $r = 4$, then $h$ is of the form of {\upshape (i)} or {\upshape (iii)} of 
theorem \ref{gradrel} with $a_i = 0$ for all $i$ and $s = 4$ or $s = 2$ respectively.
\end{compactitem}

\end{enumerate}
\end{corollary}

\begin{proof}
Notice that $R$ is both a $\gcd$-domain and a Bogal domain (but not necessarily a
B{\'e}zout domain).
Let $H = \big(\parder{}{x_1} h, \parder{}{x_2} h, \ldots, \parder{}{x_m} h\big)$.
From proposition 1.2.9 of either \cite{MR1790619}, it follows that 
$r = \trdeg_K K(\grad h)$.

From (ii) of proposition \ref{nopia}, it follows that $H$ does not have a projective 
image apex $p \in K^m$ over $K$.
From (i) of proposition \ref{nopia}, it follows that $\grad{h}$ has an image apex
$a \in K^n$ over $K$, if and only if $H$ has an image apex $a \in K^m$ over $K$.  
From (iii) of proposition \ref{nopia}, it follows that $\trdeg_K K(H) = r - s$. 
Hence the results follow from lemma \ref{slem} and theorem \ref{gradrel}
\end{proof}

\begin{proof}[Proof of {\upshape (i)} of theorem \ref{gradrel}.]
As $\trdeg_K K(H) = 0$ and $K$ is algebraically closed in $L$, it
follows that $H \in K^m$. So if we take $a := H$, then it is straightforward
to verify that (i) is satisfied.
\end{proof}

\begin{proof}[Proof of {\upshape (ii)} of theorem \ref{gradrel}.]
Notice that 
$$
0 \le \trdeg_L L(H) \le \trdeg_K K(H) = 1
$$
We first show that $\trdeg_L L(H) = 0$. So assume that $\trdeg_L L(H) = 1$.
On account of (i) and (ii) of lemma \ref{slem}, $H$ has a projective image apex $p \in L^m$ over $L$, in particular over $K$. From (i) of theorem \ref{xsc}, it follows that $\trdeg_L L(H) \le \trdeg_K K(H) - 1 = 0$, so $\trdeg_L L(H) = 0$. 

On account of (i) of theorem \ref{rtheorem}, 
$h$ is of the form
$$
h = g + b_1 x_1 + b_2 x_2 + \cdots + b_m x_m
$$
where $g \in R$ and $b \in R^m$ is any image apex of $H$ over $L$ which
is contained in $R^m$. Since $\trdeg_K K(H) = 1$, it follows that 
$b = H \notin K^m$.

So it remains to show the first claim of (ii).
Hence suppose that $H$ does have an image apex $a \in K^m$ over $K$.
Since $\trdeg_L L(H) > -1 = \trdeg_K K(H) - 2$, 
it follows from (iii) of theorem \ref{xsc} that $a$ is an image apex of $H$ over $L$. 
Hence we can take $b = a$. This contradicts $b = H \notin K^m$, so $H$ does not have 
an image apex $a \in K^m$ over $K$.
\end{proof}

\begin{lemma} \label{iilem}
Let $L$ be an extension field of a field $K$ of characteristic zero, 
such that $K$ is algebraically closed in $L$. 
Suppose that $P \in \Mat_{m,n}(L)$ and $b \in L^m$, such that $r := \rk P > 0$. 

Then there exists a $Q \in \Mat_{m,r}(L)$, an $A \in \Mat_{r,n}(L)$, 
and a $\lambda \in L^n$, such that 
$$
P = Q \cdot A
$$
and such that the following holds, where 
$$
a := b - P \cdot \lambda
$$
and $z = (z_1,z_2,\ldots,z_r)$.
\begin{enumerate}[\upshape (i)]

\item $\trdeg_K K(Q z + a) = \trdeg_K K(P y + b)$

\item $\trdeg_K K(a) \le \trdeg_K K(Q z + a) - r$ and 
$\trdeg_K K(Q e_i) \le \trdeg_K$ \allowbreak $K(Q z + a) - r$ for all $i$.

\item If $\trdeg_K K(P y + b) \le r$, then $\trdeg_K K(Q z + a) = r$
and 
$$
(Q\,|\,a) \in \Mat_{m,r+1}(K)
$$

\item If $\trdeg_K K(P y + b) = r+1$, then either 
$$
(Q\,|\,a) \in \Mat_{m,r+1}(\tilde{K})
$$
for an extension field $\tilde{K}$ of $K$ for which $\trdeg_K \tilde{K} = 1$
and $\tilde{K} \subseteq L$, or 
$$
(Q\,|\,a) - p^{(1)}\big(b^{(1)}\big)\tp \in \Mat_{m,r+1}(K)
$$
for a nonzero $p^{(1)} \in K^m$ and a $b^{(1)} \in L^r \setminus K^r$.

\end{enumerate}
\end{lemma}

\begin{proof}
Assume without loss of generality that the leading principal minor matrix
$M$ of size $r$ of $P$ has rank $r$. Take for $A$ the matrix consisting of 
the first $r$ rows of $P$. Take for $Q$ the product of the matrix 
consisting of the first $r$ columns of $P$ and $M^{-1}$. Then the leading
principal minor matrix of size $r$ of $Q$ is the identity matrix.
From $\rk M = \rk P$, it follows that the reduced column echelon form of $P$
is of the form $(Q\,|\,\emptyset)$. Since the first $r$ rows of $Q \cdot A$
are equal to those of $P$ in addition, we can deduce that $P = Q \cdot A$ 
as a whole. 

Define $\beta \in L^r$ by $\beta_i = b_i$ for all $i \le r$ and let
$\lambda = (M^{-1}\beta,0^{r+1},0^{r+2},\ldots,0^n)$. By definition of $Q$,
$$
a = b - P \lambda = b - Q \beta 
$$
Since the leading principal minor of $Q$ equals $I_r$, it follows that
$a_1 = a_2 = \cdots = a_r = 0$.
\begin{enumerate}[\upshape (i)]

\item
Notice that
$$
\trdeg_K K\big(P y + b\big) = \trdeg_K K\big(P (y-\lambda) + b\big) 
= \trdeg_K K\big(P y + a\big)
$$
Since $M$ is the matrix of the first $r$ columns of $A$, it follows
that $A$ is right invertible. Hence
$$
\trdeg_K K\big(Q z + a\big) = \trdeg_K K\big(Q (A y) + a\big) 
= \trdeg_K K\big(P y + a\big)
$$
So $\trdeg_K K\big(P y + b\big) = \trdeg_K K\big(Q z + a\big)$.

\item
Let $j \le r$. Take $\p := \big(f \in K[y]\,\big|\,f(Qz+a) = 0\big)$ and 
$\q := \big(f \in K[y]\,\big|\,f(Qz|_{z_j = 0}+a) = 0\big)$.
Then $\p \subseteq \q$. As $\p \not\ni y_j \in \q$, it follows that
$\p \subsetneq \q$. So
$$
\trdeg_K K\big(Qz|_{z_j=0}+a\big) \le \trdeg_K K(Qz+a) - 1 
$$
Since $Q_j = e_j\tp$ and $a_j = 0$ for all $j \le r$, the substitution 
$z_j = 0$ comes down to removing the $j$-th row and $j$-th column of 
$(Q\,|\,a)$. Hence it follows by induction on $r$ that 
\begin{align*}
\trdeg_K K(Q z|_{z=0}+a) &\le \trdeg_K K(Qz+a) - r 
\intertext{and that for all $i$,}
\trdeg_K K(Q z|_{z=z_ie_i}+a) &\le \trdeg_K K(Qz+a) - (r - 1)
\end{align*}
So $\trdeg_K K(a) \le \trdeg_K K(Q z + a) - r$ has been proved.

Hence it remains to show that $\trdeg_K K(Q e_i) \le \trdeg_K 
K(Qz + a) - r$. Let $s := \trdeg_K K(Qz+a)$ and suppose that 
$$
\trdeg_K K(Q e_i) > \trdeg_K K(Qz+a) - r = s - r
$$
Since the first $r$ coordinates of $Q e_i$ are contained in $K$, 
we may assume without loss of generality that $Q_{(r+1)i}, 
Q_{(r+2)i}, \ldots, Q_{si}, Q_{(s+1)i}$ are algebraically 
independent over $K$. 

Since $\trdeg_K K(Qe_i z_i+a) \le \trdeg_K K(Qz+a) - (r - 1) = s + 1 - r$, 
it follows that there exists a nonzero $f \in K[y_i,y_{r+1},y_{r+2},
\ldots,y_{s},y_{s+1}]$ such that $f(Q e_i z_i+a) = 0$. 
Looking at the coefficient of $z_i^{\deg f}$, we see that there exist a 
homogeneous polynomial $g$ over $K$ such that
$$
g\big(1,Q_{(r+1)i}, Q_{(r+2)i}, \ldots, Q_{si}, Q_{(s+1)i}\big) = 0
$$
This contradicts that $Q_{(r+1)i}, Q_{(r+2)i}, \ldots, Q_{si}, 
Q_{(s+1)i}$ are al\-ge\-brai\-cal\-ly independent over $K$. 

\item
Assume that $\trdeg_K K(Qz+a) \le r$. From (ii), it follows that 
$$
\trdeg_K K(a) \le \trdeg_K K(Qz+a) - r \le 0
$$ 
Consequently, $\trdeg_K K(a) = 0$ and $\trdeg_K K(Qz+a) = r$. 
Similarly, $\trdeg_K K(Q e_i) = 0$ for all $i$. 
Now (iii) follows, because $K$ is algebraically closed in $L$.
 
\item
The case where $Q \in \Mat_{m,r}(K)$ follows from (ii), so assume the opposite.
Assume without loss of generality that $Q_{(r+1)1} \notin K$. If $Q_{(r+2)1}$
is linearly dependent over $K$ on $Q_{(r+1)1}$ and $Q_{11}$, the we can clean
it by row operations over $K$ in $(Q\,|\,a)$. Since such row operations do
not affect the claim of (iv), we may assume that either $Q_{(r+2)1} = 0$ or 
$Q_{(r+2)1}$ is linearly independent over $K$ of $Q_{(r+1)1}$ and $Q_{11}$. 

Since $Q_{11} = 1$, $Q_{(r+1)1} \notin K$ and $K$ is algebraically closed in $L$,
it follows that $\trdeg_K K(Q_{11},Q_{(r+1)1}) = 1$. From (ii) with $m = r+1$
and $i = 1$, we deduce that the first $r+1$ components of $Q z + a$ are 
algebraically independent over $K$. 

Suppose that $\trdeg_K K(Qz+a) = r+1$. Then the first $r+1$ components of 
$Qz + a$ are a transcendence basis over $K$ of $K(Qz+a)$ and there exists 
a nonzero $f \in K[y_1,y_2,\ldots,y_r,y_{r+1},y_{r+2}]$, such that $f(Qz+a) = 0$.
Take $f$ as such of minimum degree and define
$$
g := \sum_{i=1}^m Q_{i1} \parder{}{y_i} f 
= \parder{}{y_1} f + Q_{(r+1)1} \parder{}{y_{r+1}} f + Q_{(r+2)1} \parder{}{y_{r+2}} f 
$$
Then
$$
g(Qz+a) = \parder{}{z_1} f(Qz+a) = 0
$$
We distinguish two cases.
\begin{itemize}

\item \emph{$g \ne 0$.} \\
Let $\tilde{K}$ be the algebraic closure of $K(Q e_1)$ in $L$. Then 
$g \in \tilde{K}[y]$ and $\tilde{K}$ is algebraically closed in $L$.
Furthermore, it follow from (ii) and (i) that
$$
\trdeg_K \tilde{K} \le \trdeg_K K(Q z + a) - r = \trdeg_K K(P z + b) - r \le 1
$$
If $\trdeg_{\tilde{K}} \tilde{K}(Qz+a) \le r$, then it follows from (iii) 
that $Q \in \Mat_{m,r}(\tilde{K})$ and $a \in \tilde{K}^m$, which 
gives (iv).

In order to show that $\trdeg_{\tilde{K}} \tilde{K}(Qz+a) \le r$,
suppose that $\trdeg_{\tilde{K}} \allowbreak \tilde{K}(Qz+a) \ge r+1$. 
Then 
$$
\p := \big(\tilde{g} \in \tilde{K}[y_1,y_2,\ldots,y_r,y_{r+1},y_{r+2}]\,\big|\,
      \tilde{g}(Qz+a) = 0\big)
$$
is principal. Furthermore, $f,g \in \p$. Since $K$ is algebraically closed in $\tilde{K}$, it follows from lemma \ref{ptLK} and the minimality of $\deg f$ that $\p = (f)$. This contradicts $\deg g < \deg f$, so $\trdeg_{\tilde{K}} \tilde{K}(Qz+a) \le r$ indeed. 

\item \emph{$g = 0$.} \\
Recall that either $Q_{(r+2)1} = 0$ or $Q_{(r+2)1}$ is linearly independent 
over $K$ of $Q_{(r+1)1}$ and $Q_{11}$. In both cases, it follows from lemma 
\ref{fLK} and the minimality of $\deg f$ that
$$
\parder{}{y_1} f = \parder{}{y_{r+1}} f = 0
$$
So $f \in K[y_2,y_3,\ldots,y_r,y_{r+2}]$. As $z_1,z_2,z_3,\ldots,z_r$ are
algebraically independent over $K$, we see that $(Q z+a)_{r+2}$ is algebraically
dependent over $K$ on $z_1,z_2,z_3,\ldots,z_r$. From (iii), it follows that
$Q_{r+2} \in \Mat_{1,r} (K)$ and $a_{r+2} \in K$.

We can interchange row $r+2$ of $Q$ with any other row after row $r+1$.
If we can do this in such a way that $g \ne 0$ after it, then (iv) follows from 
the case above. Otherwise, we can conclude that $Q_i \in \Mat_{1,r}(K)$ and 
$a_i \in K$ for all $i \ne r+1$. So if we take $p^{(1)} = e_{r+1}$ and 
$b^{(1)} = \big((Q\,|\,a)_{r+1}\big)\tp$, then  $b^{(1)} \notin K^r$ and
$(Q\,|\,a) - p^{(1)}\big(b^{(1)}\big)\tp \in \Mat_{n,r+1}(K)$, which gives (iv) as well.
\qedhere

\end{itemize}
\end{enumerate}
\end{proof}

\begin{proof}[Proof of {\upshape (iii)} of theorem \ref{gradrel}.]
Notice that
$$
0 \le \trdeg_L L(H) \le \trdeg_K K(H) = 2
$$
We first show that $\trdeg_L L(H) \le 1$. So assume that $\trdeg_L L(H) = 2$.
On account of (i) and (ii) of lemma \ref{slem}, $H$ has a projective image apex $p \in L^m$ over $L$, in particular over $K$. From (i) of theorem \ref{xsc}, it follows that $\trdeg_L L(H) \le \trdeg_K K(H) - 1 = 1$.

Suppose that $R$ is a $\gcd$-domain. On account of (ii) of theorem \ref{rtheorem}, $h$ is of the form
$$
h = g(p_1 x_1 + p_2 x_2 + \cdots + p_m x_m) + b_1 x_1 + b_2 x_2 + \cdots + b_m x_m
$$
where $g \in R[t]$ and $b \in R^m$ is any image apex of $H$ over $L$ which is
contained in $R^m$. Furthermore, $p \in R^m$ and 
$$
f(H) = 0 \Longleftrightarrow f\big(tp + (1-u)H + ub\big) = 0
\Longleftrightarrow f(tp + b) = 0
$$
for every $f \in L[y]$. So if $p \ne 0$, then $p$ is a projective image apex of $H$ 
over $L$. From (iv) of theorem \ref{piacalc}, it follows that $p \notin R \cdot K^m$
if $p \ne 0$.

If $a$ is an image apex of $H$ over $L$, then we can take $b = a$, so 
$$
f(H) = 0 \Longleftrightarrow f(tp+a) = 0
$$
for every $f \in K[y]$. Suppose next that $a$ is not an image apex of $H$ over $L$. From  (iii) 
of theorem \ref{xsc}, it follows that $\trdeg_L L(H) \le \trdeg_K K(H) - 2 = 0$.
On account of (i) of theorem \ref{rtheorem}, $h$ is of the form
$$
h = g + b_1 x_1 + b_2 x_2 + \cdots + b_m x_m
$$
where $g \in R$ and $b \in R^m$. So if we define $p = b - a$, then 
$$
h = g + p_1 x_1 + p_2 x_2 + \cdots + p_m x_m + a_1 x_1 + a_2 x_2 + \cdots + a_m x_m
$$
and $H = p + a$. Since $tH + (1-t)a = t(H-a) + a = tp + a$, it follows that 
\begin{align*}
f(H) = 0 &\Longrightarrow f\big((1-t)H + ta\big) = 0 \\ &\Longrightarrow 
f\big(tH + (1-t)a\big) = 0 \Longrightarrow f(tp + a) = 0
\end{align*}
for every $f \in K[y]$. The converse follows from $H = p + a$, so
$$
f(H) = 0 \Longleftrightarrow f(tp+a) = 0
$$
for every $f \in K[y]$, regardless of whether $a$ is an image apex of $H$ over $L$
or not.

Since $\trdeg_K K(a) = 0 < 1 = \trdeg_K K(H)$ we see that $p \ne 0$. From (ii) of 
lemma \ref{iilem}, it follows that there exists an $\alpha \in L$ such that
$\trdeg_K K(\alpha^{-1} p) \le \trdeg_K K(H) - 1 = 1$. So there exists a $\gamma \in L$
such that $\trdeg_{K(\gamma)} K(\gamma)(\alpha^{-1} p) = 0$. As $p \ne 0$, we
conclude that
$$
\trdeg_{K(\gamma)} K(\gamma)(tp) = 1
$$
which completes the proof of (iii) of theorem \ref{gradrel}.
\end{proof}

\begin{proof}[Proof of {\upshape (iv)} of theorem \ref{gradrel}.]
Just as in the proof of (iii) of theorem \ref{gradrel}, we can deduce that
$h$ is of the given form, and that $f(H) = 0 \Longleftrightarrow f(tp+b) = 0$
for every $f \in L[y]$. Furthermore, $p \notin R \cdot K^m$ if $p \ne 0$.

Since $f(H) = 0 \Longleftrightarrow f(tp+b) = 0$ for every $f \in K[y]$,
it follows that
$$
\trdeg_K K(tp+b) = \trdeg_K K(H) = 2
$$
If $p = 0$, then $\trdeg_K K(b) = 2$ and we take for $\gamma$ any element of 
$K(b) \setminus K$, so
$$
\trdeg_{K(\gamma)} K(\gamma)(tp) = 0 \qquad \mbox \qquad 
\trdeg_{K(\gamma)} K(\gamma)(b-\lambda) = 1
$$
for every $\lambda \in L$. Hence assume from now on that $p \ne 0$.

Since $p \ne 0$ and $\trdeg_K K(tp+b) = 2$, it follows from (iv) of lemma \ref{iilem}
that there are $\alpha, \lambda \in L$ such that one of the following is satisfied.
\begin{itemize}

\item \emph{$\alpha^{-1} p, b - \lambda p \in \tilde{K}^m$ for an extension field 
$\tilde{K}$ of $K$ for which $\trdeg_K \tilde{K} = 1$ and $\tilde{K} \subseteq L$.}

Take any $\gamma \in \tilde{K} \setminus K$. Since 
$\gamma \in \tilde{K} \setminus K \subseteq L \setminus K$ and $K$ is algebraically
closed in $L$, it follows that $\gamma$ is a transcendence basis of $\tilde{K}$
over $K$. Consequently, $\trdeg_{K(\gamma)}  K(\gamma)(\alpha^{-1} p) = 0$ and
$\trdeg_{K(\gamma)}  K(\gamma)(b - \lambda p) = 0$. As $p \ne 0$, we
conclude that
$$
\trdeg_{K(\gamma)} K(\gamma)(tp) = 1 \qquad \mbox \qquad 
\trdeg_{K(\gamma)} K(\gamma)(b-\lambda p) = 0
$$

\item \emph{There are nonzero $p^{(1)} \in K^m$, $b^{(1)} \in L^2 \setminus K^2$, 
such that $\alpha^{-1} p - b^{(1)}_1 p^{(1)} \in K^m$ and 
$b - \lambda p - b^{(1)}_2 p^{(1)} \in K^m$.}

We prove that this case cannot occur, by showing that $p^{(1)}$ is a 
projective image apex of $H$ over $K$. Let $q^{(1)} := \alpha^{-1} p - b^{(1)}_1 p^{(1)}$ and $a^{(1)} := b - \lambda p - b^{(1)}_2 p^{(1)}$. Then
\begin{align*}
2 = \trdeg_K K(t,u) &\ge \trdeg_K K\big(t q^{(1)} + a^{(1)} + u p^{(1)}\big) \\
&\ge \trdeg_K K\big(t\alpha q^{(1)} + a^{(1)} + (u+t\alpha b^{(1)}_1) p^{(1)}\big) \\
&= \trdeg_K K\big(tp + a^{(1)} +  u p^{(1)}\big) \\
&\ge \trdeg_K K\big((t+\lambda)p + a^{(1)} + (u+b^{(1)}_2) p^{(1)}\big) \\
&= \trdeg_K K\big(tp+b+up^{(1)}\big) 
\end{align*}
%Take
%$$
%\psi(t) := \alpha \cdot (t+\lambda) \cdot b^{(1)}_1 + b^{(1)}_2
%$$ 
%Then $\psi(\alpha^{-1}t-\lambda) = t b^{(1)}_1 + b^{(1)}_2$ and
%\begin{align*}
%\trdeg_K &K\big(tp+b+up^{(1)}\big) \\ 
%&= \trdeg_K K\big(tp+b+(u-\psi(t))p^{(1)}\big) \\
%&\le \trdeg_K K\big(tp+b-\psi(t) p^{(1)},u p^{(1)}\big) \\
%&= \trdeg_K K\big((\alpha^{-1}t-\lambda)p+b-\psi(\alpha^{-1}t-\lambda) p^{(1)},u p^{(1)}\big) \\
%&= \trdeg_K K\big(t \alpha^{-1}p -\lambda p+b-tb^{(1)}_1 p^{(1)} - b^{(1)}_2 p^{(1)},u p^{(1)}\big) \\
%&= \trdeg_K K\big(t (\alpha^{-1}p-b^{(1)}_1 p^{(1)}) + (b-\lambda p-b^{(1)}_2 p^{(1)}),u p^{(1)}\big) \\
%&\le \trdeg_K K\big(t (\alpha^{-1}p-b^{(1)}_1 p^{(1)})\big) + {} \\
%&\hphantom{\,\le{}} \trdeg_K K\big((b-\lambda p)-b^{(1)}_2 p^{(1)}\big) + \trdeg_K K(u p^{(1)}) \\
%&= 1 + 0 + 1 = 2 \\ &= \trdeg_K K(tp+b)
%\end{align*}
Since $tp+b$ is obtained from $tp+b+up^{(1)}$ by substituting $u = 0$ in addition, it
follows that $p^{(1)}$ is a projective image apex over $K$ of $tp+b$. From
$f(H) = 0 \Longleftrightarrow f(tp+b) = 0$ for every $f \in K[y]$ and proposition 
\ref{relprop}, we deduce that $p^{(1)}$ is a projective image apex of $H$ over $K$.
\qedhere

\end{itemize}
\end{proof}

If $R$ is a polynomial ring and $\lambda = 0$ in (iv) of theorem \ref{gradrel}, then 
one can show that there exists a $\gamma \in R$ such that $b \in K[\gamma]^m$ and
$p \in R \cdot K[\gamma]^m$. 

But $\lambda$ in (iv) of theorem \ref{gradrel} does not need to be zero. 
Take for instance $R = K[x_4,x_5]$, $g = t^2$, $\gamma = (x_4x_5 + 1)/x_5$, 
$$
p := \left( \begin{array}{c} -x_5^3 \\ (x_4 x_5 + 1)^3 \\ (x_4 x_5 + 1)^2 x_5 \end{array} \right)
\qquad \mbox{and} \qquad b := \left( \begin{array}{c} 
(-x_4x_5 + 1) x_5 \\ x_4^2 (x_4 x_5 + 1)^2 \\ x_4^2 (x_4 x_5 + 1) x_5 \end{array} \right)
$$
Then $\trdeg_K K(b) = 2$, so $\lambda \ne 0$. One can verify that
$$
p = x_5^3 \left( \begin{array}{c} -1 \\ \gamma^3 \\ \gamma^2 \end{array} \right)
\qquad \mbox{and} \qquad b - \frac{x_4^2}{x_4 x_5 + 1} p = 
\left( \begin{array}{c} \gamma^{-1} \\ 0 \\ 0 \end{array} \right)
$$
so $\lambda = x_4^2/(x_4 x_5 + 1)$ is possible.

\begin{proof}[Proof$\!$\nopunct] 
\emph{of the case where $R$ is a B{\'e}zout domain of {\upshape (v)} of theorem \ref{gradrel}.}
Notice that
$$
0 \le \trdeg_L L(H) \le \trdeg_K K(H) = 3
$$
We first show that $\trdeg_L L(H) \le 2$. So assume that $\trdeg_L L(H) = 3$.
From (ii) of theorem \ref{xsc}, it follows that $a$ is an image apex of $H$ over 
$L$ as well. On account of (ii) of lemma \ref{slem}, $H$ has a projective image apex $p \in L^m$ over $L$, in particular over $K$. From (i) of theorem \ref{xsc}, it follows that $\trdeg_L L(H) \le \trdeg_K K(H) - 1 = 2$.

Suppose that $R$ is a $\gcd$-domain. We next show that $H$ has an image apex $b \in R^m$ over $L$. If $\trdeg_L L(H) \le 1$, then $H$ has an image apex $b \in R^m$ over $L$ on account of (ii) of theorem \ref{rtheorem}, so assume that $\trdeg_L L(H) = 2$. Then it follows from (i) of lemma \ref{slem} that $H$ has a projective image apex $p \in L^m$ over $L$, in particular over $K$. On account of (iii) of theorem \ref{xsc}, $a$ is an image apex of $H$ over $L$. So $\trdeg_L L(H) \le 2$ and $H$ has an image apex $b \in R^m$ over $L$.

Suppose that $R$ is a B{\'e}zout domain. On account of (iii) of theorem \ref{rtheorem}, $h$ is of the form
\begin{align*}
h = g(p_1 x_1 + p_2 x_2 + \cdots + p_m x_m, q_1 x_1 + q_2 x_2 + \cdots + q_m x_m) + {} \\
b_1 x_1 + b_2 x_2 + \cdots + b_m x_m 
\end{align*}
where $g \in R[t,u]$. Furthermore, $p, q \in R^m$, such that 
$$
f(H) = 0 \Longleftrightarrow f\big(tp + uq + (1-v)H +vb\big) = 0
\Longleftrightarrow f(tp + uq + b) = 0
$$
for every $f \in L[y]$, where $v$ is a variable. 

It follows that every
$L$-linear combination of $p$ and $q$, which is not the zero vector, 
is a projective image apex of $H$ over $L$. Furthermore, we can take $q = 0$
if $p$ and $q$ are dependent over $L$. From (iv) of theorem \ref{piacalc},
we deduce that $q = 0$ if some nontrivial $L$-linear combination of $p$ and $q$ 
is contained in $K^m$.

If $a$ is an image apex of $H$ over $L$, then we can take $b = a$, so 
$$
f(H) = 0 \Longleftrightarrow f(tp+uq+a) = 0
$$
for every $f \in K[y]$. Suppose next that $a$ is not an image apex of $H$ over $L$. 
From (iii) of theorem \ref{xsc}, it follows that $\trdeg_L L(H) \le \trdeg_K K(H) - 2 = 1$.
On account of (ii) of theorem \ref{rtheorem}, $h$ is of the form
$$
h = g(p_1 x_1 + p_2 x_2 + \cdots + p_m x_m) + b_1 x_1 + b_2 x_2 + \cdots + b_m x_m
$$
where $g \in R[t]$ and $b \in R^m$. Furthermore, $f(H) = 0 \Longleftrightarrow f(tp + b) = 0$
for every $f \in L[y]$. From proposition \ref{relprop}, it follows that
$$
f(H) = 0 \Longleftrightarrow f\big((1-u)(tp + b) + ua\big) = 0
$$
for all $f \in K[y]$. If we substitute $t = t/(1-u)$, then we obtain
$$
f(H) = 0 \Longleftrightarrow f\big(tp + (1-u)b + ua\big) = 0
$$
for all $f \in K[y]$. So if we define $q = b - a$, then 
\begin{align*}
h = g(p_1 x_1 + p_2 x_2 + \cdots + p_m x_m) + q_1 x_1 + q_2 x_2 + \cdots + q_m x_m + {} \\
a_1 x_1 + a_2 x_2 + \cdots + a_m x_m
\end{align*}
and, since $(1-u)b + ua = (1-u)q + a$, 
$$
f(H) = 0 \Longleftrightarrow f\big(tp + (1-u)q + a\big) = 0 
\Longleftrightarrow f(tp + uq + a) = 0
$$
for every $f \in K[y]$. Consequently,
$$
f(H) = 0 \Longleftrightarrow f(tp+uq+a) = 0
$$
for every $f \in K[y]$, regardless of whether $a$ is an image apex of $H$ over $L$
or not.
\end{proof}

\begin{proof}[Proof$\!$\nopunct] 
\emph{of the case where $R = L$ of {\upshape (v)} of theorem \ref{gradrel}.}
As $L$ is a B{\'e}zout domain, it suffices to prove that 
$\trdeg_{K(\gamma)} K(\gamma)(tp) + \trdeg_{K(\gamma)} K(\gamma)(uq) = 2$
for some $\gamma \in L$.

Since $f(H) = 0 \Longleftrightarrow f(tp+uq+a) = 0$ for every $f \in K[y]$
if $R$ is a B{\'e}zout domain, it follows that
$$
\trdeg_K K(tp+uq) = \trdeg_K K(H) = 3
$$
If $q = 0$, then $\trdeg_K K(tp) = 3$ and we take for $\gamma$ any element of 
$K(p) \setminus K$, so
$$
\trdeg_{K(\gamma)} K(\gamma)(tp) = 2 \qquad \mbox \qquad 
\trdeg_{K(\gamma)} K(\gamma)(uq) = 0
$$
Hence assume from now on that $q \ne 0$.

Since $p$ and $q$ are independent over $L$ and $\trdeg_K \allowbreak K(tp+uq) = 3$, 
it follows from (iv) of lemma \ref{iilem} that we may assume that 
one of the following is satisfied.
\begin{itemize}

\item \emph{$p, q \in \tilde{K}^m$ for an extension field 
$\tilde{K}$ of $K$ for which $\trdeg_K \tilde{K} = 1$ and $\tilde{K} \subseteq L$.}

Take any $\gamma \in \tilde{K} \setminus K$. Since 
$\gamma \in \tilde{K} \setminus K \subseteq L \setminus K$ and $K$ is algebraically
closed in $L$, it follows that $\gamma$ is a transcendence basis of $\tilde{K}$
over $K$. Consequently, $\trdeg_{K(\gamma)}  K(\gamma)(p) = 0$ and
$\trdeg_{K(\gamma)}  K(\gamma)(q) = 0$. As $p,q \ne 0$, we
conclude that
$$
\trdeg_{K(\gamma)} K(\gamma)(tp) = 1 \qquad \mbox \qquad 
\trdeg_{K(\gamma)} K(\gamma)(uq) = 1
$$

\item \emph{There are $p^{(1)} \in K^m$, $b^{(1)} \in L^2 \setminus K^2$, 
such that $p - b^{(1)}_1 p^{(1)} \in K^m$ and 
$q - b^{(1)}_2 p^{(1)} \in K^m$.}

We prove that this case cannot occur, by showing that $p^{(1)}$ is a 
projective image apex of $H$ over $K$. Let $v$ be a variable, and define $q^{(1)} = p - b^{(1)}_1 p^{(1)}$ and $q^{(2)} = q - b^{(1)}_2 p^{(1)}$. Then
\begin{align*}
3 = \trdeg_K K(t,u,v) &\ge \trdeg_K K\big(t q^{(1)} + u q^{(2)} + a + v p^{(1)}\big) \\
&\ge \trdeg_K K\big(t q^{(1)} + u q^{(2)} + a + (v + t b^{(1)}_1 + u b^{(1)}_2) p^{(1)}\big) \\
&= \trdeg_K K\big(tp+uq+a\big)
\end{align*}
Since $tp+uq+a$ is obtained from $tp+uq+a+vp^{(1)}$ by substituting $v = 0$ in addition, it
follows that $p^{(1)}$ is a projective image apex over $K$ of $tp+uq+a$. From
$f(H) = 0 \Longleftrightarrow f(tp+uq+a) = 0$ for every $f \in K[y]$ and proposition 
\ref{relprop}, we deduce that $p^{(1)}$ is a projective image apex of $H$ over $K$.
\qedhere

\end{itemize}
\end{proof}

\begin{proof}[Proof$\!$\nopunct] 
\emph{of the case where $L$ satisfies L{\"u}roth's theorem as
an extension field of $K$ of {\upshape (v)} of theorem \ref{gradrel}.}
We reduce to the case $R=L$. So assume that $g \in L[t,u]$ and $p,q \in L^m$ 
satisfy (v) of theorem \ref{gradrel} for $R = L$. If $q = 0$, then it follows
from (ii) of lemma \ref{gcddomain} that we can choose $g \in R[t]$ and $p \in R^m$.
So assume from now on that $q \ne 0$.

Since $p \ne 0$ as well, and $\trdeg_{K(\gamma)} K(\gamma)(tp) + 
\trdeg_{K(\gamma)} K(\gamma)(uq) = 2$, we deduce that 
$$
\trdeg_{K(\gamma)} K(\gamma)(tp) = \trdeg_{K(\gamma)} K(\gamma)(uq) = 1
$$
Now replace $p$ by $\lambda p$ for some $\lambda \in L$, to obtain
$p_i \in K(\gamma) \setminus \{0\}$ for some $i$. 
Then $t$ is algebraic over $K(\gamma)(p + tp)$, so $p$ is not a projective image
apex over $K(\gamma)$ of $p$ on account of (ii) of proposition \ref{algdepprop}. 
From (ii) of proposition \ref{trdegprop}, we infer that
$$
\trdeg_{K(\gamma)} K(\gamma)(p) \ne \trdeg_{K(\gamma)} K(\gamma)(p+tp) = 1
$$
so $\trdeg_{K(\gamma)} K(\gamma)(p) = 0$. Similarly, we may assume that 
$\trdeg_{K(\gamma)} K(\gamma)(q) = 0$.

Consequently, $\trdeg_K K(p,q) = \trdeg_K K(\gamma) = 1$. Since
$L$ satisfies L{\"u}roth's theorem as an extension field of $K$, we can choose
$\gamma \in L$ such that $p,q \in K(\gamma)^m$. Now multiply $p$ by
an element of $K[\gamma]$, to obtain that $p \in K[\gamma]^m$. Similarly, we may 
assume that $q \in K[\gamma]^m$. 

Notice that $\gamma \notin K$. Since $K$ is algebraically closed in $L$,
we can view $\gamma$ as a variable over $K$.
From lemma \ref{weaksmith}, it follows that we may assume that $(p\,|\,q)$ is left 
invertible over $K[\gamma]$. From corollary \ref{weaksmithcor}, it follows
that we may assume that the coefficient vector $\hat{p} \in K^m$ of 
$\gamma^{\deg_{\gamma} p}$ of $p$ and the coefficient vector $\hat{q} \in K^m$ of 
$\gamma^{\deg_{\gamma} q}$ of $q$ are independent vectors.

Since $R$ is a $\gcd$-domain, we can write $\gamma = c_1/c_2$, where $c_1, c_2 \in R$ 
such that $\gcd\{c_1,c_2\} = 1$. Define
$$
\tilde{p} := c_2^{\deg_{\gamma} p}p \qquad \mbox{and} \qquad
\tilde{q} := c_2^{\deg_{\gamma} q}q \qquad \mbox{and}
$$
Then $\tilde{p}, \tilde{q} \in R^m$ and 
\begin{align*}
h = \tilde{g}(\tilde{p}_1 x_1 + \tilde{p}_2 x_2 + \cdots + \tilde{p}_m x_m, 
    \tilde{q}_1 x_1 + \tilde{q}_2 x_2 + \cdots + \tilde{q}_m x_m) + {} \\
    a_1 x_1 + a_2 x_2 + \cdots + a_m x_m 
\end{align*}
for some $\tilde{g} \in L[t,u]$. We shall show that $\tilde{g} \in R[t,u]$.

Let $C \in \Mat_{2,m}(K[\gamma]) \subseteq \Mat_{2,m}(R[c_2^{-1}])$ 
be a left inverse of $(p\,|\,q)$. Then
$$
C \cdot (\tilde{p}\,|\,\tilde{q}) = \left(\begin{array}{cc}
c_2^{\deg_{\gamma} p} & 0 \\ 0 & c_2^{\deg_{\gamma} q} 
\end{array} \right) = (\tilde{p}\,|\,\tilde{q})\tp \cdot C\tp
$$
so
$$
h\big(C\tp\tbinom{t}{u}\big) = 
\tilde{g}\Big(c_2^{\deg_{\gamma} p}t, c_2^{\deg_{\gamma} q}u\Big) + a\tp C\tp\tbinom{t}{u}
$$
Since $h\big(C\tp\tbinom{t}{u}\big) - a\tp C\tp\tbinom{t}{u} \in R[c_2^{-1}][t,u]$, we infer that $\tilde{g} \in R[c_2^{-1}][t,u]$.

Since $\hat{p} \in K^m$ and $\hat{q} \in K^m$ are independent, there exists a 
$\hat{C} \in \Mat_{2,m}(K)$ which is a left inverse of $(\hat{p}\,|\,\hat{q})$.
Consequently,
$$
\hat{C} \cdot (\tilde{p}\,|\,\tilde{q}) \equiv \left(\begin{array}{cc} 
c_1^{\deg_{\gamma} p} & 0 \\ 0 & c_1^{\deg_{\gamma} q} 
\end{array} \right) \equiv (\tilde{p}\,|\,\tilde{q})\tp \cdot \hat{C}\tp \pmod{c_2} 
$$
Now take $j$ minimal, such that $c_2^j \tilde{g}[t,u] \in R[t,u]$. Then
$$
c_2^j h\big(\hat{C}\tp\tbinom{t}{u}\big) \equiv 
c_2^j \tilde{g}\Big(c_1^{\deg_{\gamma} p}t, c_1^{\deg_{\gamma} q}u\Big) + c_2^j a\tp \hat{C}\tp\tbinom{t}{u}
\pmod{c_2} 
$$
Suppose that $j \ge 1$. Then $c_2 \mid c_1^i c_2^j \tilde{g}[t,u]$ for some $i \ge 0$. As $\gcd\{c_1,c_2\} = 1$, $c_2 \mid c_2^j \tilde{g}[t,u]$. This contradicts the minimality of $j$, so $j = 0$ and $\tilde{g} \in R[t,u]$.
\end{proof}

\begin{proof}[Proof of the third claim from the end of theorem \ref{gradrel}.]
Say that $R$ is the polynomial ring $K[x_{m+1},x_{m+2},\ldots,x_n]$.
Suppose that $h$ is of the form of (ii) of theorem \ref{gradrel} and $\deg b < 2$.
Then $\deg b = 1$, because $b \notin K^m$. So there exists an $i$ such that the 
vector $p^{(1)} \in K^m$ of coefficients of $x_i$ in $H = b$ is nonzero. By substituting 
$x_i = x_i + t$ in $H$, we see that $p^{(1)}$ is a projective image apex of $H$ over $K$.
Contradiction, so $\deg b \ge 2$ if $h$ is of the form of (ii) of theorem \ref{gradrel}.

Suppose that $h$ is of the form of (iii) of theorem \ref{gradrel} and $\deg p < 2$.
Then $\deg p = 1$, because $p \notin K^m$. So there exists an $i$ such that the 
vector $p^{(1)} \in K^m$ of coefficients of $x_i$ in $p$ is nonzero. 
By substituting $x_i = x_i + u/t$ in $tp + a$, we see that 
$p^{(1)}$ is a projective image apex of $tp + a$. By way of proposition 
\ref{relprop}, it follows from the last claim in (iii) of theorem \ref{gradrel} 
that $p^{(1)}$ is a projective image apex of $H$ over $K$. Contradiction,
so $\deg p \ge 2$ if $h$ is of the form of (iii) of theorem \ref{gradrel}.
\end{proof}

\begin{proof}[Proof of the last but one claim of theorem \ref{gradrel}.]
Suppose that $H$ is graded homogeneous of degree $\delta$. Suppose that $f \in K[y]$ such that $f(H) = 0$. Let $f^{(j)}$ the part of degree $j$ for each $j$. Then $f^{(j)}(H)$ is graded homogeneous of degree $j\delta$ or zero for each $j$. Since the grading group is torsion-free, we infer that $f^{(j)}(H) = 0$ for each $j$. So $f^{(j)}\big((1-t)H\big) = (1-t)^j f^{(j)}(H) = 0$ for each $j$. Hence $f\big((1-t)H\big) = 0$.
\end{proof}

\begin{proof}[Proof of the last claim of theorem \ref{gradrel}.]
To prove the last claim, suppose that $h$ is graded homogeneous and $a=0$ is an image apex of $H$. We first show that $h \in L[\sum_{i=1}^m p_i x_i]$ and $h \in L[\sum_{i=1}^m p_i x_i, \allowbreak\sum_{i=1}^m q_i x_i]$ respectively, such that the sums are graded homogeneous. To find $p$, we differentiate $h$ with respect to a variable $x_j$ such that it does not get constant, until it gets a nontrivial linear part. We take $p$ such that $\sum_{i=1}^m p_i x_i$ is that linear part.

To find $q$, we first take any $q$ (independent of $p$), such that we can write $h = g(\sum_{i=1}^m p_i x_i,\allowbreak\sum_{i=1}^m q_i x_i)$, where $g \in L[t,u]$. 
If $\deg_u g = 0$, then we do not need $q$, so assume that $\deg_u g > 0$.
We first adapt $h$ by replacing $g$ by a homogeneous part of it which in not killed by $\parder{}{u}g$, say of degree $d$. This is possible, because $h$ is just replaced by its homogeneous part of degee $d$. So $\deg_u g > 0$ is preserved. While $\deg_u g < d = \deg g$, we differentiate $h$ with respect to a variable $x_j$ for which $p_j \ne 0$, and replace $h$ by the result. This way, we obtain $\deg_u g = d$ without affecting $\deg_u g$.

While $d > 1$, we repeat the following. If we replace $h$ by $\parder{}{x_1} h$, then the new $g$ becomes either zero or homogeneous of degree $d-1$, with $d$ still as before, but the coefficient of $u^{d-1}$ of the new $g$ might become zero. In this case, the ratio of $p_1$ and $q_1$ corresponds to a number which is determined by the coefficients of $t^0 u^d$ and $t^1 u^{d-1}$ of the old $g$. But since $p$ and $q$ are independent, there exists a $j$ for which $p_j$ and $q_j$ do not have the critical ratio, so that replacing $h$ by $\parder{}{x_j} h$ decreases $d = \deg g$ by $1$ without affecting $\deg_u g = d$.

So we obtain $\deg_u g = \deg g = 1$. We replace $q$ by the coefficient vector of the linear part of $h$ at this stage. Notice that the first step of adapting $h$ by its homogeneous part of degree $d$ was for notational convenience only: we can apply the same differentiations on $h$ itself to obtain the new $q$ from the linear part of $h$ at the end.

Having $p$ of the desired form, we can run the above proof of (iii) without affecting this. The same holds for the case $q = 0$ of (v). So assume that we cannot take $q = 0$. Then $p$ is not an $L$-multiple of a vector over $K$. Hence we may assume that $p_2 \ne 0$ and $p_1/p_2 \notin K$. Then $p_1, p_2 \in K[\gamma]$ cannot be obtained in the same way as before, because the graded homogeneity of $p_1$ and $p_2$ may be affected. But $p_1/p_2 \in K(\gamma)$ holds. 

Notice that $p_1$ and $p_2$ do not need or homogeneous of the same graded degree, because the graded degrees of $x_1$ and $x_2$ may be different. Furthermore, $p_1$ and $p_2$ do not need to be relatively prime. From lemma \ref{gammahmg}, it follows that we can take for $\gamma$ a reduced fraction $c_1/c_2$, such that both $c_1$ and $c_2$ are graded homogeneous. If $c_1$ and $c_2$ have the same graded homogeneous degree, then so have $x_1$ and $x_2$, and the above proof of (v) can be followed. So assume that $c_1$ and $c_2$ do not have the same graded homogeneous degree.

From lemma \ref{lefgr}, it follows that the grading subgroup of $K[c_1,c_2,x_1,\ldots,x_m]$ can be embedded in $\R$. Without loss of generality, we assume that $c_1$ has larger graded homogeneous degree than $c_2$. We can obtain that each component of $p$ and $q$ is either zero, or a $K$-multiple of a nonnegative power of $\gamma$. Since it is not clear if the techniques at the end of section \ref{sec:rtheorem} preserve this, we will adapt $p$ and $q$ in another way to obtain that $C$ and $\hat{C}$ as in the proof of (v) exists.

Assume without loss of generality that $x_1$ is a variable of maximum graded homogeneous degree, for which the coefficient in either $p$ or $q$ is nonzero. Then we may assume that the coefficients in $p$ and $q$ of $x_1$ are $1$ and $0$ respectively. Assume without loss of generality that $x_2$ is a variable of maximum graded homogeneous degree, for which the coefficient in $q$ is nonzero. Then we may assume that the coefficient in $q$ of $x_2$ is $1$. From this, it follows that the left-inverse $C$ of $(p\,|\,q)$ exists. 

Assume without loss of generality that $x_3$ is a variable of minimum graded homogeneous degree, for which the coefficient in either $p$ or $q$ is nonzero. Then we may assume that the coefficient in either $p$ or $q$ of $x_3$ is zero. From this, it follows that the left-inverse $\hat{C}$ of $(\hat{p}\,|\,\hat{q})$ exists. So the above proof of (v) can be followed in this case as well.
\end{proof}

\begin{lemma} \label{gammahmg}
Let $R$ be a $\gcd$-domain over $K$ with fraction field $L$. Suppose that $R$ is graded with a torsion-free group, such that $K$ has trivial grading. Suppose that $p_1, p_2 \in R$ are graded homogeneous, and $p_1/p_2 \notin K$. Suppose that $\gamma \in L(x)$ is nonzero, such that $p_1/p_2 \in K(\gamma)$. Write $\gamma = c_1/c_2$ as a reduced fraction over $R$.

If $p_1$ and $p_2$ are graded homogeneous of the same degree, then $c_1$ and $c_2$ are graded homogeneous of the same degree as well. If $p_1$ and $p_2$ are graded homogeneous of different degrees and $K$ is algebraically closed in $L$, then $c_1$ and $c_2$ can be taken graded homogeneous without affecting $K(c_1/c_2)$, after which $p_1/p_2$ is a $K$-multiple of an integral power of $\gamma=c_1/c_2$.
\end{lemma}

\begin{proof}
There are homogeneous polynomials $g_1,g_2 \in K[y_1,y_2]$ of the same degree, say $i$, such that $g_1/g_2$ is a reduced fraction and $p_1/p_2 = g_1(c_1,c_2)/g_2(c_1,c_2)$. Since $R$ is a $\gcd$-domain, we infer that $(c_1,c_2)$ is superprimitive as on the second page of \cite{MR3806744}. From \cite[Theorem 2.5]{MR3806744}, it follows that $g_1(c_1,c_2)/g_2(c_1,c_2)$ is a reduced fraction. Since $R$ is a $\gcd$-domain, we infer that $g_1(c_1,c_2)/g_2(c_1,c_2)$ is the reduced fraction of $p_1/p_2$. So there exists a $c_3 \in R$ such that $p_1 = c_3 g_1(c_1,c_2)$ and $p_2 = c_3 g_2(c_1,c_2)$.

Since we will only use the grading to prove that some elements of $K[c_1,c_2,c_3]$ are graded homogeneous, it follows from lemma \ref{lefgr} below that we may assume that the grading group is $\R$. It follows that $c_3$,  $g_1(c_1,c_2)$ and $g_2(c_1,c_2)$ are graded homogeneous. If $c_1$ and $c_2$ are graded homogeneous of degree $\delta$, then 
$g_1(c_1,c_2)$ and $g_2(c_1,c_2)$ are graded homogeneous of degree $i\delta$, and $p_1$ and $p_2$ are graded homogeneous of the same degree.

So assume that $c_1$ and $c_2$ are not graded homogeneous of the same degree, if at all.   We will show below that $g_1$ and $g_2$ have linear factors $\lambda_1 y_1 + \lambda_2 y_2$ and $\mu_1 y_1 + \mu_2 y_2$ respectively over $K$. So $\lambda_1 c_1 + \lambda_2 c_2 \mid g_1(c_1,c_2)$ and $\mu_1 c_1 + \mu_2 c_2 \mid g_2(c_1,c_2)$ are graded homogeneous. 
Since $g_1/g_2$ is a reduced fraction, we deduce that $\lambda_1 \mu_2 \ne \lambda_2 \mu_1$, so we can replace $(c_1,c_2)$ by $(\lambda_1 c_1 + \lambda_2 c_2, \mu_1 c_1 + \mu_2 c_2)$ without affecting $K(c_1/c_2)$. The result is that $c_1$ and $c_2$ become graded homogeneous of distinct degrees. Now it is straightforward to verify that $p_1/p_2$ is a $K$-multiple of $(c_1/c_2)^i$.

So it remains to show that both $g_1$ and $g_2$ have a linear factor over $K$. We only show that $g_1$ has a linear factor over $K$, since things are similar for $g_2$. Since $c_1$ and $c_2$ are not graded homogeneous of the same degree, either the highest or lowest degree part of $(c_1,c_2)$ is a root of $g_1$. Suppose that $(\tilde{c}_1,\tilde{c}_2) \ne (0,0)$ is a nontrivial homogeneous part of $(c_1,c_2)$ which is a root of $g_1$.
If $\tilde{c}_1 = 0$ or $\tilde{c}_2 = 0$, then $y_1 \mid g_1$ or $y_2 \mid g_1$ respectively. 
So assume that $\tilde{c}_1 \tilde{c}_2 \ne 0$. Then $\tilde{c}_2^i g_1(\tilde{c}_1/\tilde{c}_2,1) = g_1(\tilde{c}_1,\tilde{c}_2) = 0$, so $\lambda_2 := \tilde{c}_1/\tilde{c}_2 \in L$ is algebraic over $K$. Since $K$ is algebraically closed in $L$, we infer that $\lambda_2 \in K$, and that $g_1(y_1,1)$ is divisible by the minimum polynomial $y_1 - \lambda_2$ of $\lambda_2$. From \cite[Lemma 2.2]{MR3806744}, it follows that $y_1 - \lambda_2 y_2 \mid g_1$. So we can take $\lambda_1 = -1$.
\end{proof}

The condition that $K$ is algebraically closed in $L$ can be weakened to that the algebraic closure of $K$ in $L$ is separable over $K$. The proof of this can be done by extending the grading of $R$ to $R[1/\tilde{c}_2]$ and showing that the minimum polynomial of $\lambda_2 = \tilde{c_1}/\tilde{c}_2$ over $K$ is a power of $y_1-\lambda_2$. But the condition cannot be omitted. Take e.g.\@ $K = \F_2[t^2]$, $p_1/p_2 = u^2$, and $\gamma = u + t$.

\begin{lemma}[Lefschetz's principle for torsion-free group gradings] \label{lefgr}
Suppose that $R$ is a commutative ring which is graded with a torsion-free group, such that $K$ has trivial grading. Let $G$ be the grading subgroup of $R$, i.e.\@ all group elements which are generated by the degrees of graded homogeneous components of elements of $R$. 

Let $\tilde{R}$ be the subring of $R$ which is generated by the graded homogeneous components of elements $c_1, c_2, c_3, \ldots \in R$, and $\tilde{G}$ be the grading subgroup of $\tilde{R}$. Then there exists a homomorphism from $G$ to $\R$ which
is injective on $\tilde{G}$.
\end{lemma}

\begin{proof}
Since $R$ is commutative, we deduce that $G$ is a torsion-free abelian group. Hence $G$ is a torsion-free $\Z$-module. Since $G$ is torsion-free, we infer that $G$ embeds in its localization with respect to $\Z \setminus \{0\}$. This localization is a vector space $V$ over $\Q$. Assume without loss of generality that $G \subseteq V$. Notice that $\tilde{G}$ is at most countably generated, say by ${\mathcal C}$. On account of Zermelo's well-ordering theorem, there exists a well-ordering on both ${\mathcal C}$ and $V \setminus {\mathcal C}$, which we combine to a well-ordering of $V$ by assuming that elements of ${\mathcal C}$ are smaller than those of $V \setminus {\mathcal C}$. 

Let ${\mathcal B}$ be the subset of nonzero elements of $V$ which cannot be written as a linear combination of smaller elements of $V$. Then the elements of ${\mathcal B}$ are linearly independent, because the largest element of ${\mathcal B}$ in a possible dependence relation does not exist. We show that every nonzero element $v \in V \setminus {\mathcal B}$ is a linear combination of elements of ${\mathcal B}$ which are smaller than $v$. Suppose that this is not the case, and let $v$ be the smallest counterexample. Since $v \notin {\mathcal B}$, we can write $v$ as a linear combination of smaller elements $w$ of $V$. Each of these $w$ is either contained in ${\mathcal B}$, or a linear combination of elements of ${\mathcal B}$ which are smaller than $w$ and hence also smaller than $v$. Contradiction.

Let ${\mathcal D} = {\mathcal C} \cap {\mathcal B}$. Then ${\mathcal D}$ is a basis of the linear span of ${\mathcal C}$, which extends to a basis ${\mathcal B}$ of $V$.
Since ${\mathcal D} \subseteq {\mathcal C}$, we see that ${\mathcal D}$ is at most countable, say ${\mathcal D} = \{d_1,d_2,d_3,\ldots\}$. Now define a group homomorphism to $\R$ by sending $d_j$ to $\eu^j$ for each $j \in \{1,2,3,\ldots\}$, and sending elements of ${\mathcal B} \setminus {\mathcal D}$ to $0$, where $\eu$ is Eulers number. Since $\eu$ is transcendental over $\Q$, we deduce that $\tilde{G}$ is embedded in $\R$.
\end{proof}

\section{Theorem \ref{HEredc} for gradient maps} \label{iiiniv}

\begin{proof}[Proof of {\upshape (iii)} of theorem \ref{HEredc}.]
The case where $\rk \jac H \le 2$ follows from (i), so assume that 
$\rk \jac H = 3$. 

Just like for $M$ in the proof of proposition \ref{HEredb} and $N$ in 
the proof of (i) and (ii) of theorem \ref{HEredc}, theorem \ref{cominherred} 
allows us to assume that $N = \rk \jac H + 2 = 5$.
Take $k$ and $k'$ as in the proofs of proposition \ref{HEredb} and (i) and (ii) of 
theorem \ref{HEredc}.
The case $k' = s+1$ follows in a similar manner as in the proofs of 
proposition \ref{HEredb} and (i) and (ii) of theorem \ref{HEredc}.

So assume that $k' \ge s+2$. The case $k \le k'$ follows in a similar manner
as in the proofs of (i) of proposition \ref{HEredb} and (i) of theorem \ref{HEredc}.
So assume that $k \ge k' + 1 \ge s + 3$. From the minimality of $k$, it follows
that $k - s \le N - k'$, so $k + k' \le N + s = s + 5$. If we combine this with
$k \ge s+3$ and $k' \ge s+2$, then we see that $s = 0$, $k = 3$ and $k' = 2$.

Let $R = K[x_4,x_5]$ and $\tilde{R}= K[x_1,x_2]$.
From (iii) and (iv) of theorem \ref{gradrel}, it follows that $H = \grad \tilde{h}$
for some $\tilde{h} \in K[X]$, such that
$$
\tilde{h} = g(p_1x_1+p_2x_2+p_3x_3) + 
            b_1x_1+b_2x_2+b_3x_3
$$
where $g \in R[t]$ and $p_i,b_i\in R$ for all $i$, and
$$
\tilde{h} = \tilde{g}(\tilde{p}_3x_3+\tilde{p}_4x_4+\tilde{p}_5x_5) + 
            \tilde{b}_3x_3+\tilde{b}_4x_4+\tilde{b}_5x_5
$$
where $\tilde{g} \in \tilde{R}[t]$ and $\tilde{p}_i,\tilde{b}_i\in \tilde{R}$ for all $i$.

We first show that we may assume that $p_2$ and $p_3$ do not have a constant term. 
Let $v := \big(p_1(0,0),p_2(0,0),p_3(0,0)\big)$ and assume that  
$v \ne 0$. Then there exists a $T \in \GL_3(K)$ such that $v\tp T = (1~0~0)$.
Now replace $H$ by $\tilde{T}\tp H(\tilde{T}X)$ and $\tilde{h}$ by $\tilde{h}(\tilde{T}X)$, 
where
$$
\tilde{T} = \left(\begin{array}{ccccc}
& & & 0 & 0 \\
& T & & 0 & 0 \\
& & & 0 & 0 \\
0 & 0 & 0 & 1 & 0 \\
0 & 0 & 0 & 0 & 1
\end{array} \right)
$$
Then $v$ becomes $(1,0,0)$, so $p_2$ and $p_3$ do not have a constant term any more. 
The property that $(H_3,H_4,H_5) = (H_{k'+1},H_{k'+2},\allowbreak H_{k'+3})$ 
does not have a projective image apex may be affected, but restoring that as in the proofs 
of proposition \ref{HEredb} and (i) and (ii) of theorem \ref{HEredc} will not affect 
$v = (1,0,0)$. 

So $p_3 \notin K^{*}$. We distinguish two cases.
\begin{itemize}

\item $\deg_{x_3} \tilde{h} \le 1$. \\
Then the coefficient of $x_3^1$ is contained in $R \cap \tilde{R} = K$,
so $H_3 \in K$. We will derive a contradiction by showing that $H_2, H_3, H_4$
are algebraically independent over $K$. From the proof of proposition \ref{HEredb}, 
it follows that $H_1$ and $H_5$ are algebraically dependent over $K$ on $H_2,H_3,H_4$. 
Since $\rk \jac H = 3$, we deduce that $H_2, H_3, H_4$ is a transcendence basis over 
$K$ of $K(H)$. 

\item $\deg_{x_3} \tilde{h} \ge 2$. \\
Then the leading
coefficient with respect to $x_3$ of $\tilde{h}$ is contained in $R$,
but not in $K$, because it is divisible by $p_3 \notin K$. On the other
hand, the leading coefficient with respect to $x_3$ of $\tilde{h}$ is 
contained in $\tilde{R}$. This contradicts
$(R \setminus K) \cap \tilde{R} = \varnothing$. \qedhere

\end{itemize}
\end{proof}

\begin{proof}[Proof of {\upshape (iv)} of theorem \ref{HEredc}.]
The case where $\rk \jac H \le 3$ follows from (ii), so assume that 
$\rk \jac H = 4$.

Just like for $M$ in the proof of proposition \ref{HEredb} and $N$ in 
the proof of (i) and (ii) of theorem \ref{HEredc}, theorem \ref{cominherred} 
allows us to assume that $N = \rk \jac H + 2 = 6$.
Take $k$ and $k'$ as in the proofs of proposition \ref{HEredb} and (i) and (ii) of 
theorem \ref{HEredc}.
The case $k' = s+1$ follows in a similar manner as in the proofs of 
proposition \ref{HEredb} and (i) and (ii) of theorem \ref{HEredc}.

So assume that $k' \ge s+2$. The case $k \le k'+1$ follows in a similar manner
as in the proofs of (ii) of proposition \ref{HEredb} and (ii) of theorem \ref{HEredc}.
So assume that $k \ge k' + 2 \ge s + 4$. From the minimality of $k$, it follows
that $k - s \le N - k'$, so $k + k' \le N + s = s + 6$. If we combine this with
$k \ge s+4$ and $k' \ge s+2$, then we see that $s = 0$, $k = 4$ and $k' = 2$.

Let $R = K[x_5,x_6]$, $L = K(x_5,x_6)$, $\tilde{R}= K[x_1,x_2]$ and $\tilde{L}= K(x_1,x_2)$.
Suppose that $a \in K^6$ is an image apex of $H$.
From (v) of theorem \ref{gradrel}, it follows that $H = \grad \tilde{h} + a$
for some $\tilde{h} \in K[X]$, such that
$$
\tilde{h} = g(p_1x_1+p_2x_2+p_3x_3+p_4x_4,q_1x_1+q_2x_2+q_3x_3+q_4x_4)
$$
where $g \in L[t,u]$ and $p_i,q_i\in R$ for all $i$, and
$$
\tilde{h} = \tilde{g}(\tilde{p}_3x_3+\tilde{p}_4x_4+\tilde{p}_5x_5+\tilde{p}_6x_6,
            \tilde{q}_3x_3+\tilde{q}_4x_4+\tilde{q}_5x_5+\tilde{q}_6x_6)
$$            
where $\tilde{g} \in \tilde{L}[u]$ and $\tilde{p}_i,\tilde{q}_i\in \tilde{R}$ for all $i$.

Let $(c_1,c_2,c_3,c_4) \in R^4$ be an $L$-linear combination of $(p_1,p_2,p_3,p_4)$ and
$(q_1, \allowbreak q_2,q_3,q_4)$, such that its constant part 
$v := \big(c_1(0,0),c_2(0,0),c_3(0,0),c_4(0,0)\big) \ne 0$ 
in the case where it is possible to take $(c_1,c_2,c_3,c_4)$ as such. Just like for the 
assumption that $v = (1,0,0)$ if $v \ne 0$ in the proof of (iii) of theorem \ref{HEredc}, 
we may assume in this proof that $v = (1,0,0,0)$ if $v \ne 0$. 

From the proof of proposition \ref{HEredb}, it follows that $H_1$ and $H_6$ are 
algebraically dependent over $K$ on $H_2,H_3,H_4,H_5$. Since $r = 4$, we deduce that 
$H_2, H_3, H_4, H_5$ is a transcendence basis over $K$ of $K(H)$. So
$H_3 - a_3$ and $H_4 - a_4$ are linearly independent over $K$.

Assume without loss of generality that $q_3 = 0$ and $\tilde{q}_3 = 0$.
Then $p_3\tilde{p}_3 \ne 0$, because $H_3 - a_3 \ne 0$. 
We distinguish three cases.
\begin{itemize}
 
\item $q_4 = 0$. \\
From $q_3 = q_4 = 0$, it follows that
$$
p_4 (H_3-a_3) - p_3 (H_4-a_4) = 0
$$
Since $H_3 - a_3$ and $H_4 - a_4$ are linearly independent over $K$, it follows
that $p_4 \ne 0$ and $p_4^{-1}p_3 \notin K$. 

From $q_3 = q_4 = 0$, it follows that
$$
j := \deg_t g = \deg_{x_3} \tilde{h} = \deg_{x_4} \tilde{h} 
$$
Let 
$$
\tilde{c}_3 := \parder[j]{}{x_3} \tilde{h} \qquad \mbox{and} \qquad 
\tilde{c}_4 := \parder[j]{}{x_4} \tilde{h}
$$
Then
$$
\tilde{c}_3, \tilde{c}_4 \in K[x_1,x_2,x_5,x_6] = \tilde{R}[x_5,x_6] 
\qquad \mbox{and} \qquad \tilde{c}_3 = p_4^{-j}p_3^j \tilde{c}_4
$$
Since $p_4^{-j}p_3^j \in L \setminus K = L \setminus \tilde{L}$,  
it follows that either $\tilde{c}_3 \notin \tilde{L}$ or 
$\tilde{c}_4 \notin \tilde{L}$, say that $\tilde{c}_3 \notin \tilde{L}$.

Notice that $\tilde{c}_3 \in K[x_1,x_2,x_5,x_6] \setminus \tilde{L} \subseteq \tilde{L}[x_5,x_6] \setminus \tilde{L}$. Since additionally
$$
\tilde{c}_3 \in \tilde{L}[\tilde{p}_3x_3+\tilde{p}_4x_4+\tilde{p}_5x_5+\tilde{p}_6x_6,
                          \tilde{q}_3x_3+\tilde{q}_4x_4+\tilde{q}_5x_5+\tilde{q}_6x_6] 
$$
it follows that $\tilde{q}_4 = 0$ along with $\tilde{q}_3 = 0$.
Furthermore, $\tilde{c}_3 \in \tilde{L}[\tilde{q}_5x_5+\tilde{q}_6x_6]$ and
$$
\tilde{p}_4 (H_3-a_3) - \tilde{p}_3 (H_4-a_4) = 0 = p_4 (H_3-a_3) - p_3 (H_4-a_4)
$$
So $p_4^{-1} p_3 = \tilde{p}_4^{-1} \tilde{p}_3 \in L \cap \tilde{L} = K$. Contradiction.

\item $q_4 \in K^{*}$. \\
From (iv) of lemma \ref{gcddomain}, it follows that we may assume that $p_4 = 0$ and
$g \in R[t,u]$. 

Suppose first that $p_3 \in K$. Notice that
$v \ne 0$ was possible, so $v = (1,\allowbreak 0,\allowbreak 0,0)$. 
Since $q_3 = 0 = p_4$ and $p_3q_4 \ne 0$, it follows that 
$c_i = c_3 p_3^{-1} p_i + c_4 q_4^{-1} q_i$ for all $i$. 
As $p_3 \in K^{*}$ and $q_4 \in K^{*}$, this contradicts that $c_1(0) = 1$ and $c_3(0) = c_4(0) = 0$.

Suppose next that $p_3 \notin K$. From $q_3 = 0 \ne H_3 - a_3$, it follows that
$$
j := \deg_t g = \deg_{x_3} \tilde{h} \ge 1
$$
Take $d$ as large as possible, such that $t^j u^{d-j}$ is a term of $g$, and define 
$$
\tilde{c}_3 := \parder[j]{}{x_3} \parder[d-j]{}{x_4} \tilde{h} 
$$
Then $\tilde{c}_3 \in R = K[x_5,x_6] \subseteq \tilde{L}[x_5,x_6]$.
Since $g \in R[t,u]$, it follows that $p_3 \mid p_3^j \mid \tilde{c}_3$.
So $\deg_{x_5,x_6} \tilde{c}_3 \ge 1$ and $\tilde{c}_3 \notin \tilde{L}$. Just as in the proof of the case $q_4 = 0$
above, we can deduce that $\tilde{q}_4 = 0$ along with $\tilde{q}_3 = 0$. Since the condition $v = (1,0,0,0)$ is not used in the proof of the case $q_4 = 0$ above, we can get a contradiction to $\tilde{q}_4 = 0$ in a similar manner.

\item $q_4 \notin K$. \\
Let $d := \deg g$. We will show below that there exists a 
$j \in \{1,2,\ldots,d\}$, such that either $\tilde{c}_3 \ne 0$ or 
$\tilde{c}_4 \ne 0$, where
$$
\tilde{c}_i := \parder{}{x_i}\parder[j-1]{}{x_3}\parder[d-j]{}{x_4} \tilde{h}
$$
for all $i \le 4$. Since $1 + (j-1) + (d-j) = d$, we see that 
$\tilde{c}_i \in R$ for all $i$.

We will show now that $j$ can be taken as claimed.
If the leading homogeneous part of $g$ is contained in 
$K[u]$, then we take $j = 1$, so
$$
\tilde{c}_4 = \parder{}{x_4} \parder[1-1]{}{x_3}\parder[d-1]{}{x_4} \tilde{h} 
= \parder[d]{}{x_4} \tilde{h} \ne 0
$$
If the leading homogeneous part of $g$ is not contained in $K[u]$, then 
we take for $j$ its positive degree with respect to $t$, so
$$
\tilde{c}_3 = \parder[j]{}{x_3}\parder[d-j]{}{x_4} \tilde{h} \ne 0
$$

Suppose first that either $\tilde{c}_3 \notin K$ or $\tilde{c}_4 \notin K$. Since $\tilde{c}_i \in R$ for all $i$, either $\tilde{c}_3 \in R \setminus K = K[x_5,x_6] \setminus \tilde{L}$, or $\tilde{c}_4 \in R \setminus K = K[x_5,x_6] \setminus \tilde{L}$.
So we can derive in a similar manner as in the proof of the case 
$q_4 = 0$ above that $\tilde{q}_4 = 0$ along with $\tilde{q}_3 = 0$. Since the condition $v = (1,0,0,0)$ is not used in proof of the case $q_4 = 0$ above, we can get a contradiction to $\tilde{q}_4 = 0$ in a similar manner.

Suppose next that $\tilde{c}_3 \in K$ and $\tilde{c}_4 \in K$. Since
$$ 
\parder[j-1]{}{x_3} \parder[d-j]{}{x_4} \tilde{h} = 
\tilde{c}_1 x_1 + \tilde{c}_2 x_2 + \tilde{c}_3 x_3 + \tilde{c}_4 x_4
$$
is an $L$-linear combination of $p_1x_1+p_2x_2+p_3x_3+p_4x_4$ and
$q_1x_1+q_2x_2+q_3x_3+q_4x_4$, we can replace $q_i$ by $\tilde{c}_i$ for all $i$, where it might be necessary to adapt $p$ as well. Furthermore, $q_3 = 0$ might be affected, but just as we obtained $v = (1,0,0,0)$ above, we can obtain $(q_3,q_4) = (0,1)$, which yields the case $q_4 \in K^{*}$ above. \qedhere

\end{itemize}
\end{proof}

\bibliographystyle{hessmallrank}
\bibliography{hessmallrank}

\end{document}